\documentclass[twoside, 14pt]{article}
\usepackage{amsmath,amssymb,amsthm,mathrsfs}
\usepackage[margin=2.4cm]{geometry} 
\usepackage{lipsum}
\usepackage{titlesec,hyperref}
\usepackage{fancyhdr}
\usepackage[numbers,sort&compress]{natbib}
\usepackage{color}
\usepackage[titletoc]{appendix}

\fancypagestyle{plain}
{
	\fancyhf{}
	
}

\titleformat{\subsection}{\it}{\thesubsection.\enspace}{1.5pt}{}
\titleformat{\subsubsection}{\it}{\thesubsubsection.\enspace}{1.5pt}{}

\newtheorem{theo}{Theorem}[section]
\newtheorem{lemm}[theo]{Lemma}

\newtheorem{prop}[theo]{Proposition}
\newtheorem{rema}{Remark}[section]
\numberwithin{equation}{section}

\allowdisplaybreaks

\def\beq{\begin{equation}}
	\def\bal{\begin{aligned}}
		\def\dal{\end{aligned}}
	\def\deq{\end{equation}}
\def\beqq{\begin{equation*}}
	\def\deqq{\end{equation*}}

\def\p{\partial}
\def\ka{\kappa}

\def\al{\alpha}
\def \f{\frac}

\def \ka{\kappa}
\def \i {\int_{\mathbb{R}^3}}

\def \ah{\alpha_h}
\def \bh{\beta_h}
\def\me{\mathcal{E}}
\def\md{\mathcal{D}}

\def \du{{\rm{div}} u}

\def \vro {\varrho}

\begin{document}
	\begin{sloppypar}
		\title{Nonlinear stability of a background magnetic field for the 3D compressible MHD equations with anisotropic dissipation}
		\author{$\mbox{Jincheng Gao}^1$ \footnote{Corresponding author. Email: gaojch5@mail.sysu.edu.cn}, \quad
        	$\mbox{Xianpeng Hu}^2$ \footnote{Email: xianphu@polyu.edu.hk}, \quad
			$\mbox{Lianyun Peng}^2$ \footnote{Email: lianyun.peng@polyu.edu.hk}, \quad
			$\mbox{Jiahong Wu}^3$ \footnote{Email: jwu29@nd.edu}, \\
			\quad
			$^1\mbox{School}$ of Mathematics, Sun Yat-sen University,\\
			Guangzhou 510275, China\\
			$^2\mbox{Department}$ of Applied Mathematics, 
			The Hong Kong Polytechnic University,\\
			Hong Kong 999077, China\\
			$^3\mbox{Department}$ of Mathematics,
			University of Notre Dame, \\
			Notre Dame 46556, USA\\
		}
		
		\date{}
		
		\maketitle

        	\begin{abstract}
			{We study the nonlinear stability of an equilibrium with a background magnetic field for the three-dimensional compressible magnetohydrodynamic (MHD) equations in the whole space $\mathbb{R}^3$, in the strongly anisotropic regime where the velocity is dissipated only in the horizontal directions and the magnetic field is diffused in a single direction. We prove that for initial data sufficiently close to the equilibrium in a Sobolev space, the system admits a unique global-in-time solution that remains close to the equilibrium and enjoys quantitative dissipation estimates. The proof overcomes the severe lack of dissipation through two mechanisms: the background magnetic field is shown to generate enhanced dissipation for the magnetic field and the density, while a nonlinear cancellation mechanism is devised to resolve the loss of vertical derivatives caused by the compressible coupling.}
			
			\vspace*{5pt}

			\vspace*{3pt}
		\end{abstract}

		\tableofcontents

		\section{Introduction}

Magnetohydrodynamics (MHD) governs the motion of electrically conducting fluids such as plasmas, liquid metals, and many astrophysical and geophysical flows, in which the velocity field and the magnetic field evolve through a strong, two-way nonlinear coupling: the fluid advects and stretches the magnetic field, while the field exerts a Lorentz force back on the fluid. For compressible fluids, the resulting model couples the compressible Navier--Stokes equations with the equations of magnetic induction, and it occupies a central place in plasma physics, astrophysics, and geophysical fluid dynamics \cite{D1993, Braginskii1965, D2001, Kulsrud2005}. A question of fundamental importance, both physically and mathematically, is the stability of equilibria. When a quiescent, magnetized fluid is slightly perturbed, does it relax back toward rest, or do the perturbations amplify? It is in this regime that the stabilizing role of an ambient magnetic field can be isolated and quantified rigorously.

A distinctive feature of many physically relevant MHD flows is that the dissipation is strongly anisotropic. In a strongly magnetized plasma, charged particles spiral tightly around the magnetic field lines, so that transport of momentum and diffusion of the field proceed almost freely along the field but are sharply suppressed across it; this directional dependence is built into the classical transport theory for magnetized fluids \cite{Braginskii1965, Kulsrud2005}. A parallel horizontal--vertical anisotropy is ubiquitous in rotating and stratified geophysical and astrophysical flows, where stable stratification and rotation organize the motion into nearly horizontal layers and inhibit vertical exchange. Such considerations make it natural to study compressible MHD models in which dissipation acts only along certain directions, and to ask whether so degenerate a dissipative structure can still produce stability. The analytical difficulty of these problems lies precisely in the missing directions of dissipation, where one must uncover hidden stabilizing mechanisms, typically generated by the equilibrium magnetic field through the velocity--field coupling, to compensate.

The present paper is concerned with one such model. We study the following anisotropic compressible MHD system in $\mathbb{R}^3$,
		\begin{equation}\label{eqr0}
			\left\{
            \begin{aligned}
				&\p_t \rho+{\rm div}(\rho u)=0,\\
				&\rho \p_t u+\rho u \cdot \nabla u-\p_{1}^2u-\p_{2}^2u
				-\nabla {\rm div}u+\nabla P(\rho)
				=B\cdot \nabla B-\frac{1}{2}\nabla |B|^2,\\
				&\partial_t B + u \cdot \nabla B
				-\p_{1}^2 B
				= B\cdot \nabla u-B {\rm div}u,\\
				&{\rm div} B= 0.
			\end{aligned}
            \right.
		\end{equation}
		Here $\rho$, $u$ and $B$ represent the density, velocity
		and magnetic field, respectively, and $x=(x_1,x_2,x_3)$.
		The dissipative structure of \eqref{eqr0} is degenerate: the velocity is dissipated only in the two horizontal directions $x_1, x_2$, while the magnetic field is diffused in the single direction $x_1$. There is no viscous dissipation of $u$ in $x_3$, no magnetic diffusion of $B$ in the $x_2$ or $x_3$ directions, and, as always for compressible flow, no dissipation of the density in any direction.
		For sake of simplicity, we assume the pressure satisfying the law
		\beqq
		P(\rho)=\frac{\rho^3}{3}.
		\deqq
        The qualitative conclusions of this paper extend to the general
$\gamma$-law $P(\rho) = \rho^\gamma/\gamma$ ($\gamma > 1$) with
minor modifications.
		System \eqref{eqr0} admits the physically important steady state
		\beqq
		\rho^* \equiv 1, \quad u^* \equiv 0,
		\quad B^* \equiv e_2:= (0,1,0).
		\deqq
Then, we introduce the perturbation variables
\[
\varrho:=\rho-1,
\qquad 
b:=B-e_2 .
\]
Substituting
\[
\rho=1+\varrho,
\qquad 
B=e_2+b
\]
into \eqref{eqr0}, we obtain the following system governing the perturbation $(\varrho, u, b)$:
		\begin{equation}\label{eqr}
			\left\{\begin{aligned}
				&\p_t \vro+{\rm div}u=F_1,\\
				&\p_t u-\p_{1}^2u-\p_{2}^2u-\nabla {\rm div}u
				+\nabla \vro+\nabla b_2-\p_2 b=F_2,\\
				&\p_t b-\p_{1}^2 b
				+e_2{\rm div}u-\p_2 u=F_3,\\
				&{\rm div}b=0,
			\end{aligned}\right.
		\end{equation}	
		with the initial data
		\beq\label{eqr-i}
		(\vro, u, b)|_{t=0}=(\vro_0, u_0, b_0).
		\deq
		Here the force terms $F_i(i=1,2,3)$ are defined by
		\begin{equation}\label{nonlinear}
			\left\{
            \begin{aligned}
				&F_1:=-u\cdot \nabla \vro-\vro \,{\rm div}u,\\
				&F_2:=-u \cdot \nabla u-\vro\nabla  \vro-\frac{\vro}{1+\vro}(\p_{1}^2u+\p_{2}^2u+\nabla {\rm div}u -\nabla b_2+\p_2 b)\\
				&\quad \quad
				+\frac{1}{1+\vro}b\cdot \nabla b
				-\frac{1}{2(1+\vro)}\nabla |b|^2,\\
				&F_3:=-u\cdot \nabla b-b\,{\rm div}u+b\cdot \nabla u.
			\end{aligned}
            \right.
		\end{equation}	
		The aim of this paper is to prove that the equilibrium $(\rho^*,u^*,B^*)=(1,0,e_2)$ is nonlinearly stable: for initial data that are a sufficiently small perturbation in a Sobolev space, the system \eqref{eqr} admits a unique global-in-time solution that stays close to the equilibrium and enjoys quantitative dissipation estimates. The challenge is that the dissipation built into \eqref{eqr} is far too weak, by itself, to control the nonlinear interactions: there is no dissipation for the velocity in $x_3$, none for the magnetic field in $x_2$ and $x_3$, and none for the density at all. Whether such a degenerate structure can be stabilized is not at all obvious, and resolving it requires extracting hidden dissipation from the ambient magnetic field $e_2$ and devising a mechanism to absorb a genuine loss of vertical derivatives. Our main result, stated next, gives an affirmative answer.

		\begin{theo}\label{main_result}
			Let $m \ge 3$ be an integer. Assume the initial data $(\vro_0,u_0, b_0)$ satisfying
			${\rm div}b_0=0$ and there exists a small positive constant $\delta_0$
			such that
			\begin{equation}\label{condition}
				\|(\vro_0, u_0, b_0)\|_{H^m}^2 \le \delta_0,
			\end{equation}
			then the equations \eqref{eqr}-\eqref{eqr-i} have a unique global solution $(\vro, u, b)$ satisfying
			\begin{equation}\label{global_estimate}
				\|(\vro, u, b)(t)\|_{H^m}^2
 +\int_0^t (\|(\p_1 u, \p_2 u, \mathrm{div} u, \p_1  b)(\tau)\|_{H^m}^2 
		+ \|(\p_1 \vro, \p_2 \vro, \partial_2 b)(\tau)\|_{H^{m-1}}^2) 
        d\tau \le C\delta_0,
			\end{equation}
			where $C$ is a positive constant independent of time $t$.
		\end{theo}

\begin{rema}
	The background magnetic field $B^*=(0,1,0)$ provides enhanced dissipative  structure for the derivative of magnetic field in the $x_2$ direction in Theorem \ref{main_result}.
	Due to the absence of vertical dissipation of magnetic field, we can only obtain the  horizontal dissipation of density by using the pressure term in the momentum equation.
\end{rema}

\begin{rema}
Due to the absence of dissipative structure in the
$x_3$ direction, the nonlinear term $b\cdot \nabla b$ including
in F$_2$ in \eqref{nonlinear} will cause the loss of vertical derivative problem.
Thus, we take the strategy of nonlinear cancellation mechanism
to overcome this problem, see next section in detail.
\end{rema}

\vskip .1in
The proof rests on two intertwined mechanisms: an enhanced dissipation extracted from the equilibrium magnetic field, which supplies the directions of damping that are missing from \eqref{eqr}, and a nonlinear cancellation mechanism, which removes a genuine loss of vertical derivatives created by the compressible coupling. We outline both here; the full argument occupies Sections \ref{section-approach} and \ref{global-estimate}.

The natural energy estimate for \eqref{eqr} only controls the dissipation $\|(\p_1 u,\p_2 u,\mathrm{div}\,u,\p_1 b)\|_{L^2_t H^m}$, which is too weak to absorb the nonlinear terms. The first idea is that the background field $e_2$ creates additional, hidden dissipation. Writing the velocity equation as
\beqq
\p_t u-\p_1^2 u-\p_2^2 u-\nabla\,\mathrm{div}\,u
+\underbrace{\nabla\vro+\nabla b_2-\p_2 b}_{\text{linear coupling}}=F_2,
\deqq
and testing successively against $\p_2 b$ and against $(\p_1 \vro, \p_2 \vro)$, the divergence-free condition $\mathrm{div}\,b=0$ first yields control of $\p_2 b$, and then, through the pressure term $\nabla\vro$, control of the horizontal density gradient $(\p_1 \vro, \p_2 \vro)$. In this way the linear coupling alone produces the enhanced dissipation estimate for $(\p_1 \vro, \p_2 \vro, \p_2 b)$, recovering damping in directions where neither the velocity nor the magnetic field is directly dissipated.

The second, and more serious, difficulty is the absence of any dissipation in the $x_3$ direction together with the very weak dissipative structure of the density. When estimating the highest-order vertical derivatives, several critical terms arise, for instance those of the form $\int \mathrm{div}\,u\,|\p_3^m\vro|^2 \, dx$ and $\int \p_2 u_1\,\p_3^m b_1\,\p_3^m b_2 \, dx$, which cannot be handled by the anisotropic Sobolev inequalities. Repeatedly substituting the momentum equation to trade vertical derivatives of the density for derivatives of the better-behaved magnetic field, and exploiting the pressure, eventually exposes a term that loses one vertical derivative and cannot be closed directly. The resolution is the nonlinear cancellation mechanism: the offending term is produced by the Lorentz nonlinearity $\tfrac{1}{1+\vro}b\cdot\nabla b$ in F$_2$, and an identical term of opposite sign is generated when the velocity equation is differentiated and paired against a suitable weighted multiple of $\p_3^m u$. Adding the two contributions, the derivative-losing terms cancel exactly, and the remaining pieces carry enough dissipation supplied by the very enhanced-dissipation estimates obtained above to be controlled. A closely related cancellation, again driven by $\mathrm{div}\,b=0$, disposes of the terms $\int b\cdot\nabla\p_i^m u\,\p_i^m b \, dx$ and $\int\frac{1}{1+\vro}b\cdot\nabla\p_i^m b\,\p_i^m u \, dx$ for $i=2,3$, and explains why the velocity field must be weighted by $1+\vro$ in the energy. With the horizontal and enhanced-dissipation estimates in hand, a bootstrap argument closes the global-in-time bounds and yields Theorem \ref{main_result}.

\vskip .1in
The mathematical study of stability and large-time behavior for compressible Navier--Stokes and MHD systems goes back to the seminal work of Matsumura and Nishida \cite{Matsumura-Nishida1980,Matsumura-Nishida1983} and the structural theory of Shizuta, Kawashima and collaborators \cite{Umeda-Kawashima-Shizuta1984,Shizuta-Kawashima1985,Kawashima1984}, which established global existence and decay for small perturbations of equilibria under full dissipation. For the compressible MHD equations with full dissipation, a rich theory of strong solutions, weak solutions with large data, and blow-up criteria has been developed \cite{Chen-Wang2002,HW2010,HL2013}. The problems considered in the present paper are far more degenerate, and a central theme of recent research has been to identify the minimal dissipation, often supplemented by the stabilizing effect of a background magnetic field, under which stability still holds.

This stabilizing effect is by now firmly established in the incompressible setting. Nonlinear stability for the ideal MHD equations near a background field was obtained in \cite{Longtime1988,HXY2018,PZZ2018,WZ2017}, and an extensive literature treats incompressible MHD with partial dissipation, fractional dissipation, or damping
\cite{DZ2018,RWXZ2014,TW2018,WWX2015,CWY2014,DLW2017,Fefferman2014,
Wu2021Advance,new-Abidi-Zhang2017,new-Boardman-Lin-Wu2020,
new-Cao-Regmi-Wu2013,new-Cao-Wu2011,new-Chen-Zhang-Zhou2022,
new-Du-Zhou2015,new-Fefferman-McCormick-Robinson-Rodrigo2017,
new-Jiang-Jiang-Zhao2022,new-Lai-Wu-Zhang2021,new-Lai-Wu-Zhang2022,
new-Li-Yang-2024,new-Lin-Ji-Wu2020,new-Lin-Zhang2014,new-Lin-Zhang2015,
new-Sermange-Temam1983,new-Xie-Jiu-Liu2024,new-Xu-Zhang2015,
new-Zhang2016,new-Hu-Lin-arxiv,Zhang-2009}.
A recurring conclusion is that the velocity--field coupling generates an effective dissipation that compensates for missing viscosity; most strikingly, the three-dimensional incompressible MHD equations have recently been shown to be globally well-posed near a background field with dissipation in only one direction of the momentum equation \cite{Gao2025,Lai-Wu-Zhang-Zhao2026,Lin-Wu-Zhu-2025}. These developments are the direct inspiration for the present work.

For the compressible MHD system near a background field, the picture is considerably less complete. In the viscous, non-resistive case, Hu and Lin \cite{new-Hu-Lin-arxiv} established well-posedness for the two-dimensional equations with a special class of data; Wu and Wu \cite{WW2017} introduced a systematic wave-equation reformulation and Fourier-analytic approach in the whole space; Wu and Zhu \cite{Wu-Zhu-2022} obtained global existence and decay in the two-dimensional periodic setting; and, under a Diophantine condition on the background field, Wu and Zhai \cite{Wu-Zhai-2023} proved global existence and stability in a two-dimensional periodic domain. By contrast, whether smooth solutions of the three-dimensional compressible viscous, non-resistive MHD equations in the whole space are always global remains a challenging open problem. The compressible framework is genuinely harder than the incompressible one: the density equation carries no dissipation of its own, so the enhanced dissipation produced by the magnetic field must simultaneously stabilize the velocity and the density, and the incompressible techniques do not transfer directly.

\vskip .1in
Our system \eqref{eqr0} retains only horizontal velocity dissipation and one-directional magnetic diffusion, which places it strictly between the fully resistive compressible MHD equations and the resistivity-free model. Theorem \ref{main_result} shows that even this minimal dissipative structure, assisted by the background field, suffices for global stability in the three-dimensional whole space. The result is also natural from the hydrodynamic side: deleting the magnetic field from \eqref{eqr0} leaves the compressible Navier--Stokes equations with only horizontal dissipation, whose global stability was recently established in \cite{FHWW2025}; our theorem incorporates the magnetic coupling while preserving the same anisotropic, partially dissipative character. We mention that the corresponding anisotropic incompressible Navier--Stokes equations have a well-developed theory of well-posedness in Sobolev and Besov spaces \cite{ani-CDGG2000,Chemin2007,Iftimie2002,Liu-Paica-Zhang-2020,Paicu2005,Zhang-2009,new-Cao-Wu2025} and of large-time behavior in $\mathbb{T}^2\times\mathbb{R}$ and $\mathbb{R}^3$ \cite{jWy2021,JTW2023}. More important than the specific result, we believe, are the two mechanisms developed here: the extraction of enhanced dissipation from the background field and the nonlinear cancellation that defeats the loss of vertical derivatives, which are robust and should be useful for other anisotropic compressible models with severely degenerate dissipation.

\medskip 
Throughout this paper, we use symbol $A \lesssim B$ for $ A\le C B$ where $C>0$ is a constant which may change from line to line and independent of time $t$.

\medskip 
        The rest of the paper is organized as follows.
		In Section \ref{section-approach}, we explain the difficulties and our approach to establish
		the global in time estimates
		for the system  \eqref{eqr}.
        In Section  \ref{global-estimate}, applying the energy method, we establish the global in time estimates
		for system  \eqref{eqr}  and then prove Proposition \ref{main_pro}. 
        Finally, under the help of Proposition \ref{main_pro} and using continuity argument, we prove Theorem \ref{main_result} in Section \ref{sec-theorem}. 
        
        \section{Difficulties and outline of our approach}\label{section-approach}
In this section, we will explain the main difficulties of proving Theorem 
			\ref{main_result} and our strategies for overcoming it.
           In order to establish the global estimates
\eqref{global_estimate}
in Theorem \ref{main_result}, our proof will divide into the following three steps.

{\textbf{Step 1: Estimates of horizontal derivative and enhanced dissipation.}}  
First of all, we establish the estimates for the horizontal derivative
		of density, velocity and magnetic field as well as control the nonlinear terms by energy norm that includes the vertical derivative of solution
		and the horizontal derivative of dissipation norm
		with the help of anisotropic Sobolev inequalities.
        In this process, the natural dissipation $\|(\p_1 u, \p_2 u, {\rm div u}, \p_1 b)\|_{L^2(0,T;H^{m})}$ cannot fully control the
nonlinear terms. To overcome this, we must exploit the enhanced dissipation induced by the background magnetic
field.
Indeed, from  equation \eqref{eqr}, we have
\begin{equation}\label{201}
\p_t u-\p_{1}^2u-\p_{2}^2u-\nabla {\rm div}u
				+\underset{\rm linear~coupling}{\underbrace{\nabla \vro+\nabla b_2-\p_2 b}}=F_2.
\end{equation}
It seems difficult to establish the enhanced dissipation
estimates $\|(\p_1 \vro, \p_2 \vro, \partial_2 b)\|_{L^2(0, T; H^{m-1})}$ by using the linear terms in \eqref{201}.
However, by taking the inner product of $\p_2 b$ with 
\eqref{201}, we have
\begin{equation*}
\i |\p_2 b|^2 dx
=\i (\nabla \vro+\nabla b_2) \cdot \p_2 b \, dx
+\underset{\rm Easy~terms}{\underbrace{\i (\p_t u-\p_{1}^2u-\p_{2}^2u-\nabla {\rm div}u-F_2)
\cdot \p_2 b \, dx}},
\end{equation*}
which, together with the  condition ${\rm div} b=0$, yields 
the enhanced dissipation estimates $\|\p_2 b\|_{L^2(0, T; L^2)}$.
This, together with the dissipation of magnetic field in the 
$x_1$ direction, yields the horizontal dissipation for 
the magnetic field $\|(\p_1 b, \p_2 b)\|_{L^2(0, T; L^2)}$.

Next, we take inner product of the horizontal part of \eqref{201} 
with $\nabla_h \varrho$ to obtain
\begin{equation*}
\i |\nabla_h \vro|^2 dx
=\i (\p_2 b_h-\nabla_h b_2) \cdot \nabla_h \vro \, dx
+\i (-\p_t u+\p_{1}^2u+\p_{2}^2u+\nabla {\rm div}u+F_2)_h
\cdot \nabla_h \vro \, dx,
\end{equation*}
where $\nabla_h:=(\p_1, \p_2, 0)$.
It is easy to 
obtain the enhanced dissipation estimates
$\|\nabla_h \vro\|_{L^2(0, T; L^2)}$.
Thus, we can take this strategy to establish the important
enhanced dissipation estimates 
$\|(\nabla_h \vro, \partial_2 b)\|_{L^2(0, T; H^{m-1})}$
from the linear terms in \eqref{201}.

    {\textbf{Step 2: Cancellation  mechanism.}}     
        Due to the only horizontal dissipation of velocity and magnetic field, we can not apply the anisotropic
		Sobolev inequalities to estimate the terms  
        for $i=2,3$
		$$
		\i b \cdot \nabla \p_i^m u \, \p_i^m b \ dx
		~{\rm and}~ \i \f{1}{1+\vro} b \cdot \nabla \p_i^{m} b \, \p_i^m u\ dx.
		$$
		Then, we can overcome this difficulty by using 
        cancellation mechanism. To be precise, we have
		\beqq
		\i  b \cdot \nabla \p_i^{m} u \, \p_i^m b \ dx 
        +  \i  b \cdot \nabla \p_i^{m} b \, \p_i^m u \ dx  = - \i {\rm{div}} b \, \p_i^{m} u \, \p_i ^m b \ dx =0.
		\deqq
		Thus, it is necessary to  add the weight $1+\vro$ to velocity field as we establish the energy estimates, see Lemma \ref{lemma32} in detail.

{\textbf{Step 3: Estimates of vertical derivative.}}  
In order to close the energy estimates, we need to establish the estimates for the vertical derivative of density, velocity and magnetic field.
Due to the lack of vertical dissipation of the
magnetic field and the lower dissipative structure of density in the horizontal directions, we will encounter some problematic terms, such as $-\i \du \, |\p_3^m \vro|^2 \, dx,
\i \nabla_h \cdot u_h \, |\p_3^m \vro|^2 \, dx$
and $ \i \p_2 u_1 \, \p_3^m b_1 \, \p_3^m b_2 \, dx$.
In order to deal with these weak dissipative structure terms,
we will use the good effect of pressure and background magnetic field to control this term.

First of all, let us deal with the difficult term
$-\i \du \, |\p_3^m \vro|^2 \, dx$  in four iterations as follows.

{\textbf{The first iteration:}}  
Since the magnetic field has the good dissipative structure
in the $x_1$ direction, we apply the momentum equation to
convert the critical vertical derivative of density into
magnetic field.
Thus, we  substitute the equation of $\p_3 \vro$
\beq\label{202}
         \p_3 \vro = -\p_t u_3 + \p_1^2 u_3 + \p_2^2 u_3+ \p_3 \du+\p_2 b_3-
         \underset{\rm Bad~terms}{\underbrace{ \p_3 b_2+ (F_2)_3 }}
 \deq
into $-\i \du \, |\p_3^m \vro|^2 \, dx $,
then we have
\beq\label{203}
\begin{aligned}
-\i \du \, |\p_3^m \vro|^2 \, dx  
=&-\i \du \, \p_3^m \vro  \,
\p_3^{m-1}( -\p_t u_3 + \p_1^2 u_3 + \p_2^2 u_3+ \p_3 \du+\p_2 b_3)\, dx  \\
&+  \underset{\rm Bad~terms}{\underbrace{\i \du \, \p_3^m \vro \, \p_3^m b_2 \, dx  
-\i \vro \, \du \, \p_3^m \vro \,  
\p_3^{m-1} (F_2)_3  \, dx.}}
\end{aligned}
\deq
The last term on the right-hand side of \eqref{203}
will generate a difficult term
$\i \vro \, \du \, |\p_3^m \vro|^2 \, dx$.
Substituting the equation of $\du$
        \beq \label{204}
       \du = -\p_t \vro - u \cdot \nabla \vro - \vro \, \du,
        \deq
        into this term, we have
        \beq\label{205}
        \i \vro \, \du \, |\p_3^m \vro|^2 \, dx = 
       -\i \vro \, (\p_t \vro + \vro \, \du+u \cdot \nabla \vro) |\p_3^m \vro|^2 \, dx.
        \deq
       Then, we transfer the time derivative from $\p_t \vro$ into $\p_t \p_3^m \vro$ and using $\eqref{eqr}_1$ again to
		deal with the first term and the third term by the cancellation mechanism on the right-hand side of \eqref{205}. 

{\textbf{The second iteration:}}  
In order to deal with the term 
$\i \du \, \p_3^m \vro \, \p_3^m b_2 \, dx$,
    we  substitute the equation of $\p_3 \vro$ to the second term of $\eqref{203}$ again, then it holds
         \beq \label{206}
         \begin{aligned}
         \i \du \, \p_3^m \vro \, \p_3^m b_2 \, dx =
         &\underset{\rm Easy~terms}{\underbrace{\i \du \{\p_t u_3 + \p_1^2 u_3 + \p_2^2 u_3+ \p_3 \du+\p_2 b_3+ (F_2)_3\} \p_3^m b_2 \, }}
          dx\\
         &\underset{\rm Bad~ term}{\underbrace{-\i \du \, |\p_3^m b_2|^2 \, dx}}.
        \end{aligned}
         \deq
Due to the good dissipative structure of magnetic field 
in the $x_1$ direction, we convert the estimate of 
$-\i \du \, |\p_3^m \vro|^2 \, dx $ 
into that of $-\i \du \, |\p_3^m b_2|^2 \, dx$
by employing equation \eqref{202} in the above two iterations.

{\textbf{The third iteration:}}  
Using equations \eqref{204} and \eqref{eqr}$_3$, 
we can rewrite the difficult term in \eqref{206}  as follows
\begin{equation}\label{207}
\begin{aligned}
-\i \du \, |\p_3^m b_2|^2 \, dx
=&\frac{d}{dt}\i \vro \, |\p_3^m b_2|^2 \, dx
+\i  (u \cdot \nabla \vro+\vro \, \du) |\p_3^m b_2|^2 \, dx\\
&+2\i \vro \, \p_3^m b_2 \, \p_3^m (-\p_{1}^2 b_2
				+{\rm div}u-\p_2 u_2
                +u\cdot \nabla b_2+b_2\,{\rm div}u)\, dx\\
& -\underset{\rm Bad ~ term}{\underbrace{2\i \vro\, \p_3^m b_2  \,\p_3^m(b\cdot \nabla u_2)\, dx}}.
\end{aligned}
\end{equation}
The last term on the right-hand side of \eqref{207} will 
generate the most difficult term $-2\i \vro \, \p_3^m b_2 \, b \cdot \p_3^m \nabla u_2 \, dx$. 
In order to deal with this term, we integrate by parts to obtain
        \beq\label{208}
        -2\i \vro \, \p_3^m b_2 \, b \cdot \p_3^m \nabla u_2 \, dx =\underset{\rm Bad~ term:~loss~of~vertical~derivative}{\underbrace{ 2\i \vro \, \p_3^m u_2 \, b \cdot \p_3^m \nabla b_2 \, dx}} + 2\i b \cdot \nabla \vro\,\p_3^m u_2 \, \p_3^m  b_2 \, dx,
        \deq
       where the former term creates the loss of vertical derivative problem.

{\textbf{The fourth iteration and nonlinear cancellation  mechanism:}}  
 However,  we can eliminate the problematic term by subtracting them 
	in an appropriate way by using the equation of magnetic field. To be precise,  applying $\p_3^m$-operator to the equation $\eqref{eqr}_2$ 
    \beqq
    \p_t \p_3^m u_2=
    \p_3^m\{\p_{1}^2 u_2+\p_{2}^2 u_2+\p_2 {\rm div}u
				-\p_2 \vro+(F_2)_2\}
    \deqq
    and multiplying by $2(1+\vro) \vro\, \p_3^m u_2$, we have
        \beq\label{209}
        \bal
           2\i \p_t \p_3^m u_2 \, (1+ \vro) \vro \, \p_3^m u_2 \, dx
          = &\;  
          \underset{\rm Cancelate~bad~term~in~\eqref{208} }{\underbrace{2\i \vro \, \p_3^m u_2 \, b \cdot \p_3^m \nabla b_2 \, dx }}
          - 2\i (1+ \vro)^2 \vro \, \p_3^m  \p_2 \vro \, \p_3^m u_2 \, dx + \cdots,
        \dal 
        \deq
       where the symbol $\cdots$ represents the good dissipative structure terms which are easier to be controlled that
	    we want the reader to ignore at this moment. 
        Here we point out that the nonlinear term
       $\frac{1}{1+\vro}b\cdot \nabla b_2$ in F$_2$ generates
       the new nonlinear term 
       $\i \vro \, \p_3^m u_2 \, b \cdot \p_3^m \nabla b_2 \, dx$
       on the right-hand side of \eqref{209}.
      Substituting  \eqref{209} into \eqref{208}, we obtain that
       \beqq \bal
       - 2\i \vro \, \p_3^m b_2 \, b \cdot \p_3^m \nabla u_2 \, dx = &2\i b \cdot \nabla \vro\,\p_3^m u_2 \, \p_3^m  b_2 \, dx+ \i \p_t \p_3^m u_2 \, (1+ \vro) \vro \, \p_3^m u_2 \, dx \\
       &+ 2
       \underset{\rm Bad~ term}{\underbrace{\i (1+ \vro)^2 \vro \, \p_3^m  \p_2 \vro \, \p_3^m u_2 \, dx }} + \cdots.
       \dal \deqq
Thus, we overcome the loss of vertical derivative problem
by the method of  \textit{nonlinear cancellation mechanism}.
Thanks to the derivative in the $x_2$ direction on the density, admitting enough dissipation through integration by parts and 
using the equation \eqref{202}, we can control the bad terms
$\i (1+ \vro)^2 \vro \, \p_3^m  \p_2 \vro \, \p_3^m u_2 \, dx $.
Thus, we apply this nonlinear cancellation  mechanism method to 
overcome the loss of vertical derivative problem.

\medskip 
Next, let us deal with difficult term 
$\i \nabla_h \cdot u_h \, |\p_3^m \vro|^2 \, dx$.
Similar to the difficult term $-\i \du \, |\p_3^m \vro|^2 \, dx$,
we use the equation \eqref{202} twice to convert the estimate of this term into that of new term
$\i \nabla_h \cdot u_h \, |\p_3^m b_2|^2 \, dx$.
Obviously, we have
         \beqq
         \i \nabla_h \cdot u_h \, |\p_3^m b_2|^2 \, dx = 
         \underset{\rm Easy~ term}{\underbrace{\i \p_1 u_1 \, |\p_3^m b_2|^2 \, dx}} 
         +\underset{\rm Bad~ term}{\underbrace{\i \p_2 u_2 \, |\p_3^m b_2|^2 \, dx}}.
         \deqq
 Apply the equation of $\p_2 u_2$ in \eqref{eqr}
         \beqq
         \p_2 u_2 = \p_t b_2 - \p_1^2 b_2 + \du - (F_3)_2,
         \deqq
         thus we have
         \beq \label{210}
         \begin{aligned}
      \i \p_2 u_2 \, |\p_3^m b_2|^2 \, dx
      &=  \f{d}{dt} \i  b_2 \, |\p_3^m b_2|^2 \, dx 
      + 2 \i b_2  \p_3^m b_2 \,  \p_3^m (-\p_1^2 b_2+\p_1 u_1+\p_3 u_3) \, dx\\
      &- 2 \underset{\rm Bad~ term}{\underbrace{\i b_2  \p_3^m b_2 \,  \p_3^m (F_3)_2 \, dx}}
      +\i\{\du-\p_1^2 b_2- (F_3)_2\} \, |\p_3^m b_2|^2 \, dx.
      \end{aligned}
         \deq
    The bad term \eqref{210} generates the difficult term
    $\i b_2 \, \p_3^m b_2 \, b \cdot \p_3^m \nabla u_2 \, dx$ 
    that cause the loss of vertical derivative.
    Thus, we integrate by parts to obtain 
        \beqq
         \i b_2 \, \p_3^m b_2 \, b \cdot \p_3^m \nabla u_2 \, dx =  \underset{\rm Bad~ term}{\underbrace{ -\i b_2 \, \p_3^m u_2 \, b \cdot \p_3^m \nabla b_2 \, dx}} - \i b \cdot \nabla b_2 \,\p_3^m u_2 \, \p_3^m  b_2 \, dx,
        \deqq
         which can be estimated similar to the estimate of \eqref{208} by the nonlinear cancellation mechanism. 
         Thus we end the estimate of difficult term 
         $\i \nabla_h \cdot u_h \, |\p_3^m \vro|^2 \, dx$,
         despite the complexity and length of the process. 
         
\medskip 
Finally, similar to the estimate of 
$\i \p_2 u_2 \, |\p_3^m b_2|^2 \, dx$, 
we can apply the equation
\beqq
\p_2 u_1 =\p_t b_1 - \p_1^2 b_1 - (F_3)_1,
\deqq
and the nonlinear cancellation  mechanism to control the
difficult term $ \i \p_2 u_1 \, \p_3^m b_1 \, \p_3^m b_2 \, dx$.
Therefore, we complete the proof for the vertical 
derivative of solution.

        \section{Global in time estimates}\label{global-estimate}
In this section, we will establish global-in-time estimates for the system \eqref{eqr} under the condition of  small initial data \eqref{condition}.  First of all, 
		similar to the result in \cite{Matsumura-Nishida1980},
		one can establish the 
		local-in-time existence and uniqueness for system \eqref{eqr}.
		Then, we will extend the local-in-time solution to be global.
		Thus, our target in this section is to establish the global-in-time
		regularity under the condition of small initial data  \eqref{condition}.

        Let us define  $\alpha_h:=(\alpha_1, \alpha_2, 0)$.
		We use the following notations, for every $m \in \mathbb{N}$:
		\begin{equation*}
			\begin{aligned}
				&\Vert f \Vert_{H^m_{tan}}^2 :=
				\sum\limits_{0\le |\alpha_h| \le m} \Vert \p^{\alpha_h} f \Vert_{L^2}^2,
				\quad
				\Vert f \Vert_{H^m}^2 :=
				\sum\limits_{0\le|\alpha| \le m} \Vert \p^\alpha f \Vert_{L^2}^2.
			\end{aligned}
		\end{equation*}
        Recall 
         $\nabla_h=(\p_1, \p_2, 0)$, then we define the energy functional
\begin{equation*}
	\begin{aligned}
		\me(t)
		:=& \|(\vro, u, b)(t)\|_{H^m}^2,
	\end{aligned}
\end{equation*}
and the dissipation functional
\begin{equation*}
	\begin{aligned}
		\md(t)
		:= \|(\nabla_h u, \mathrm{div} u, \p_1  b)(t)\|_{H^m}^2 
		+ \|(\partial_2 b, \nabla_h \vro)(t)\|_{H^{m-1}}^2.
	\end{aligned}
\end{equation*}
        
        Now, let us state the global estimates as follows. 
		\begin{prop}\label{main_pro}
			Let $m \ge 3$ be an integer, assume the initial data $(\vro_0, u_0, b_0)$ satisfying
			${\rm div} b_0=0$.
			For the solution $(\vro,u, b)$ of equation \eqref{eqr}
			defined on $ [0,T] \times \mathbb{R}^3$, assume there exists
			a small positive constant $\delta$ such that for any $t \in (0, T]$
			\begin{equation}\label{assumption}
				\underset{0\le \tau \le t}{\sup} \me(\tau)+\int_0^t \md(\tau) d\tau \le \delta,
			\end{equation}
			then the  solution of equation \eqref{eqr} has the estimate
			\begin{equation}\label{close_assumption}
				\underset{0\le \tau \le t}{\sup} \me(\tau)+\int_0^t \md(\tau) d\tau \le \frac{\delta}{2},
			\end{equation}
			where  the small positive constant $\delta:= 2C \me(0)$,
			and $C$ is a positive constant independent of time $t$.
		\end{prop}

        Before we give the proof of the above Proposition
        \ref{main_pro}, let us state some anisotropic Sobolev inequalities used frequently in our paper.
			\begin{lemm}\label{lemm:sobolev-ie}
				For any suitable functions $(f(x), g(x), h(x), v(x))$ defined on $\mathbb{R}^3$ and different numbers $i,j,k \in \{1, 2, 3\}$,
				the following estimates hold
				\beq \label{ie:Sobolev}
				\bal
				\|f\|_{L^\infty}  \lesssim &\; \|f\|_{L^2}^{\f18}\|\p_1 f\|_{L^2}^{\f18}\|\p_2 f\|_{L^2}^{\f18}\|\p_{12} f\|_{L^2}^{\f18}
				\|\p_3 f\|_{L^2}^{\f18}\|\p_{13} f\|_{L^2}^{\f18}\|\p_{23} f\|_{L^2}^{\f18}\|\p_{123} f\|_{L^2}^{\f18},\\
				\int_{\mathbb{R}^3} |f g h |\, dx
				\lesssim &\;  \|f\|_{L^2}
				\|g\|_{L^2}^{\frac12}
				\|\p_i g\|_{L^2}^{\frac12}
				\|h\|_{L^2}^{\frac14}
				\|\p_j h\|_{L^2}^{\frac14}
				\|\p_{k} h\|_{L^2}^{\frac14}
				\|\p_{jk} h\|_{L^2}^{\frac14},\\
				\int_{\mathbb{R}^3} |f g h |\, dx
				\lesssim &\; \|f\|_{L^2}^{\frac12}
				\|\p_1 f\|_{L^2}^{\frac12}
				\|g\|_{L^2}^{\frac12}
				\|\p_2 g\|_{L^2}^{\frac12}
				\|h\|_{L^2}^{\frac12}
				\|\p_3 h\|_{L^2}^{\frac12},\\
				\int_{\mathbb{R}^3} |f g h v|\, dx
				\lesssim &\; \|f\|_{L^2}^{\frac12}
				\|\p_i f\|_{L^2}^{\frac12}
				\|g\|_{L^2}^{\frac14}
				\|\p_j g\|_{L^2}^{\frac14}
                \|\p_k g\|_{L^2}^{\frac14}
                \|\p_{jk} g\|_{L^2}^{\frac14} \\
				&\; \times \|h\|_{L^2}^{\frac12}
				\|\p_i h\|_{L^2}^{\frac12}\|v\|_{L^2}^{\frac14}
				\|\p_j v\|_{L^2}^{\frac14}
                \|\p_k v\|_{L^2}^{\frac14}
                \|\p_{jk} v\|_{L^2}^{\frac14}.
				\dal
				\deq
			\end{lemm}
			 One can follow the idea in \cite[Lemma 1.2]{Wu2021Advance} to establish the inequalities $\eqref{ie:Sobolev}_1$-$\eqref{ie:Sobolev}_4$. Here we omit the proof of this lemma.
		
         \subsection{Vertical derivative estimates}
         It is well-known that the norm $ \|f\|_{H^m}$ is equivalent to the norms $\|f\|_{H_{tan}^m}$ and $\|  \p_3^m f\|_{L^2}$. 
        In this subsection, we will first establish the estimates for the vertical derivative of velocity, density and magnetic field.
  This part poses difficulties, since the low order dissipation for density and the lack of dissipation in the $x_3$ direction create obstacles in the treatment of nonlinear terms. 
  Specifically, we need to
deal with the  two types of challenging terms which clearly cannot be controlled by $\sqrt{\me(t)} \md(t)$. The first type is as follows
\beqq \bal
-\i \du \, |\p_3^m \vro|^2 \, dx~ {\rm and}~\i \nabla_h \cdot u_h \, |\p_3^m \vro|^2 \, dx,
\dal \deqq
which both contain the highest-order derivatives of density and are estimated in Lemma \ref{lemm:Nlinear1} and Lemma \ref{lemm:Nlinear2} respectively.
The second type 
involves integrals where contain the highest-order derivatives of magnetic field:
\beqq
-\i \du \, |\p_3^m b_h|^2 \, dx, ~ \i \p_2 u_2 \, |\p_3^m b_h|^2 \, dx ~ {\rm and}~  \i \p_2 u_1 \, \p_3^m b_1 \, \p_3^m b_2 \, dx.
\deqq
The first and second term have already been estimated in the process of estimating the first type, see \eqref{Nlinear1-01-2} in Lemma \ref{lemm:Nlinear1} and \eqref{Nlinear2-01-2}-\eqref{Nlinear2-01-3} in  Lemma \ref{lemm:Nlinear2} respectively, and refer to Lemma \ref{lemm:Nlinear3} for the estimate of the last term.
 
\subsubsection{Preparatory tools}
First of all, we estimate the former term in the first type, which contains the highest-order vertical derivatives of density. 
 In this case, we should exploit the energy structure of $\p_3 \vro$ and $\du$ as well as use the cancellation mechanism to overcome the loss of vertical derivative. 
\begin{lemm}\label{lemm:Nlinear1}
     Under the assumption \eqref{assumption}, for any smooth solution $(\vro ,u, b)$ of equation \eqref{eqr}, it holds
     \beq\label{Nlinear1-01-1} \bal
  -\i \du \, |\p_3^m \vro|^2 \, dx 
\le  &\; \f{d}{dt} \i \du \, (\p_3^{m} \vro - \p_3^m b_2) \, \p_3^{m-1} u_3 \, dx  +\f{d}{dt} \i \vro \, |\p_3^m b_2 |^2\, dx\\
&\; + \f{d}{dt}  \i  (1+ \vro) \vro \,|\p_3^m u_2|^2  \, dx   - 2 \f{d}{dt}  \i (1+ \vro) \, \vro \,  \p_3^{m-1} \p_2 u_3 \, \p_3^m  u_2 \, dx  \\
&\;  - \f12\f{d}{dt} \i \vro^2 \, |\p_3^{m} \vro|^2 \, dx +C\sqrt{\me(t)} \md(t),
     \dal \deq
    and
     \beq\label{Nlinear1-01-2} \bal
   - \i \du \, |\p_3^m b_h|^2 \, dx 
  \le &\; \f{d}{dt} \i \vro \, |\p_3^m b_h |^2\, dx+ \f{d}{dt}  \i  (1+ \vro) \vro \,|\p_3^m u_h|^2  \, dx \\
&\; - 2 \f{d}{dt}  \i (1+ \vro) \, \vro \, \p_3^m u_h \cdot \p_3^{m-1} \nabla_h u_3  \, dx + C \sqrt{\me(t)} \md(t).
     \dal \deq
     \end{lemm}
\begin{proof}
     Substituting the equation of $\p_3 \vro$
         \beqq \bal
         \p_3 \vro = &\;  -\p_t u_3 +\p_1^2 u_3 + \p_2^2 u_3+ \p_3 \du+\p_2 b_3 - \p_3 b_2 + (F_2)_3 \\
         = &\; -\p_t u_3 - \p_3 b_2 - \vro\, \p_3 \vro + \f{1}{1+\vro} \Big\{\p_1^2 u_3 + \p_2^2 u_3+ \p_3 \du+\p_2 b_3+ b \cdot \nabla b_3 - \f12 \p_3 (|b|^2)\Big\} \\
         &\; +\f{\vro}{1+\vro} \p_3 b_2 - u\cdot \nabla u_3,
         \dal \deqq
         into $-\i \du \, |\p_3^m \vro|^2 \, dx$, we have
         \beq \label{Nlinear1-02} \bal
         &\; -\i \du \, |\p_3^m \vro|^2 \, dx \\
         = &\; \f{d}{dt} \i \du \, \p_3^{m} \vro \, \p_3^{m-1} u_3 \, dx   - \i \p_t \du \, \p_3^{m} \vro \, \p_3^{m-1} u_3  \, dx -  \i \du \, \p_3^{m} \p_t \vro \, \p_3^{m-1} u_3  \, dx   \\
         &\; + \i \du \, \p_3^{m} \vro \, \p_3^{m} b_2 \,dx  +  \i 
         \vro \, \du \, |\p_3^{m} \vro|^2 \, dx + \i \du \, \p_3^{m} \vro \, (\p_3^{m-1} (\vro \,\p_3 \vro) - \vro \, \p_3^m \vro) \, dx \\
         &\; - \i \du \, \p_3^{m} \vro \, \p_3^{m-1} \Big\{\f{1}{1+\vro} (\p_1^2 u_3 + \p_2^2 u_3+ \p_3 \du+\p_2 b_3) \Big\} \, dx \\
         &\; - \i \du \, \p_3^{m} \vro \, \p_3^{m-1} \Big\{\f{1}{1+\vro} (b \cdot \nabla b_3 - \f12 \p_3 (|b|^2)) +\f{\vro}{1+\vro} \p_3 b_2- u\cdot \nabla u_3\Big\} \, dx\\
        := &\;  \f{d}{dt} \i \du \, \p_3^{m} \vro \, \p_3^{m-1} u_3 \, dx  
        +\sum_{i=1}^{7} I_{i}.
         \dal \deq
        {\bf{Estimate of the term $I_1$.}} Substituting the equation $\eqref{eqr}_2$ into the term $I_1$ and due to the condition ${\rm div} b=0$, we decompose it as follows
        \beqq \bal
        I_1 = &\;- \i {\rm div}(\p_1^2 u + \p_2^2 u + \nabla \du) \, \p_3^{m} \vro \, \p_3^{m-1} u_3  \, dx  + \i \Delta b_2 \, \p_3^{m} \vro \, \p_3^{m-1} u_3  \, dx \\
        &\; +\i \Delta \vro \, \p_3^{m} \vro \, \p_3^{m-1} u_3  \, dx  -\i {\rm div} F_2 \, \p_3^{m} \vro \, \p_3^{m-1} u_3  \, dx\\
        := &\; \sum_{i=1}^{4} I_{1,i}.
        \dal \deqq
        Using the anisotropic type inequality \eqref{ie:Sobolev} and integrating by parts, we can check that
        \begin{align*}
    I_{1,1} \lesssim &\; \| \p_3^m \vro \|_{L^2} \| \p_3^{m-1} u_3 \|_{L^2}^{\f14} \| \p_1 \p_3^{m-1} u_3 \|_{L^2}^{\f14} \| \p_3^{m} u_3 \|_{L^2}^{\f14} \| \p_1 \p_3^{m} u_3 \|_{L^2}^{\f14}\\
    &\; \times \| {\rm div}(\p_1^2 u + \p_2^2 u  + \nabla \du)\|_{L^2}^{\f12} \| \p_2 {\rm div}(\p_1^2 u + \p_2^2 u  + \nabla \du)\|_{L^2}^{\f12}\\
    \lesssim &\;  \sqrt{\me(t)} \md(t),\\
    I_{1,2} = &\;- \i \p_3^{m-1} \vro \,  ( \Delta b_2 \, \p_3^{m} u_3+\p_3 \Delta b_2 \, \p_3^{m-1} u_3 )\, dx \\
     \lesssim &\; \| \p_3^{m-1} \vro \|_{L^2}^{\f12} \| \p_2 \p_3^{m-1} \vro \|_{L^2}^{\f12} (\| \p_3^{m} u_3 \|_{L^2}^{\f12} \|  \p_3^{m+1} u_3 \|_{L^2}^{\f12} \| \Delta b_2 \|_{L^2}^{\f12} \| \p_1 \Delta b_2\|_{L^2}^{\f12}\\ 
     &\;  +  \| \p_3^{m-1} u_3 \|_{L^2}^{\f12} \| \p_3^{m} u_3 \|_{L^2}^{\f12} \| \p_3 \Delta b_2 \|_{L^2}^{\f12} \| \p_1  \p_3 \Delta b_2 \|_{L^2}^{\f12} )\\
     \lesssim &\;  \sqrt{\me(t)} \md(t).
    \end{align*}
    If $m=3$,  using the anisotropic type inequality \eqref{ie:Sobolev} and integrating by parts, we have
    \begin{align*}
    I_{1,3} = &\;   \f12 \i \p_3(\p_3^{2} \vro)^2 \, \p_3^{2} u_3  \, dx + \i \Delta_h \vro \, \p_3^{3} \vro \, \p_3^{2} u_3  \, dx \\
    = &\; -\f12 \i (\p_3^{2} \vro)^2 \, \p_3^{3} u_3  \, dx+\i \Delta_h \vro \, \p_3^{3} \vro \, \p_3^{2} u_3  \, dx \\
    \lesssim &\; \|\p_3^{2} \vro\|_{L^2} \| \p_1\p_3^{2} \vro \|_{L^2}^{\f12}    \| \p_2\p_3^{2} \vro \|_{L^2}^{\f12}  \| \p_3^3 u_3\|_{L^2}^{\f12} \| \p_3^4 u_3\|_{L^2}^{\f12} \\
    &\; + \|\p_3^{3} \vro\|_{L^2} \| \Delta_h \vro \|_{L^2}^{\f12}    \| \p_3 \Delta_h \vro  \|_{L^2}^{\f12}  \| \p_3^2 u_3\|_{L^2}^{\f14} \| \nabla_h \p_3^2 u_3\|_{L^2}^{\f12} \| \nabla_h^2 \p_3^2 u_3\|_{L^2}^{\f14}  \\
    \lesssim &\;  \sqrt{\me(t)} \md(t),
    \end{align*}
    and if  $m \ge 4$, integrating by parts, we have
    \beqq \bal
     I_{1,3} = &\; -\i \p_3^{m-1} \vro \,  (\Delta \vro \, \p_3^{m} u_3 + \p_3 \Delta \vro \, \p_3^{m-1} u_3 )\, dx \\
     \lesssim &\; \| \p_3^{m-1} \vro \|_{L^2}^{\f12} \| \p_2 \p_3^{m-1} \vro \|_{L^2}^{\f12} (\| \p_3^{m} u_3 \|_{L^2}^{\f12} \|\p_3^{m+1} u_3 \|_{L^2}^{\f12} \| \Delta \vro\|_{L^2}^{\f12} \| \p_1 \Delta \vro\|_{L^2}^{\f12} \\
     &\; + \| \p_3^{m-1} u_3 \|_{L^2}^{\f12} \| \p_3^{m} u_3 \|_{L^2}^{\f12} \| \p_3  \Delta \vro\|_{L^2}^{\f12} \| \p_{13} \Delta \vro\|_{L^2}^{\f12}) \\
     \lesssim &\;  \sqrt{\me(t)} \md(t),
        \dal \deqq
        which, together with above estimate, yields that
        \beqq 
  I_{1,3} \lesssim \sqrt{\me(t)} \md(t).
        \deqq
        Now we estimate the term $I_{1,4}$. 
\beqq \bal
I_{1,4} = &\; \i {\rm div} (u\cdot \nabla u) \, \p_3^{m} \vro \, \p_3^{m-1} u_3  \, dx +\i {\rm div} ( \vro \, \nabla \vro) \, \p_3^{m} \vro \, \p_3^{m-1} u_3  \, dx  \\
&\; +\i {\rm div} \Big\{ \f{\vro}{1+\vro} (\p_1^2 u+ \p_2^2 u  + \nabla \du + \p_2 b - \nabla b_2) \Big\} \, \p_3^{m} \vro \, \p_3^{m-1} u_3  \, dx  \\
&\;-\i {\rm div} \Big\{ \f{1}{1+\vro} ( b\cdot \nabla b - \f12 \nabla(|b|^2)) \Big\} \, \p_3^{m} \vro \, \p_3^{m-1} u_3  \, dx.
\dal \deqq
We only give the first and second term on the right-hand side, and the other terms can be estimated similarly. 
By Lemma \ref{lemm:sobolev-ie}, we can check 
\beqq \bal 
&\; \i {\rm div} (u\cdot \nabla u) \, \p_3^{m} \vro \, \p_3^{m-1} u_3  \, dx \\
\lesssim &\;  \| \p_3^m \vro\|_{L^2} \| \p_3^{m-1} u_3 \|_{L^2}^{\f14} \| \p_3^{m} u_3 \|_{L^2}^{\f14} \| \p_1 \p_3^{m-1} u_3 \|_{L^2}^{\f14} \| \p_1 \p_3^{m} u_3 \|_{L^2}^{\f14} ( \| \nabla \du \|_{L^2}^{\f12} \| \p_2 \nabla \du \|_{L^2}^{\f12} \| u \|_{L^{\infty}} \\
&\; + \| \nabla u\|_{L^2}^{\f12} \| \p_2 \nabla u\|_{L^2}^{\f12} \|( \nabla u, \p_3 \nabla u)\|_{L^2}^{\f14} \|\nabla_h( \nabla u, \p_3 \nabla u)\|_{H_{tan}^1}^{\f34}) \\
\lesssim &\; \sqrt{\me(t)} \md(t),
\dal \deqq
and integrating by parts, by Lemma \ref{lemm:sobolev-ie}, we have
\begin{align*}
&\; \i {\rm div} ( \vro \, \nabla \vro) \, \p_3^{m} \vro \, \p_3^{m-1} u_3  \, dx 
=  \i  ( \vro \, \Delta \vro + |\nabla \vro|^2) \, \p_3^{m} \vro \, \p_3^{m-1} u_3  \, dx \\
= &\; \i   \vro \, \Delta \vro  \, \p_3^{m} \vro \, \p_3^{m-1} u_3  \, dx  -  \i \p_3^{m-1} \vro \, ( |\nabla \vro|^2 \, \p_3^{m} u_3  +  2 \nabla \vro\,  \p_3 \nabla \vro \, \p_3^{m} u_3 ) \, dx \\
\lesssim &\;  \| \p_3^m \vro\|_{L^2} \| \p_3^{m-1} u_3 \|_{L^2}^{\f14} \| \p_3^{m} u_3 \|_{L^2}^{\f14} \| \p_1 \p_3^{m-1} u_3 \|_{L^2}^{\f14} \| \p_1 \p_3^{m} u_3 \|_{L^2}^{\f14}  \| \Delta \vro\|_{L^2}^{\f12} \| \p_2 \Delta \vro \|_{L^2}^{\f12} \| \vro\|_{L^{\infty}} \\
&\; +  \| \p_3^{m-1} \vro\|_{L^2}^{\f12}  \| \p_1 \p_3^{m-1} \vro\|_{L^2}^{\f12} \| \nabla \vro\|_{L^2}^{\f14}  \| \p_2 \nabla \vro\|_{L^2}^{\f14} \| \p_3 \nabla \vro\|_{L^2}^{\f14} \| \p_{23} \nabla \vro\|_{L^2}^{\f14}\\
&\; \times ( \| \p_3^{m-1} u_3 \|_{L^2}^{\f14} \| \p_3^{m} u_3 \|_{L^2}^{\f14} \| \p_1 \p_3^{m-1} u_3 \|_{L^2}^{\f14} \| \p_1 \p_3^{m} u_3 \|_{L^2}^{\f14} \| \p_3 \nabla \vro\|_{L^2}^{\f12} \| \p_2 \p_3 \nabla \vro\|_{L^2}^{\f12} \\
&\; + \| \p_3^{m} u_3 \|_{L^2}^{\f12} \| \p_3^{m+1} u_3 \|_{L^2}^{\f12} \| \nabla \vro\|_{L^2}^{\f14} \| \nabla_h \nabla \vro\|_{L^2}^{\f12} \| \nabla_h^2 \nabla \vro\|_{L^2}^{\f14}) \\
\lesssim &\; \sqrt{\me(t)} \md(t).
\end{align*} 
Thus, we can obtain that
\beqq 
 I_{1,4} \lesssim \sqrt{\me(t)} \md(t),
\deqq
which, together with the estimates from $I_{1,1}$
to $I_{1,3}$, yields that
\beq \label{Nlinear1-03}
 I_{1} \lesssim \sqrt{\me(t)} \md(t).
\deq
{\bf{Estimate of the term $I_2$.}} Substituting equation $\eqref{eqr}_1$ into term $I_2$, and integrating by parts, we have
\beqq \bal
I_2 = &\;  \i \du \, \p_3^m \du \, \p_3^{m-1} u_3 \, dx +  \i \du \, \p_3^m ( \vro\, \du ) \, \p_3^{m-1} u_3 \, dx  \\
&\; + \i \du \, \p_3^m (u \cdot \nabla \vro) \, \p_3^{m-1} u_3 \, dx  \\
= &\;  \i \du \, \p_3^m \du \, \p_3^{m-1} u_3 \, dx + \sum_{ 0 \le l \le m} C_{m}^{l}  \i \du \, \p_3^l \vro\, \p_3^{m-l}  \du  \, \p_3^{m-1} u_3 \, dx  \\
&\; - \sum_{ 0 \le l \le m-1} C_{m-1}^{l} \i   \p_3^{l} u \cdot \nabla \p_3^{m-1-l} \vro \, ( \p_3 \du\,  \p_3^{m-1} u_3 + \du \,  \p_3^{m} u_3) \, dx.
\dal \deqq
By Lemma \ref{lemm:sobolev-ie}, we can check 
\begin{align} \label{Nlinear1-04} 
I_2 \lesssim &\; \| \p_3^{m-1} u_3 \|_{L^2}^{\f14} \| \p_3^{m} u_3 \|_{L^2}^{\f14} \| \p_1 \p_3^{m-1} u_3 \|_{L^2}^{\f14} \| \p_1 \p_3^{m} u_3 \|_{L^2}^{\f14} \Big( \| \p_3^m \du\|_{L^2} \| \du \|_{L^2}^{\f12} \| \p_2 \du \|_{L^2}^{\f12} (1+ \| \vro\|_{L^{\infty}}) \notag\\
&\; + \sum_{1\le l \le m}  \| \p_3^l \vro\|_{L^2} \| \p_3^{m-l} \du \|_{L^2}^{\f12} \| \p_2  \p_3^{m-l} \du \|_{L^2}^{\f12} \| \du \|_{L^{\infty}}  + \|  \p_3 \du \|_{L^2}^{\f12} \| \p_2 \p_3 \du \|_{L^2}^{\f12} \notag\\
&\;\times (\| u\|_{L^{\infty}} \| \p_3^{m-1} \nabla \vro \|_{L^2}  + \sum_{1\le l \le m-1} \| \p_3^{l} u \|_{L^2}^{\f14} \| \nabla_h \p_3^{l} u\|_{L^2}^{\f12} \| \nabla_h^2 \p_3^{l} u_3 \|_{L^2}^{\f14} \| \p_3^{m-1-l} \nabla \vro \|_{L^2}^{\f12} \| \p_3^{m-l} \nabla \vro \|_{L^2}^{\f12}) \Big)\notag\\
&\; + \| \p_3^{m} u_3 \|_{L^2}^{\f12} \| \p_1 \p_3^{m-1} u_3 \|_{L^2}^{\f12} \|\du \|_{L^2}^{\f14} \| \nabla_h \du \|_{L^2}^{\f12} \| \nabla_h^2 \du \|_{L^2}^{\f14}\notag\\
&\;\times (\| u\|_{L^{\infty}} \| \p_3^{m-1} \nabla \vro \|_{L^2}  + \sum_{1\le l \le m-1} \| \p_3^{l} u \|_{L^2}^{\f14} \| \nabla_h \p_3^{l} u\|_{L^2}^{\f12} \| \nabla_h^2 \p_3^{l} u_3 \|_{L^2}^{\f14} \| \p_3^{m-1-l} \nabla \vro \|_{L^2}^{\f12} \| \p_3^{m-l} \nabla \vro \|_{L^2}^{\f12}) \notag\\
\lesssim &\; \sqrt{\me(t)} \md(t).
\end{align}
{\bf{Estimate of the term $I_3$.}}
 Due to the lower dissipation of density, $I_3$ cannot be controlled by $\sqrt{\me(t)} \md(t)$ by the anisotropic inequalities \eqref{ie:Sobolev}. Our method is to substitute equation of $\p_3 \vro$ into  term $I_3$, which yields that
\beq \label{Nlinear1-05} \bal
I_3 = &\;- \f{d}{dt} \i \du \, \p_3^m b_2 \, \p_3^{m-1} u_3 \, dx  +   \i  \p_t \du \,  \p_3^m b_2 \, \p_3^{m-1} u_3 \, dx +   \i  \du \, \p_t \p_3^m b_2 \, \p_3^{m-1} u_3 \, dx \\
&\; - \i \du \, |\p_3^m b_2|^2 \, dx 
+ \i \du \, \p_3^m b_2 \, \p_3^{m-1} ( \p_1^2 u_3 + \p_2^2 u_3  + \p_3 \du+\p_2 b_3 + (F_2)_3 ) \, dx \\
:= &\; - \f{d}{dt} \i \du \, \p_3^m b_2 \, \p_3^{m-1} u_3 \, dx +\sum_{i=1}^{4} I_{3,i}.
\dal \deq
Using the equation $\eqref{eqr}_1$ and the condition ${\rm div} b= 0$, we have
\beqq \bal
\p_t \du = &\; {\rm div} \left\{ \f{1}{1+\vro} (\p_1^2 u + \p_2^2 u + \nabla \du + \p_2 b -  \nabla b_2 + b \cdot \nabla b - \f12 \nabla(|b|^2)) - (1+\vro) \nabla \vro - u\cdot \nabla u \right\}  \\
= &\; \nabla( \f{1}{1+\vro} ) \left\{\p_1^2 u + \p_2^2 u + \nabla \du + \p_2 b -  \nabla b_2 + b \cdot \nabla b - \f12 \nabla(|b|^2)\right\} \\
&\; +  \f{1}{1+\vro} \left\{\p_1^2 \du + \p_2^2 \du + \Delta \du - \Delta b_2 + \nabla b \cdot \nabla b - \f12 \Delta (|b|^2) \right\} \\
&\; - \nabla \vro \, \nabla \vro  - (1+\vro) \Delta \vro - \nabla u  \cdot \nabla u - u \cdot \nabla \du.
\dal \deqq
Using the anisotropic type inequalities \eqref{ie:Sobolev},  we can check 
\beqq \bal
\| \p_t \du \|_{L^2} \lesssim \sqrt{\me(t)} + \sqrt{\md(t)}, ~~~
\| \nabla_h \p_t \du \|_{L^2} \lesssim  \sqrt{\md(t)}.
\dal \deqq
Using above estimate and the anisotropic type inequality \eqref{ie:Sobolev}, we have
\begin{align*}
I_{3,1} \lesssim &\; \| \p_3^m b_2 \|_{L^2}^{\f12} \| \p_1 \p_3^m b_2 \|_{L^2}^{\f12} \| \p_3^{m-1} u_3 \|_{L^2}^{\f12} \| \p_3^{m} u_3 \|_{L^2}^{\f12} \| \p_t \du \|_{L^2}^{\f12} \| \p_2 \p_t \du\|_{L^2}^{\f12} \\
\lesssim&\; \sqrt{\me(t)} \md(t),\\
I_{3,4} \lesssim  &\; \| \p_3^m b_2 \|_{L^2}^{\f12} \| \p_1 \p_3^m b_2 \|_{L^2}^{\f12} \| \du\|_{L^2}^{\f14} \| \p_2 \du\|_{L^2}^{\f14} \| \p_3 \du\|_{L^2}^{\f14} \| \p_{23} \du\|_{L^2}^{\f14} \\
&\; \times \|\p_3^{m-1} \{ \p_1^2 u_3 + \p_2^2 u_3 + \p_3 \du+\p_2 b_3+ (F_2)_3 \}\|_{L^2}\\
\lesssim&\; \sqrt{\me(t)} \md(t),
\end{align*}
which we have used the following estimate 
\beq\label{Nlinear1-help-1} \bal
&\; \|\p_3^{m-1} \left\{ \p_1^2 u_3 + \p_2^2 u_3  + \p_3 \du+\p_2 b_3+(F_2)_3 \right\} \|_{L^2} \\
= &\; \|\p_3^{m-1} \Big\{ \f{1}{1+\vro} (\p_1^2 u_3 + \p_2^2 u_3 + \p_3 \du+\p_2 b_3) +\f{\vro}{1+\vro} \p_3 b_2  - u\cdot \nabla u_3 - \vro \, \p_3 \vro \\
&\; + \f{1}{1+\vro} (b \cdot \nabla b_3 - \f12 \p_3 (|b|^2)) \Big\}\|_{L^2}\\
\lesssim &\; \sqrt{\md(t)} + \md(t)^{\f38}.
\dal \deq
It should be noticed that due to the lack of dissipation in the $x_2$ and $x_3$ directions for the magnetic field, the term $I_{3,3}$ still cannot be controlled by $\sqrt{\me(t)} \md(t)$. Here we substitute the equation of $\du$
\beqq
\du = - \p_t \vro - u \cdot \nabla \vro  - \vro \, \du 
\deqq
into this term, thus we have
\beq\label{Nlinear1-06}  \bal
I_{3,3} = &\; \f{d}{dt} \i \vro \, |\p_3^m b_2 |^2\, dx - 2 \i \vro \, \p_3^m b_2\, \p_t \p_3^m b_2 \, dx +  \i (u \cdot \nabla \vro + \vro \, \du) \, |\p_3^m b_2 |^2\, dx  \\
= &\; \f{d}{dt} \i \vro \, |\p_3^m b_2 |^2\, dx - 2 \i  \vro \, \p_3^m b_2\, \p_3^m \p_1^2 b_2  \, dx   - 2 \i  \vro \, \p_3^m b_2\, \p_3^m (-\du + \p_2 u_2 ) \, dx \\
&\; -2 \i \vro \, \p_3^m b_2 \, \p_3^m ( - u \cdot  \nabla b_2 - b_2 \, \du +  b \cdot \nabla u_2) \, dx + \i (u \cdot \nabla \vro + \vro \, \du) \, |\p_3^m b_2 |^2\, dx  \\
= &\; \f{d}{dt} \i \vro \, |\p_3^m b_2 |^2\, dx - 2 \i  \vro \, \p_3^m b_2\, \p_3^m \p_1^2 b_2  \, dx   - 2 \i  \vro \, \p_3^m b_2\, \p_3^m (-\du + \p_2 u_2 ) \, dx \\
&\; + 2 \sum_{ 1 \le  l \le m} C_{m}^{l} \i \vro \, \p_3^m b_2\, \p_3^{l} u \cdot \p_3^{m-l}  \nabla b_2 \, dx
+ 2 \i  \vro \, \p_3^m b_2\, \p_3^m (b_2 \, \du ) \, dx \\
&\;- 2 \i  \vro \, \p_3^m b_2\, b \cdot \p_3^m   \nabla u_2 \, dx - 2\sum_{ 1 \le  l \le m} C_{m}^{l} \i \vro \, \p_3^m b_2 \, \p_3^{l} b \cdot \p_3^{m-l}   \nabla u_2 \, dx   \\
:= &\; \f{d}{dt} \i \vro \, |\p_3^m b_2 |^2\, dx + \sum_{i=1}^{6} I_{3,3,i},
\dal \deq
where we have used the fact that
\beqq 
2 \i  \vro \, \p_3^m b_2\, u \cdot \p_3^m  \nabla b_2 \, dx + \i (u \cdot \nabla \vro + \vro \, \du) \, |\p_3^m b_2 |^2\, dx  =0.
\deqq
Integrating by parts and using the anisotropic type inequalities \eqref{ie:Sobolev}, it holds that
\begin{align*}
I_{3,3,1} = &\; 2 \i \p_3^m \p_1 b_2 \, (\p_1 \vro \, \p_3^{m} b_2 + \vro \, \p_1 \p_3^{m} b_2 ) \, dx  \\ 
\lesssim &\; \| \vro\|_{L^{\infty}} \| \p_1 \p_3^m b_2\|_{L^2}^2    + \| \p_3^{m} b_2\|_{L^2}^{\f12}  \| \p_1 \p_3^{m} b_2\|_{L^2}^{\f32}   \| \p_1 \vro\|_{L^2}^{\f14} \| \p_{13} \vro\|_{L^2}^{\f14}  \| \p_{12} \vro\|_{L^2}^{\f14} \| \p_{23}  \p_1\vro\|_{L^2}^{\f14}  \\
\lesssim &\; \sqrt{\me(t)} \md(t), \\
I_{3,3,3} \lesssim &\;( \sum_{1 \le l\le m-1 } \| \p_3^l u\|_{L^2}^{\f14} \| \p_3^{l+1} u\|_{L^2}^{\f14} \| \p_2 \p_3^l u\|_{L^2}^{\f14} \| \p_2 \p_3^{l+1} u\|_{L^2}^{\f14}   \|\nabla \p_3^{m-l} b_2 \|_{L^2}^{\f12} \|\p_1 \nabla \p_3^{m-l} b_2 \|_{L^2}^{\f12}  \\
&\; +  \| \p_3^m u\|_{L^2}^{\f12} \| \p_2 \p_3^m u\|_{L^2}^{\f12} \|\nabla b_2 \|_{L^2}^{\f14} \|\p_3 \nabla b_2 \|_{L^2}^{\f14}   \|\p_1 \nabla b_2 \|_{L^2}^{\f14} \|\p_{13} \nabla b_2 \|_{L^2}^{\f14} ) \\
&\; \times \| \vro \|_{L^2}^{\f14} \| \p_3 \vro \|_{L^2}^{\f14} \| \p_2 \vro \|_{L^2}^{\f14} \| \p_{23} \vro \|_{L^2}^{\f14} \| \p_3^m b_2\|_{L^2}^{\f12} \| \p_1 \p_3^m b_2\|_{L^2}^{\f12}\\
\lesssim &\; \sqrt{\me(t)} \md(t).
\end{align*}
Similarly, we can check that
\beqq 
I_{3,3,2} + I_{3,3,4} + I_{3,3,6}\lesssim \sqrt{\me(t)} \md(t).
\deqq
Then we deal with the difficult term $I_{3,3,5}$, which gives rise to the loss of vertical derivative. 
Integrating by parts, due to the condition ${\rm div} b=0$, we have
\beq\label{Nlinear1-equ} \bal
I_{3,3,5} = &\; 2 \i \p_3^m u_2 (\vro \, \p_3^m b_2 \, {\rm div} b + \vro \, b \cdot \nabla \p_3^m b_2 + \p_3^m b_2 \,b\cdot \nabla \vro ) \, dx \\
= &\;  2 \i \vro \, \p_3^m u_2 \, b \cdot \nabla \p_3^m b_2 \, dx + 2 \i \p_3^m u_2 \, \p_3^m b_2 \, b\cdot \nabla \vro \, dx.
\dal \deq
In order to solve this problem, we exploit the cancellation mechanism. More precisely, first of all, from the equation of $\eqref{eqr}_2$, we have
\beqq \bal
\p_t u_2 = &\;  \f{1}{1+\vro} b \cdot \nabla b_2  +\f{1}{1+\vro} (\p_1^2 u_2 + \p_2^2 u_2  + \p_2 \du )-  u \cdot \nabla u_2 - ( 1+ \vro) \p_2 \vro-  \f{1}{2(1+\vro)} \p_2 (| b|^2).
\dal \deqq
Applying the $\p_3^m$-operator to the above equality, and multiplying by $(1+ \vro) \vro \, \p_3^m u_2$, we have
\begin{align*}
&\; \i \p_t \p_3^m u_2 \, (1+ \vro) \vro \, \p_3^m u_2 \, dx \\
= &\; \i \vro \, \p_3^m u_2 \, b \cdot \nabla \p_3^m b_2 \, dx  + \i \p_3^m \Big\{\f{1}{1+\vro} (\p_1^2 u_2 + \p_2^2 u_2  + \p_2 \du ) \Big\} \, (1+ \vro) \vro \, \p_3^m u_2 \, dx \\
&\; - \i \p_3^m ( u \cdot \nabla u_2 )  \, (1+ \vro) \vro \, \p_3^m u_2 \, dx- \i \p_3^m \Big\{ ( 1 +\vro ) \p_2 \vro +    \f{1}{2(1+\vro)} \p_2 (| b|^2) \Big\} \, (1+ \vro) \vro \, \p_3^m u_2 \, dx  \\
&\; + \i \Big \{\p_3^m (\f{1}{1+\vro} (b \cdot \nabla b_2) -\f{1}{1+\vro} b \cdot \p_3^m \nabla b_2 \Big\}  \, (1+ \vro) \vro \, \p_3^m u_2 \, dx.
\end{align*}
Subtracting above equality to \eqref{Nlinear1-equ} yields that
\begin{align*}
I_{3,3,5} = &\;   \f{d}{dt}  \i  (1+ \vro) \vro \,|\p_3^m u_2|^2  \, dx -   \i  \,|\p_3^m u_2|^2  \p_t \vro \, (1+ 2\vro) \, dx   + 2 \i \p_3^m u_2 \, \p_3^m b_2 \, b\cdot \nabla \vro \, dx\\
&\;- 2 \i \p_3^m \Big\{\f{1}{1+\vro} (\p_1^2 u_2 + \p_2^2 u_2  + \p_2 \du ) \Big\} \, (1+ \vro) \vro \, \p_3^m u_2 \, dx \\
&\;+  \i  (1+ \vro) \vro \,  u \cdot \nabla |\p_3^m u_2 |^2\,  dx 
+ 2 \i \Big\{ \p_3^m ( u \cdot \nabla u_2 ) -  u \cdot \nabla \p_3^m u_2 \Big\} \, (1+ \vro) \vro \, \p_3^m u_2 \, dx\\
&\; + 2 \i (1+ \vro)^2 \vro \, \p_3^m \p_2 \vro \, \p_3^m u_2 \, dx + 2\i \Big\{ \p_3^m  (( 1 +\vro ) \p_2 \vro )- ( 1 +\vro ) \p_3^m \p_2 \vro \Big\}  \, (1+ \vro) \vro \, \p_3^m u_2 \, dx  \\
&\; +  \i \p_3^m \Big\{  \f{1}{(1+\vro)} \p_2 (| b|^2) \Big\}  \, (1+ \vro) \vro \, \p_3^m u_2 \, dx\\
&\; - 2 \i \Big \{\p_3^m (\f{1}{1+\vro} b \cdot \nabla b_2) -\f{1}{1+\vro} b \cdot \p_3^m \nabla b_2 \Big\}  \, (1+ \vro) \vro \, \p_3^m u_2 \, dx .
\end{align*}
 We focus on the estimate of seventh term on the right-hand side, i.e. $\i (1+ \vro)^2 \, \vro \,  \p_3^m \p_2 \vro \, \p_3^m u_2 \, dx $, since the  other nonlinear terms can be controlled as expected by using the anisotropic type inequalities \eqref{ie:Sobolev}.
 From the equation $\eqref{eqr}_2$, we have
\beqq \bal
 (1+\vro) \p_3 \vro = &\; -\p_t u_3 + \f{1}{1+\vro} ( \p_1^2 u_3 + \p_2^2 u_3  + \p_3 \du + \p_2 b_3 -\p_3 b_2 ) \\
&\; +\f{1}{1+\vro} (b \cdot \nabla b_3 - \f12 \p_3(|b|^2) ) - u \cdot \nabla u_3.
\dal \deqq
Thus, we have
  \beq\label{Nlinear1-07} \bal
&\;   \i (1+ \vro)^2 \, \vro \,  \p_3^m \p_2 \vro \, \p_3^m u_2 \, dx   \\
= &\; - \f{d}{dt}  \i (1+ \vro) \, \vro \,  \p_3^{m-1} \p_2 u_3 \, \p_3^m  u_2 \, dx + \i (1+ 2\vro) \p_t \vro \,  \p_3^{m-1} \p_2 u_3 \, \p_3^m  u_2 \, dx  \\
&\; +  \i (1+ \vro) \, \vro \,  \p_3^{m-1} \p_2 u_3 \, \p_3^{m} \p_t u_2 \, dx  -  \i \p_3^{m-1} \p_2 (  u \cdot \nabla u_3) \, (1+ \vro) \, \vro \,  \p_3^m  u_2 \, dx\\
&\; +  \i \p_3^{m-1} \p_2 \Big\{  \f{1}{1+\vro} ( \p_1^2 u_3 + \p_2^2 u_3  + \p_3 \du + \p_2 b_3 -\p_3 b_2 )  \Big\} \, (1+ \vro) \, \vro \,  \p_3^m  u_2 \, dx \\
&\; +  \i \p_3^{m-1} \p_2 \Big\{  \f{1}{1+\vro}  (b \cdot \nabla b_3 - \f12 \p_3(|b|^2))  \Big\} \, (1+ \vro) \, \vro \,  \p_3^m  u_2 \, dx \\
&\; -\i  \Big\{ \p_3^{m-1}\p_2 ( (1+ \vro)\p_3 \vro) -  (1+ \vro)\p_3^{m} \p_2 \vro \Big\} \, (1+ \vro) \, \vro \,  \p_3^m  u_2 \, dx \\
:= &\; - \f{d}{dt}  \i (1+ \vro) \, \vro \,  \p_3^{m-1} \p_2 u_3 \, \p_3^m  u_2 \, dx  + \sum_{i=1}^{6} J_{i}. 
\dal \deq
Using the equations $\eqref{eqr}_1$-$\eqref{eqr}_2$ and the anisotropic inequalities \eqref{ie:Sobolev}, integrating by parts, we can check
\begin{align*}
    J_{1} = &\; -\i (1+ 2\vro) (\du +\vro \,\du + u\cdot \nabla \vro) \,  \p_3^{m-1} \p_2 u_3 \, \p_3^m  u_2 \, dx  \\
    \lesssim &\; (1+\| \vro\|_{L^{\infty}} ) \| \p_3^m u_2\|_{L^2}^{\f12} \| \p_1 \p_3^m u_2\|_{L^2}^{\f12} \|  \p_2 \p_3^{m-1} u_2\|_{L^2}^{\f12} \| \p_2 \p_3^m u_2\|_{L^2}^{\f12} \\
    &\; \times \Big\{  (1+\| \vro\|_{L^{\infty}} ) \| \du\|_{L^2}^{\f12} \|\p_2 \du\|_{L^2}^{\f12} + \| u\|_{L^{\infty}} \| \nabla \vro\|_{L^2}^{\f12} \|\p_2 \nabla \vro\|_{L^2}^{\f12} \Big\}\\
    \lesssim &\; \sqrt{\me(t)} \md(t),\\
    J_{2} = &\; - \i \p_3 \Big\{ (1+ \vro) \, \vro \, \p_3^{m-1} \p_2 u_3 \Big\} \, \p_3^{m-1} \p_t u_2 \, dx \\
    \lesssim &\;  (1+ \|\vro\|_{L^{\infty}}) (\| \p_3 \vro\|_{L^2}^{\f14} \| \nabla_h \p_3 \vro\|_{L^2}^{\f12}   \| \nabla_h^2 \p_3 \vro\|_{L^2}^{\f14} \| \p_3^{m-1} \p_2 u_3  \|_{L^2}^{\f12} \| \p_3^{m} \p_2 u_3  \|_{L^2}^{\f12} \\
    &\; + \| \vro\|_{L^{\infty}} \| \p_3^{m} \p_2 u_3  \|_{L^2} ) \| \p_3^{m-1} \p_t u_2\|_{L^2} \\
    \lesssim &\; \sqrt{\me(t)} \md(t),
\end{align*}
where we have used
       \beq\label{Nlinear1-help-2} \bal
       \| \p_t u_h \|_{H^{m-1}} \lesssim &\; \|\p_1^2 u_h + \p_2^2 u_h + \nabla_h \du+\p_2 b_h  - \nabla_h b_2 +(F_2)_h \|_{H^{m-1} } \\ \lesssim &\; \| \f{1}{1+\vro} (\p_1^2 u_h + \p_2^2 u_h + \nabla_h \du -\nabla_h \vro +\p_2 b_h  - \nabla_h b_2 ) \|_{H^{m-1} } \\
       &\; +  \| u\cdot \nabla u_h\|_{H^{m-1}} +\| (1+\vro) \nabla_h \vro\|_{H^{m-1}} + \| \f{1}{1+\vro} ( b \cdot \nabla b_h - \f12 \nabla_h(|b|^2) )\|_{H^{m-1}} \\
        \lesssim &\; \sqrt{\md(t)}.
       \dal\deq
Integrating by parts and using the anisotropic inequalities \eqref{ie:Sobolev},  we can check
\begin{align*}
J_{3} = &\;  \sum_{0 \le l \le m-1} C_{m-1}^{l} \i (\p_3^l  u \cdot \nabla \p_3^{m-l-1} \p_2 u_3 + \p_3^l \p_2 u \cdot \nabla \p_3^{m-l-1} u_3) \, (1+ \vro) \, \vro \,  \p_3^m  u_2 \, dx\\
\lesssim &\; (1+ \| \vro \|_{L^{\infty}}) \| \vro\|_{L^{\infty}} \|  \p_3^m  u_2 \|_{L^2}^{\f12}\|  \p_1 \p_3^m  u_2 \|_{L^2}^{\f12} \Big(\| u\|_{L^2}^{\f14}  \| \p_3 u\|_{L^2}^{\f14} \| \p_2 u\|_{L^2}^{\f14} \| \p_{23} u\|_{L^2}^{\f14} \| \nabla \p_3^{m-1} \p_2 u_3\|_{L^2} \\
&\; + \sum_{1 \le l \le m-1} \|\p_3^{l} u\|_{L^2}^{\f12} \|\p_2 \p_3^{l} u\|_{L^2}^{\f12}  \|\nabla \p_3^{m-l-1} \p_2  u\|_{L^2}^{\f12} \|\nabla \p_3^{m-l} \p_2  u\|_{L^2}^{\f12}\\
 &\; + \sum_{0 \le l \le m-1} \|\p_3^{l} \p_2 u\|_{L^2}^{\f12} \|\p_3^{l+1} \p_2 u\|_{L^2}^{\f12}  \|\nabla \p_3^{m-l-1}  u\|_{L^2}^{\f12} \|\p_2 \nabla \p_3^{m-l}  u\|_{L^2}^{\f12} \Big)\\
 \lesssim &\; \sqrt{\me(t)} \md(t),\\
 J_{4} = &\; - \i   \p_3^{m-1}  \Big\{  \f{1}{1+\vro} (\p_1^2 u_3 + \p_2^2 u_3  + \p_3 \du + \p_2 b_3 - \p_3 b_2)  \Big\} \, \p_2 \Big\{ (1+ \vro) \, \vro \, \p_3^{m}  u_2 \Big\} \, dx \\
 = &\; - \sum_{0\le l\le m -1} C_{m-1}^{l} \i \p_3^{l} (\f{1}{1+\vro} ) \, \p_3^{m-l-1} (\p_1^2 u_3 + \p_2^2 u_3  + \p_3 \du + \p_2 b_3) \, \p_2 \Big\{ (1+ \vro) \, \vro \, \p_3^{m}  u_2 \Big\} \, dx\\
 &\; + \sum_{0\le l\le m -1} C_{m-1}^{l} \i \p_3^{l} (\f{1}{1+\vro} ) \, \p_3^{m-l} b_2 \, \p_2 \Big\{ (1+ \vro) \, \vro \, \p_3^{m}  u_2 \Big\} \, dx\\
 \lesssim &\; (1+ \| \vro\|_{L^{\infty}}) (  \| \vro\|_{L^{\infty}} \| \p_3^m \p_2 u\|_{L^2} + \| \p_2 \vro\|_{H_{tan}^1}^{\f12} \| \p_2 \p_3 \vro\|_{H_{tan}^1}^{\f12}  \| \p_3^m u\|_{L^2}^{\f12} \| \p_3^m \p_1 u\|_{L^2}^{\f12} ) \\
 &\; \times \|  \p_3^{m-1} (\f{1}{1+\vro} (\p_1^2 u_3 + \p_2^2 u_3  + \p_3 \du + \p_2 b_3))  \|_{L^2} \\
 &\;+ (1+ \| \vro\|_{L^{\infty}}) (  \| \vro\|_{L^{2}}^{\f14} \| \p_3 \vro\|_{L^2}^{\f14}  \| \p_2 \vro\|_{L^2}^{\f14}  \| \p_{23} \vro\|_{L^2}^{\f14}  \| \p_3^m \p_2 u\|_{L^2} + \| \p_2 \vro\|_{L^2}^{\f12} \| \p_2 \p_3 \vro\|_{L^2}^{\f12}  \| \p_3^m u\|_{L^2}^{\f12} \| \p_3^m \p_1 u\|_{L^2}^{\f12} ) \\
 &\; \times \|\p_3^{m-1} (\f{1}{1+\vro} \p_3 b_2)\|_{L^2}^{\f12} \|\p_1 \p_3^{m-1} (\f{1}{1+\vro} \p_3 b_2)\|_{L^2}^{\f12}\\
 \lesssim &\; \sqrt{\me(t)} \md(t).
\end{align*}
Similarly, we can check that
\beqq
 J_{5} +  J_{6}  \lesssim \sqrt{\me(t)} \md(t).
\deqq
Substituting the estimates from $J_{1}$ to $J_{6}$ into \eqref{Nlinear1-07}, we have
\beqq \bal
2 \i  (1+ \vro)^2 \vro \, \p_3^m  \p_2 \vro \,\p_3^m u_2 \, dx 
\le  - 2 \f{d}{dt}  \i (1+ \vro) \, \vro \,  \p_3^{m-1} \p_2 u_3 \, \p_3^m  u_2 \, dx  + C \sqrt{\me(t)} \md(t),
\dal \deqq
which yields that
\beqq \bal
I_{3,3,5} \le \f{d}{dt}  \i  (1+ \vro) \vro \,|\p_3^m u_2|^2  \, dx - 2 \f{d}{dt}  \i (1+ \vro) \, \vro \,  \p_3^{m-1} \p_2 u_3 \, \p_3^m  u_2 \, dx + C \sqrt{\me(t)} \md(t).
\dal \deqq
Substituting  the estimates from $I_{3,3,1}$ to $I_{3,3,6}$ into the equality \eqref{Nlinear1-06}, it holds that
\beqq \bal
I_{3,3} \le &\; \f{d}{dt} \i \vro \, |\p_3^m b_2 |^2\, dx+ \f{d}{dt}  \i  (1+ \vro) \vro \,|\p_3^m u_2|^2  \, dx \\
&\;- 2 \f{d}{dt}  \i (1+ \vro) \, \vro \,  \p_3^{m-1} \p_2 u_3 \, \p_3^m  u_2 \, dx  + C \sqrt{\me(t)} \md(t).
\dal \deqq
Finally, using the equation $\eqref{eqr}_3$ and integrating by parts, due to  the anisotropic type inequalities \eqref{ie:Sobolev},  we can check that
\begin{align*}
I_{3,2} = &\;  \i  \du \, \p_3^{m-1} u_3 \, \p_3^m \p_1^2 b_2 \, dx   +  \i  \du \, \p_3^{m-1} u_3 \, \p_3^m (-\du + \p_2 u_2 ) \, dx \\
&\; + \i \du \, \p_3^{m-1} u_3 \, \p_3^m ( - u \cdot  \nabla b_2 - b_2 \, \du +  b \cdot \nabla u_2) \, dx \\
= &\; - \i   \p_3^m \p_1 b_2 \,(\p_1 \du \, \p_3^{m-1} u_3 + \du \,  \p_1 \p_3^{m-1} u_3) \, dx   +  \i  \du \, \p_3^{m-1} u_3 \, \p_3^m (-\du + \p_2 u_2 ) \, dx \\
&\; + \i {\rm div} (\du \, \p_3^{m-1} u_3 \,  u)  \,\p_3^m  b_2  \, dx - \i {\rm div} (\du \, \p_3^{m-1} u_3 \,  b) \,  \p_3^m  u_2  \, dx \\
&\; + \sum_{1\le l \le m} C_{m}^{l}\i \du \, \p_3^{m-1} u_3 \, (-\p_3^l  u \cdot  \nabla  \p_3^{m-l} b_2 + \p_3^l b \cdot \p_3^{m-l} \nabla u_2) \, dx  \\
&\; - \i \du \, \p_3^{m-1} u_3 \, \p_3^m (  b_2 \, \du ) \, dx \\
\lesssim &\; \sqrt{\me(t)} \md(t).
\end{align*}
Substituting  the estimates from $I_{3,1}$ to $I_{3,4}$ into \eqref{Nlinear1-05}, it holds
\beq\label{Nlinear1-08} \bal
I_{3} \le &\; - \f{d}{dt} \i \du \, \p_3^m b_2 \, \p_3^{m-1} u_3 \, dx +\f{d}{dt} \i \vro \, |\p_3^m b_2 |^2\, dx+ \f{d}{dt}  \i  (1+ \vro) \vro \,|\p_3^m u_2|^2  \, dx \\
&\;  - 2 \f{d}{dt}  \i (1+ \vro) \, \vro \,  \p_3^{m-1} \p_2 u_3 \, \p_3^m  u_2 \, dx + C \sqrt{\me(t)} \md(t).
\dal \deq
{\bf Estimate of the term $I_4$.}
Substituting the equation of $\du$ into this term, we have  
\beqq \bal
I_4
       = &\; -\i \vro \, |\p_3^{m} \vro|^2 \,  ( \p_t \vro + u\cdot \nabla \vro + \vro \, \du) \, dx \\
        = &\; -\f12 \f{d}{dt} \i \vro^2 \, |\p_3^{m} \vro|^2 \, dx - \i \vro^2 \, \p_3^m \vro \, \p_3^{m}( \du +  u\cdot \nabla \vro + \vro \, \du)  \, dx  \\
        &\; -\i \vro \, |\p_3^{m} \vro|^2 \,  ( u\cdot \nabla \vro + \vro \, \du) \, dx \\
        = &\;-\f12 \f{d}{dt} \i \vro^2 \, |\p_3^{m} \vro|^2 \, dx - \i \vro^2 \, \p_3^m \vro \, \p_3^{m} \du  \, dx - \sum_{1\le l \le m} C_{m}^{l} \i \vro^2 \, \p_3^{m} \vro \, \p_3^{l} u \cdot \nabla \p_3^{m-l} \vro \, dx  \\
        &\;  - \sum_{0\le l \le m-1} C_{m}^{l} \i \vro^2 \, \p_3^m \vro \, \p_3^{l} \vro \, \p_3^{m-l} \du  \, dx - \f32 \i \vro^2 \, |\p_3^{m} \vro|^2  \, \du \, dx\\
        := &\; -\f12 \f{d}{dt} \i \vro^2 \, |\p_3^{m} \vro|^2 \, dx +\sum_{i=1}^{4} I_{4,i},
        \dal \deqq
        where we have used the following basic fact that
        \beqq \bal
 &\;  \i \vro^2 \, \p_3^{m} \vro\, u \cdot \nabla \p_3^{m} \vro \, dx + \i \vro^2 \, |\p_3^{m} \vro|^2 \, \du \, dx + \i \vro \, |\p_3^{m} \vro|^2 \,  ( u\cdot \nabla \vro + \vro \, \du) \, dx\\
= &\;- \f12 \i |\p_3^{m} \vro|^2 ( u\cdot \nabla (\vro^2) + \vro^2 \, \du )\, dx + \i \vro^2 \, |\p_3^{m} \vro|^2 \, \du \, dx + \i \vro \, |\p_3^{m} \vro|^2 \,  ( u\cdot \nabla \vro + \vro \, \du) \, dx\\
= &\; \f32 \i \vro^2 \, |\p_3^{m} \vro|^2  \, \du \, dx.
     \dal   \deqq
Using the anisotropic type inequality \eqref{ie:Sobolev}, we can check that
\begin{align*}
I_{4,2} \lesssim &\; \| \p_3^m \vro\|_{L^2} \| \vro\|_{L^{\infty}}^2 (\| \p_3^m \vro\|_{L^2} \|\p_3 u_3\|_{L^{\infty}} + \| \nabla_h \p_3^{m-1} \vro\|_{L^2} \|\p_3 u_h\|_{L^{\infty}} ) \\
&\; + \sum_{ 2 \le l \le m-1}\| \p_3^m \vro\|_{L^2} \| \vro\|_{L^{\infty}}^2 \|\p_3^l u\|_{L^2}^{\f14} \| \p_1 \p_3^l u\|_{L^2}^{\f14} \| \p_3^{l+1} u\|_{L^2}^{\f14} \| \p_1 \p_3^{l+1} u\|_{L^2}^{\f14} \| \nabla \p_3^{m-l} \vro\|_{L^2}^{\f12} \| \p_2 \nabla \p_3^{m-l} \vro\|_{L^2}^{\f12} \\
&\; + \| \p_3^m \vro\|_{L^2} \| \vro\|_{L^{\infty}}^2 (\|\p_3 \vro\|_{L^2}^{\f14} \| \nabla_h \p_3 \vro\|_{L^2}^{\f12} \| \nabla_h^2 \p_3 \vro\|_{L^2}^{\f14} \| \p_3^m u_3\|_{L^2}^{\f12} \| \p_3^{m+1} u_3\|_{L^2}^{\f12}\\ 
&\; + \|\nabla_h \vro\|_{L^2}^{\f14} \| \p_3 \nabla_h \vro\|_{L^2}^{\f14} \|\p_2 \nabla_h  \vro\|_{L^2}^{\f14} \|\p_{23} \nabla_h \vro\|_{L^2}^{\f14} \| \p_3^m u_h\|_{L^2}^{\f12} \| \p_1 \p_3^{m} u_h\|_{L^2}^{\f12})\\
  \lesssim &\; \sqrt{\me(t)} \md(t),
\end{align*}
where we have used the following estimate
\beqq \bal
\| \vro\|_{L^{\infty}} \lesssim &\; \|(\vro, \p_3 \vro)\|_{L^2}^{\f14} \|\nabla_h (\vro, \p_3 \vro)\|_{H_{tan}^1}^{\f34}, \\
\| \p_3 u\|_{L^{\infty}} \lesssim &\; \|( \p_3 u, \p_3^2 u)\|_{L^2}^{\f14} \|\nabla_h (\p_3 u, \p_3^2 u)\|_{H_{tan}^1}^{\f34}.
\dal \deqq
Similarly, we can check
\beqq 
I_{4,1} + I_{4,3} + I_{4,4}  \lesssim   \sqrt{\me(t)} \md(t),
\deqq
which, together with above estimate, yields that
\beq\label{Nlinear1-09}
I_4 \le -\f12\f{d}{dt} \i \vro^2 \, |\p_3^{m} \vro|^2 \, dx + C \sqrt{\me(t)} \md(t).
\deq
{\bf Estimate of the terms $I_5$-$I_7$}.
Using the anisotropic type inequality \eqref{ie:Sobolev},  we obtain that
\beq\label{Nlinear1-10} \bal
I_5 = &\; \sum_{ 1 \le l \le m-1} C_{m-1}^{l} \i \du \, \p_3^m \vro \,  \p_3^{l} \vro \, \p_3^{m-l} \vro \, dx \\
\lesssim &\; \| \du\|_{L^{\infty}} \| \p_3^m \vro\|_{L^2} (\| \p_3 \vro\|_{L^2}^{\f14} \| \p_1 \p_3 \vro\|_{L^2}^{\f14} \| \p_3^{2} \vro\|_{L^2}^{\f14} \|\p_1 \p_3^{2} \vro\|_{L^2}^{\f14} \| \p_3^{m-1} \vro\|_{L^2}^{\f12} \| \p_2 \p_3^{m-1} \vro\|_{L^2}^{\f12}\\
&\; +  \sum_{ 2 \le l \le m-1}  \| \p_3^l \vro\|_{L^2}^{\f12} \| \p_1 \p_3^l \vro\|_{L^2}^{\f12} \| \p_3^{m-l} \vro\|_{L^2}^{\f14} \| \p_2 \p_3^{m-l} \vro\|_{L^2}^{\f14} \| \p_3^{m-l+1} \vro\|_{L^2}^{\f14}\|\p_2 \p_3^{m-l+1} \vro\|_{L^2}^{\f14} )\\
\lesssim &\; \sqrt{\me(t)} \md(t).
\dal \deq
Similarly, we can obtain that
\beqq
I_6+I_7 \lesssim  \sqrt{\me(t)} \md(t).
\deqq 
Substituting above estimate, \eqref{Nlinear1-03}, \eqref{Nlinear1-04}, \eqref{Nlinear1-08}, \eqref{Nlinear1-09} and \eqref{Nlinear1-10} into \eqref{Nlinear1-02}, we have
\beqq \bal
 -\i \du \, |\p_3^m \vro|^2 \, dx 
\le  &\; \f{d}{dt} \i \du \, \p_3^{m} \vro \, \p_3^{m-1} u_3 \, dx   - \f{d}{dt} \i \du \, \p_3^m b_2 \, \p_3^{m-1} u_3 \, dx +\f{d}{dt} \i \vro \, |\p_3^m b_2 |^2\, dx\\
&\; + \f{d}{dt}  \i  (1+ \vro) \vro \,|\p_3^m u_2|^2  \, dx   - 2 \f{d}{dt}  \i (1+ \vro) \, \vro \,  \p_3^{m-1} \p_2 u_3 \, \p_3^m  u_2 \, dx  \\
&\; - \f12\f{d}{dt} \i \vro^2 \, |\p_3^{m} \vro|^2 \, dx + \sqrt{\me(t)} \md(t).
\dal \deqq
Therefore, we finish the proof of the estimate 
\eqref{Nlinear1-01-1}.
Finally we give the proof of \eqref{Nlinear1-01-2}. Notice that we have established the estimate of $- \i \du \, |\p_3^m b_2|^2 \, dx$, see the term $I_{3,3}$ in detail.
Then, similar to the estimate of $I_{3,3}$, we can check
\beqq \bal
- \i \du \, |\p_3^m b_1|^2 \, dx \le &\; \f{d}{dt} \i \vro \, |\p_3^m b_1 |^2\, dx+ \f{d}{dt}  \i  (1+ \vro) \vro \,|\p_3^m u_1|^2  \, dx \\
&\; - 2 \f{d}{dt}  \i (1+ \vro) \, \vro \,  \p_3^{m-1} \p_1 u_3 \, \p_3^m  u_1 \, dx   + C \sqrt{\me(t)} \md(t),
\dal \deqq
which, together with the estimate of $I_{3,3}$, yields the estimate \eqref{Nlinear1-01-2}. Thus we finish the proof of this lemma.
\end{proof}
Next, we estimate another difficult term involving the highest-order vertical derivatives of density.
\begin{lemm}\label{lemm:Nlinear2}
     Under the assumption \eqref{assumption}, for any smooth solution $(\vro ,u, b)$ of equation \eqref{eqr}, it holds
     \beq\label{Nlinear2-01-1} \bal
       \i \nabla_h \cdot u_h \, |\p_3^m \vro|^2 \, dx 
 \le &\;   \f{d}{dt} \i (1+\vro) \nabla_h \cdot u_h \, \p_3^{m} \vro \, \p_3^{m-1} u_3 \, dx +\f{d}{dt} \i \nabla_h \cdot u_h \, \p_3^m b_2 \, \p_3^{m-1} u_3 \, dx \\
&\; + \f{d}{dt} \i (b_2 -\vro)\,  |\p_3^m b_2|^2 \, dx  + \f{d}{dt}  \i  (1+ \vro) (b_2 -\vro) \,|\p_3^m u_2|^2  \, dx \\
&\;  - 2 \f{d}{dt}  \i (1+ \vro) \, (b_2 - \vro) \,  \p_3^{m-1} \p_2  u_3 \, \p_3^m u_2 \, dx  + C \sqrt{\me(t)} \md(t).
     \dal \deq
Furthermore, we also have
     \beq\label{Nlinear2-01-2} \bal
 \i \p_2 u_2 \,  |\p_3^m b_2|^2 \, dx  
\le &\;  \f{d}{dt} \i (b_2 -\vro)\,  |\p_3^m b_2|^2 \, dx + \f{d}{dt}  \i  (1+ \vro) (b_2 -\vro) \,|\p_3^m u_2|^2  \, dx \\
&\; - 2 \f{d}{dt}  \i (1+ \vro) \, (b_2 - \vro) \,  \p_3^{m-1} \p_2 u_3 \, \p_3^m  u_2 \, dx + C \sqrt{\me(t)} \md(t),
\dal \deq
and
     \beq\label{Nlinear2-01-3} \bal
 \i \p_2 u_2 \,  |\p_3^m b_1|^2 \, dx 
\le &\;  \f{d}{dt} \i (b_2 -\vro)\,  |\p_3^m b_1|^2 \, dx + \f{d}{dt}  \i  (1+ \vro) (b_2 -\vro) \,|\p_3^m u_1|^2  \, dx \\
&\; - 2 \f{d}{dt}  \i (1+ \vro) \, (b_2 - \vro) \,  \p_3^{m-1} \p_1 u_3 \, \p_3^m  u_1 \, dx  + C \sqrt{\me(t)} \md(t).
\dal \deq
     \end{lemm}
     \begin{proof}
         Similar to the equality \eqref{Nlinear1-02}, we have
         \beq \label{Nlinear2-02} \bal
         &\; \i \nabla_h \cdot u_h \, |\p_3^m \vro|^2 \, dx \\
         = &\; -\f{d}{dt} \i \nabla_h \cdot u_h \, \p_3^{m} \vro \, \p_3^{m-1} u_3 \, dx   + \i \p_t \nabla_h \cdot u_h \, \p_3^{m} \vro \, \p_3^{m-1} u_3  \, dx +  \i \nabla_h \cdot u_h \, \p_3^{m} \p_t \vro \, \p_3^{m-1} u_3  \, dx   \\
         &\; - \i \nabla_h \cdot u_h \, \p_3^{m} \vro \, \p_3^{m} b_2 \,dx  - \i 
         \vro \, \nabla_h \cdot u_h \, |\p_3^{m} \vro|^2 \, dx - \i \nabla_h \cdot u_h \, \p_3^{m} \vro \, \{\p_3^{m-1} (\vro \,\p_3 \vro) - \vro \, \p_3^m \vro\} \, dx \\
         &\;+ \i \nabla_h \cdot u_h \, \p_3^{m} \vro \, \p_3^{m-1} \Big\{\f{1}{1+\vro} (\p_1^2 u_3 + \p_2^2 u_3+ \p_3 \du+\p_2 b_3) \Big\} \, dx \\
         &\; + \i \nabla_h \cdot u_h \, \p_3^{m} \vro \, \p_3^{m-1} \Big\{\f{1}{1+\vro} (b \cdot \nabla b_3 - \f12 \p_3 (|b|^2))+\f{\vro}{1+\vro} \p_3 b_2- u\cdot \nabla u_3\Big\}\, dx\\
        := &\;   -\f{d}{dt} \i \nabla_h \cdot u_h \, \p_3^{m} \vro \, \p_3^{m-1} u_3 \, dx 
        +\sum_{i=1}^{7} II_{i}.
         \dal \deq
         {\bf Estimate of the term $II_1$.} 
         Using the estimate \eqref{Nlinear1-help-2}, we can immediately get that 
         \beq\label{Nlinear2-03} \bal
         II_1 
         \lesssim &\;  \|\nabla_h \cdot \p_t u_h \|_{L^2}^{\f12} \|\p_1 \nabla_h \cdot \p_t u_h \|_{L^2}^{\f12} \|  \p_3^{m} \vro\|_{L^2} \| \p_3^{m-1} u_3 \|_{L^2}^{\f14} \\
          &\; \times \| \p_2 \p_3^{m-1} u_3 \|_{L^2}^{\f14} \| \p_3^{m} u_3 \|_{L^2}^{\f14} \| \p_2 \p_3^{m} u_3 \|_{L^2}^{\f14} \\
     \lesssim &\;  \sqrt{\me(t)} \md(t).
         \dal \deq
         {\bf Estimate of the term $II_2$.}
         Substituting the equation $\eqref{eqr}_1$ into the term $II_2$, integrating by parts, we have
         \beqq \bal
         II_2 = &\; - \i \nabla_h  \cdot  u_h \, \p_3^{m} \du\, \p_3^{m-1} u_3 \, dx  - \i   \nabla_h \cdot u_h \,  \p_3^{m} (u \cdot \nabla \vro + \vro \, \du  )\, \p_3^{m-1} u_3 \, dx \\
          = &\;  -\i \nabla_h  \cdot  u_h \,  \p_3^{m} \du\, \p_3^{m-1} u_3 \, dx  + \i \p_3^{m-1} (u \cdot \nabla \vro ) \, \p_3  (\nabla_h \cdot  u_h \,  \p_3^{m-1} u_3 ) \, dx \\
         &\;   - \i   \nabla_h  \cdot  u_h \, \p_3^{m} ( \vro \, \du  )\, \p_3^{m-1} u_3 \, dx\\
         := &\; \sum_{i=1}^{3} II_{2,i}.
         \dal \deqq
         Using the anisotropic type inequality \eqref{ie:Sobolev}, we can check that
         \begin{align*}
II_{2,2} = &\;  \sum_{ 0 \le l \le m-1 }C_{m-1}^{l} \i \p_3^{l} u \cdot \p_3^{m-1-l} \nabla \vro  \,  (\p_3 \nabla_h \cdot u_h \,  \p_3^{m-1} u_3 + \nabla_h \cdot u_h \,  \p_3^{m} u_3) \, dx \\
\lesssim &\; ( \sum_{1 \le  l \le m-1} \| \p_3^{m-1-l} \nabla \vro\|_{L^2}^{\f12} \| \p_2 \p_3^{m-1-l} \nabla \vro\|_{L^2}^{\f12} \| \p_3^{l} u\|_{L^2}^{\f14} \| \p_3^{l+1} u \|_{L^2}^{\f14} \| \p_1 \p_3^{l} u \|_{L^2}^{\f14} \| \p_1 \p_3^{l+1} u \|_{L^2}^{\f14}  \\
&\; +\| \p_3^{m-1} \nabla \vro\|_{L^2} \| u\|_{L^{\infty}} ) (\| \p_3^{m} u_3 \|_{L^2}^{\f12} \| \p_3^{m+1} u_3 \|_{L^2}^{\f12} \| \nabla_h  u \|_{L^2}^{\f14} \|\nabla_h^2 u\|_{L^2}^{\f12} \|\nabla_h^3 u\|_{L^2}^{\f14}\\
&\; + \| \p_3^{m-1} u_3 \|_{L^2}^{\f14} \|\p_2 \p_3^{m-1} u_3 \|_{L^2}^{\f14} \| \p_3^{m} u_3 \|_{L^2}^{\f14} \| \p_2 \p_3^{m} u_3 \|_{L^2}^{\f14} \| \p_3 \nabla_h \cdot u_h \|_{L^2}^{\f12} \|\p_1  \p_3 \nabla_h \cdot u_h\|_{L^2}^{\f12}) \\
     \lesssim &\;  \sqrt{\me(t)} \md(t).
         \end{align*}
         Similarly, we can check
         \beqq
         II_{2,1} + II_{2,3} \lesssim \sqrt{\me(t)} \md(t),
         \deqq
         which, together with above estimate, yields that
         \beq \label{Nlinear2-04}
         II_2 \lesssim \sqrt{\me(t)} \md(t).
         \deq
         {\bf Estimate of the term $II_3$.}  Notice that $II_3$ cannot be controlled by $\sqrt{\me(t)} \md(t)$ in a similar way to the term $I_3$. Our method here is to substitute the equation of $\p_3 \vro$ into the term $II_3$, which yields that
         \beq \label{Nlinear2-05} \bal
II_3 = &\; \f{d}{dt} \i \nabla_h \cdot u_h \, \p_3^m b_2 \, \p_3^{m-1} u_3 \, dx  -   \i  \p_t \nabla_h \cdot u_h \,  \p_3^m b_2 \, \p_3^{m-1} u_3 \, dx -   \i \nabla_h \cdot u_h \, \p_t \p_3^m b_2 \, \p_3^{m-1} u_3 \, dx \\
&\; + \i \nabla_h \cdot u_h \, |\p_3^m b_2|^2 \, dx 
- \i \nabla_h \cdot u_h \, \p_3^m b_2 \, \p_3^{m-1} ( \p_1^2 u_3 + \p_2^2 u_3 + \p_3 \du+\p_2 b_3 + (F_2)_3 ) \, dx \\
:= &\;  \f{d}{dt} \i \nabla_h \cdot u_h \, \p_3^m b_2 \, \p_3^{m-1} u_3 \, dx +\sum_{i=1}^{4} II_{3,i}.
\dal \deq
Using the estimates \eqref{Nlinear1-help-1} and \eqref{Nlinear1-help-2}, we can check that
\beqq \bal
II_{3,1} \lesssim &\; \| \p_t \nabla_h \cdot u_h \|_{L^2} \| \p_3^m b_2 \|_{L^2}^{\f12} \| \p_1 \p_3^m b_2 \|_{L^2}^{\f12} \| \p_3^{m-1} u_3 \|_{L^2}^{\f14} \| \p_2 \p_3^{m-1} u_3 \|_{L^2}^{\f14} \| \p_3^{m} u_3 \|_{L^2}^{\f14} \| \p_2 \p_3^{m} u_3 \|_{L^2}^{\f14}  \\
\lesssim &\;  \sqrt{\me(t)} \md(t),\\
II_{3,4} \lesssim &\; \| \nabla_h \cdot u_h \|_{L^2}^{\f14} \| \p_3 \nabla_h \cdot u_h \|_{L^2}^{\f14} \| \p_2 \nabla_h \cdot u_h \|_{L^2}^{\f14} \|\p_{23} \nabla_h \cdot u_h \|_{L^2}^{\f14} \| \p_3^m b_2 \|_{L^2}^{\f12} \| \p_1 \p_3^m b_2 \|_{L^2}^{\f12} \\
&\; \times \|\p_3^{m-1} ( \p_1^2 u_3 + \p_2^2 u_3 + \p_3 \du+\p_2 b_3+ (F_2)_3 )\|_{L^2}\\
\lesssim &\; \sqrt{\me(t)} \md(t).
\dal \deqq
Next, we estimate the term $II_{3,3}$. Integrating by parts, we have 
\beq \label{Nlinear2-06} \bal 
II_{3,3} = &\;  - \i u_1 \p_1(|\p_3^m b_2|^2) \, dx + \i \p_2 u_2 \,  |\p_3^m b_2|^2 \, dx \\
\lesssim &\; \|u_1 \|_{L^2}^{\f14} \|\p_3 u_1 \|_{L^2}^{\f14} \|\p_2 u_1 \|_{L^2}^{\f14} \|\p_{23} u_1 \|_{L^2}^{\f14} \|\p_3^m b_2\|_{L^2}^{\f12} \|\p_1 \p_3^m b_2\|_{L^2}^{\f12} \|\p_1 \p_3^m b_2\|_{L^2}  + \i \p_2 u_2 \,  |\p_3^m b_2|^2 \, dx \\
\le &\; \i \p_2 u_2 \,  |\p_3^m b_2|^2 \, dx +  C\sqrt{\me(t)} \md(t).
\dal \deq
Due to the lack of dissipation for the magnetic field in two directions, the first term on the right-hand side can not be controlled by $\sqrt{\me(t)} \md(t)$ as expected. Thus, we substitute the equation of $\p_2 u_2$ 
\beqq
\p_2 u_2 = \p_t b_2 - \p_1^2 b_2 + \du - (F_3)_2
\deqq
into $\i \p_2 u_2 \,  |\p_3^m b_2|^2 \, dx$, which yields that 
\beq \label{Nlinear2-07} \bal 
&\; \i \p_2 u_2 \,  |\p_3^m b_2|^2 \, dx \\
= &\;  \f{d}{dt} \i b_2 \,  |\p_3^m b_2|^2 \, dx - 2 \i b_2 \,  \p_3^m b_2\,  \p_3^m \p_t b_2 dx +  \i \du \,  |\p_3^m b_2|^2 \, dx  - \i ( \p_1^2 b_2  + (F_3)_2 ) \,  |\p_3^m b_2|^2 \, dx\\
= &\;  \f{d}{dt} \i b_2 \,  |\p_3^m b_2|^2 \, dx  - 2 \i b_2 \,  \p_3^m b_2\,  \p_3^m (\p_1^2 b_2  +\p_2 u_2  -\du) dx  +2 \i b_2 \,  \p_3^m b_2\,  \p_3^m ( u \cdot \nabla b_2 ) \, dx \\
&\;+2 \i b_2 \,  \p_3^m b_2\,  \p_3^m ( b_2 \, \du) \, dx
 - 2 \i b_2 \,  \p_3^m b_2\,  \p_3^m (b \cdot \nabla u_2) \, dx
 +  \i \du \,  |\p_3^m b_2|^2 \, dx  \\
&\;- \i ( \p_1^2 b_2  + b \cdot \nabla u_2  ) \,  |\p_3^m b_2|^2 \, dx + \i (  u \cdot \nabla b_2 +b_2 \, \du ) \,  |\p_3^m b_2|^2 \, dx\\
:= &\; \f{d}{dt} \i b_2 \,  |\p_3^m b_2|^2 \, dx  + \sum_{i=1}^{7} K_{i}.
\dal \deq
Integrating by parts, we have
\begin{align*}
K_{1} = &\;  2 \i \p_3^m \p_1 b_2 \, \p_1 ( b_2 \, \p_3^{m} b_2 ) \, dx  - \i b_2 \, \p_3^m  b_2 \, (\p_3^m \p_2 u_2 -\p_3^m \du)\, dx \\
\lesssim &\; \| \p_3^m \p_1 b_2\|_{L^2}^2 \| b_2 \|_{L^\infty} + \|\p_3^m  b_2\|_{L^2}^{\f12} \|\p_3^m \p_1 b_2\|_{L^2}^{\f32} \| \p_1 b_2\|_{H_{tan}^1}^{\f12} \|\p_3  \p_1 b_2\|_{H_{tan}^1}^{\f12} \\
&\; + \|\p_3^m  b_2\|_{L^2}^{\f12} \|\p_1 \p_3^m  b_2\|_{L^2}^{\f12} \| \p_3^m ( \p_2 u_2, \du)\|_{L^2} \|b_2\|_{L^2}^{\f14} \| \p_2 b_2\|_{L^2}^{\f14}  \| \p_3 b_2\|_{L^2}^{\f14} \|\p_{23} b_2\|_{L^2}^{\f14} \\
\lesssim &\; \sqrt{\me(t)} \md(t), \\
K_{6} = &\; 2 \i \p_1  b_2 \, \p_3^{m} b_2  \,\p_3^m \p_1 b_2 \,  dx - \i b \cdot \nabla u_2 \, |\p_3^m  b_2|^2 \, \, dx \\
\lesssim&\;  \|\p_3^m  b_2\|_{L^2}^{\f12} \|\p_3^m \p_1 b_2\|_{L^2}^{\f32} \| \p_1 b_2\|_{H_{tan}^1}^{\f12} \|\p_3  \p_1 b_2\|_{H_{tan}^1}^{\f12}  +  \|\p_3^m  b_2\|_{L^2} \|\p_1 \p_3^m  b_2\|_{L^2} \\
&\; \times \|b\|_{L^2}^{\f14} \| \p_2 b\|_{L^2}^{\f14}  \| \p_3 b\|_{L^2}^{\f14} \|\p_{23} b\|_{L^2}^{\f14}  \|\nabla u_2 \|_{L^2}^{\f14} \| \p_2 \nabla u_2\|_{L^2}^{\f14}  \| \p_3 \nabla u_2\|_{L^2}^{\f14} \|\p_{23} \nabla u_2\|_{L^2}^{\f14}  \\
\lesssim &\; \sqrt{\me(t)} \md(t),
\end{align*}
and
\begin{align*}
    K_{2} + K_{7} = &\; \i b_2 \, u \cdot \nabla |\p_3^m b_2|^2 \, dx + 2 \sum_{1\le l\le m} C_{m}^{l} \i b_2 \, \p_3^m b_2 \, \p_3^{l} u \cdot \p_3^{m-l} \nabla b_2 \, dx \\
    &\; + \i (  u \cdot \nabla b_2 +b_2 \,\du ) \,  |\p_3^m b_2|^2 \, dx\\
    = &\; 2 \sum_{1\le l\le m} C_{m}^{l} \i b_2 \, \p_3^m b_2 \, \p_3^{l} u \cdot \p_3^{m-l} \nabla b_2 \, dx\\
    \lesssim &\; \| \p_3^m b_2 \|_{L^2}^{\f12} \| \p_1 \p_3^m b_2 \|_{L^2}^{\f12} \| b_2\|_{L^2}^{\f14} \| \p_3 b_2\|_{L^2}^{\f14} \| \p_2 b_2\|_{L^2}^{\f14} \| \p_{23} b_2\|_{L^2}^{\f14} \\
    &\; \times (\sum_{1\le l \le m-1} \| \p_3^l u\|_{L^2}^{\f14} \| \p_3^{l+1} u\|_{L^2}^{\f14} \| \p_2 \p_3^l u\|_{L^2}^{\f14} \| \p_2 \p_3^{l+1} u\|_{L^2}^{\f14} \| \p_3^{m-l} \nabla b_2\|_{L^2}^{\f12} \| \p_1 \p_3^{m-l} \nabla b_2\|_{L^2}^{\f12} \\
    &\; + \| \p_3^m u\|_{L^2}^{\f12} \| \p_2 \p_3^m u\|_{L^2}^{\f12} \| \nabla b_2\|_{L^2}^{\f14}  \| \p_3 \nabla b_2\|_{L^2}^{\f14} \|\p_1 \nabla b_2\|_{L^2}^{\f14} \|\p_{13} \nabla b_2\|_{L^2}^{\f14} )\\
    \lesssim &\; \sqrt{\me(t)} \md(t).
\end{align*}
Similarly, we can check
\beqq
K_3 \lesssim \sqrt{\me(t)} \md(t).
\deqq
Similarly to the estimate of $I_{3,3}$, we have
\beqq \bal
K_{5} \le &\; -\f{d}{dt} \i \vro \, |\p_3^m b_2 |^2\, dx- \f{d}{dt}  \i  (1+ \vro) \vro \,|\p_3^m u_2|^2  \, dx \\
&\;+ 2 \f{d}{dt}  \i (1+ \vro) \, \vro \,  \p_3^{m-1}\p_2 u_3 \, \p_3^m  u_2 \, dx   + C \sqrt{\me(t)} \md(t).
\dal \deqq
It remains to estimate the term $K_{4}$. Integrating by parts, due to the condition ${\rm div}b=0$, it holds
\beqq \bal
K_{4} =  &\; - 2 \i b_2 \, \p_3^m b_2 \, b \cdot \nabla \p_3^m u_2 \, dx - 2 \sum_{1 \le l \le m} C_{m}^{l} \i b_2 \, \p_3^m b_2 \, \p_3^{l} b \cdot  \nabla \p_3^{m-l} u_2 \, dx \\
=  &\;  2 \i  b_2 \,  \p_3^m u_2  \, b \, \cdot \nabla \p_3^m b_2 \, dx  +  2 \i \p_3^m u_2 \, \p_3^m b_2 \, b \, \cdot \nabla b_2  \, dx \\
&\; - 2 \sum_{1 \le l \le m} C_{m}^{l} \i b_2 \, \p_3^m b_2 \, \p_3^{l} b \cdot  \nabla \p_3^{m-l} u_2 \, dx \\
:= &\; \sum_{i=1}^{3} K_{4,i}.
\dal \deqq
Similar to the estimate of $K_{2}+K_{7}$, we can check
\beqq \bal 
K_{4,2} + K_{4,3} \lesssim  \sqrt{\me(t)} \md(t).
\dal \deqq
It remains to estimate the term $K_{4,1}$. Similar to the estimate of $I_{3,3,5}$, we can check that
\beqq \bal
K_{4,1} \le &\;  \f{d}{dt}  \i  (1+ \vro) \, b_2 \,|\p_3^m u_2|^2  \, dx - 2 \f{d}{dt}  \i (1+ \vro) \, b_2 \,  \p_3^{m-1} \p_2 u_3 \, \p_3^m  u_2 \, dx   + C \sqrt{\me(t)} \md(t),
\dal \deqq
which, combining the estimates from $K_{4,2}$ to $K_{4,3}$, yields that
\beqq \bal
K_{4} \le &\;  \f{d}{dt}  \i  (1+ \vro) \, b_2 \,|\p_3^m u_2|^2  \, dx - 2 \f{d}{dt}  \i (1+ \vro) \, b_2 \,  \p_3^{m-1} \p_2 u_3 \, \p_3^m  u_2 \, dx + C \sqrt{\me(t)} \md(t).
\dal \deqq
Substituting  the estimates from $K_{1}$ to $K_{7}$ into the equality \eqref{Nlinear2-07}, we have
\beq\label{Nlinear2-tar} \bal
 \i \p_2 u_2 \,  |\p_3^m b_2|^2 \, dx  
\le &\;  \f{d}{dt} \i (b_2 -\vro)\,  |\p_3^m b_2|^2 \, dx + \f{d}{dt}  \i  (1+ \vro) (b_2 -\vro) \,|\p_3^m u_2|^2  \, dx \\
&\; - 2 \f{d}{dt}  \i (1+ \vro) \, (b_2 - \vro) \,  \p_3^{m-1} \p_2 u_3 \, \p_3^m  u_2 \, dx.
\dal \deq
Substituting above estimates into \eqref{Nlinear2-06}, we have
\beqq \bal
II_{3,3}
\le &\; \f{d}{dt} \i (b_2 -\vro)\,  |\p_3^m b_2|^2 \, dx + \f{d}{dt}  \i  (1+ \vro) (b_2 -\vro) \,|\p_3^m u_2|^2  \, dx \\
&\; - 2 \f{d}{dt}  \i (1+ \vro) \, (b_2 - \vro) \,  \p_3^{m-1} \p_2 u_3 \, \p_3^m  u_2 \, dx   + C \sqrt{\me(t)} \md(t).
\dal \deqq
In a similar way, integrating by parts, we can estimate
\beqq \bal
II_{3,2} = &\;  \i \p_3 (\nabla_h \cdot u_h \, \p_3^{m-1} u_3 ) \, \p_3^{m-1} \p_t b_2 \, dx \\
= &\;  \i \p_3 (\nabla_h \cdot u_h \, \p_3^{m-1} u_3 ) \, \p_3^{m-1} (\p_1^2 b_2  - \du + \p_2 u_2  + (F_3)_2) \, dx \\
\lesssim &\; \sqrt{\me(t)} \md(t).
\dal \deqq
Substituting the estimates from $II_{3,1}$ to 
$II_{3,4}$ into the equality \eqref{Nlinear2-05}, we have
\beq \label{Nlinear2-08} \bal
II_3 \le &\; \f{d}{dt} \i \nabla_h \cdot u_h \, \p_3^m b_2 \, \p_3^{m-1} u_3 \, dx + \f{d}{dt} \i (b_2 -\vro)\,  |\p_3^m b_2|^2 \, dx \\
&\; + \f{d}{dt}  \i  (1+ \vro) (b_2 -\vro) \,|\p_3^m u_2|^2  \, dx  - 2 \f{d}{dt}  \i (1+ \vro) \, (b_2 - \vro) \,  \p_3^{m-1} \p_2 u_3 \, \p_3^m u_2 \, dx \\
&\;  + C \sqrt{\me(t)} \md(t).
\dal \deq
{\bf Estimate of the term $II_4$.}  
     Due to the lower dissipative structure of density, we substitute the equation of $\p_3 \vro$ to deal with this term, which yields that
        \begin{align*}
       II_4 
         = &\;  \f{d}{dt} \i \vro \, \nabla_h \cdot u_h \, \p_3^{m} \vro \, \p_3^{m-1} u_3 \, dx  -\i \p_t \vro \, \nabla_h \cdot u_h \, \p_3^{m} \vro \, \p_3^{m-1} u_3  \, dx  \\
         &\; -\i \vro \, \p_t \nabla_h \cdot u_h \, \p_3^{m} \vro \, \p_3^{m-1} u_3  \, dx  -\i \vro\, \nabla_h \cdot u_h \, \p_3^{m} \p_t \vro \, \p_3^{m-1} u_3  \, dx  \\
        &\; - \i \vro \, \nabla_h \cdot u_h \,  \p_3^{m} \vro \,  \p_3^{m-1} (\p_1^2 u_3 + \p_2^2 u_3  + \p_3 \du+\p_2 b_3  - \p_3 b_2 + (F_2)_3 )\, dx \\
        := &\; \f{d}{dt} \i \vro \, \nabla_h \cdot u_h \, \p_3^{m} \vro \, \p_3^{m-1} u_3 \, dx 
        + \sum_{i=1}^{4} II_{4,i}.
        \end{align*}
        Using the estimate \eqref{Nlinear1-help-1}, \eqref{Nlinear1-help-2} and the following estimate 
        \beqq
        \| \p_t \vro\|_{H^{1}} \lesssim \| \du\|_{H^{1}} + \| \vro \, \du \|_{H^{1}} + \| u \cdot \nabla \vro\|_{H^{1}} \lesssim  \sqrt{\md(t)},
        \deqq
        it holds, 
        \begin{align*}
II_{4,1} + II_{4,2} \lesssim &\;( \| \nabla_h \cdot  u_h\|_{L^{\infty}} \| \p_t \vro\|_{L^{2}}^{\f12}  \| \p_1 \p_t \vro\|_{L^2}^{\f12} + \| \vro\|_{L^{\infty}} \| \nabla_h \cdot \p_t u_h\|_{L^2}^{\f12} \| \p_1 \nabla_h \cdot \p_t u_h\|_{L^2}^{\f12} )\\
&\; \times \| \p_3^m \vro\|_{L^2} \| \p_3^{m-1} u_3 \|_{H_{tan}^1}^{\f12} \| \p_3^{m} u_3 \|_{H_{tan}^1}^{\f12}\\
 \lesssim &\; \sqrt{\me(t)} \md(t),\\
 II_{4,4} \lesssim &\;  \| (\vro, \p_3 \vro)\|_{L^{2}}^{\f14} \|\nabla_h (\vro, \p_3 \vro)\|_{H_{tan}^1}^{\f34} \| \nabla_h \cdot u_h\|_{L^{\infty}} \| \p_3^m \vro\|_{L^2}\\
&\; \times \|  \p_3^{m-1} (\p_1^2 u_3 + \p_2^2 u_3  + \p_3 \du+\p_2 b_3  + (F_2)_3 ) \|_{L^2} \\
&\; + \| (\vro, \p_3 \vro)\|_{L^{2}}^{\f14} \|\nabla_h (\vro, \p_3 \vro)\|_{H_{tan}^1}^{\f34} \| \nabla_h \cdot u_h\|_{H^2}  \| \p_3^m \vro\|_{L^2}\|  \p_3^{m}  b_2\|_{L^2}^{\f12} \| \p_1 \p_3^{m}  b_2\|_{L^2}^{\f12}\\
\lesssim &\; \sqrt{\me(t)} \md(t).
\end{align*}
Similar to the estimate $II_{2}$, we can immediately obtain 
\beqq 
  II_{4,3}  \lesssim  \sqrt{\me(t)} \md(t),
\deqq
since we have more terms than $II_2$.
Thus, combining the estimates from $II_{4,1}$ to $II_{4,4}$, 
\beq\label{Nlinear2-09}
II_4 \le  \f{d}{dt} \i \vro \, \nabla_h \cdot u_h \, \p_3^{m} \vro \, \p_3^{m-1} u_3 \, dx + C  \sqrt{\me(t)} \md(t).
\deq
{\bf Estimate of the terms $II_{5}-II_{7}$.} 
Similar to the estimate of the terms $I_{5}-I_{7}$, we can check that
\beqq
II_5+ II_6+II_7 \lesssim \sqrt{\me(t)} \md(t).
\deqq
Substituting above estimate, \eqref{Nlinear2-03}, \eqref{Nlinear2-04}, \eqref{Nlinear2-08} and \eqref{Nlinear2-09} into \eqref{Nlinear2-02}, we finally obtain that
\beqq \bal 
  \i \nabla_h \cdot u_h \, |\p_3^m \vro|^2 \, dx 
 \le &\;   \f{d}{dt} \i (1+\vro) \nabla_h \cdot u_h \, \p_3^{m} \vro \, \p_3^{m-1} u_3 \, dx +\f{d}{dt} \i \nabla_h \cdot u_h \, \p_3^m b_2 \, \p_3^{m-1} u_3 \, dx \\
&\;+ \f{d}{dt} \i (b_2 -\vro)\,  |\p_3^m b_2|^2 \, dx  + \f{d}{dt}  \i  (1+ \vro) (b_2 -\vro) \,|\p_3^m u_2|^2  \, dx \\
&\; - 2 \f{d}{dt}  \i (1+ \vro) \, (b_2 - \vro) \,  \p_3^{m-1} \p_2 u_3 \, \p_3^m  u_2 \, dx + C \sqrt{\me(t)} \md(t),
\dal \deqq
which yields the estimate \eqref{Nlinear2-01-1}. We have already obtained the estimate $\eqref{Nlinear2-01-2}$,  see \eqref{Nlinear2-tar} for details. Similar to the proof of  \eqref{Nlinear2-tar}, we obtain that
\beqq \bal
\i \p_2 u_2 \,  |\p_3^m b_1|^2 \, dx  
\le &\;  \f{d}{dt} \i (b_2 -\vro)\,  |\p_3^m b_1|^2 \, dx + \f{d}{dt}  \i  (1+ \vro) (b_2 -\vro) \,|\p_3^m u_1|^2  \, dx \\
&\;- 2 \f{d}{dt}  \i (1+ \vro) \, (b_2 - \vro) \,  \p_3^{m-1} \p_1  u_3 \, \p_3^m  u_1 \, dx   + C \sqrt{\me(t)} \md(t),
\dal \deqq
which yields the estimate \eqref{Nlinear2-01-3}. Thus, we complete the proof of this lemma.
     \end{proof}
     Finally, we estimate the difficult term involving two highest-order vertical derivatives of magnetic field.
     \begin{lemm}\label{lemm:Nlinear3}
     Under the assumption \eqref{assumption}, for any smooth solution $(\vro ,u, b)$ of equation \eqref{eqr}, it holds
     \beq \label{Nlinear3-01}\bal
       \i \p_2 u_1  \, \p_3^m b_1  \,  \p_3^m b_2 \,dx 
\le &\; \f{d}{dt}  \i  (2+ \vro) b_1 \, \p_3^m b_1 \,\p_3^m b_2  \, dx -  \f{d}{dt}  \i (1+ \vro) \, b_1 \,  \p_3^{m-1}  \p_2 u_3 \, \p_3^m u_1\, dx   \\
&\; -  \f{d}{dt}  \i (1+ \vro) \, b_1 \,  \p_3^{m-1} \p_1 u_3 \,  \p_3^m  u_2\, dx 
 + C \sqrt{\me(t)} \md(t).
       \dal \deq
     \end{lemm}
     \begin{proof}
        Substituting the equation of $\p_2 u_1$ 
        \beqq
      \p_2 u_1 = \p_t b_1 - \p_1^2 b_1 + u \cdot \nabla b_1 - b \cdot \nabla u_1 + b_1 \, \du
        \deqq
        into the term $ \i \p_2 u_1  \, \p_3^m b_1  \,  \p_3^m b_2 \,dx $, we have
        \begin{align} \label{Nlinear3-02} 
   &\; \i \p_2 u_1  \, \p_3^m b_1  \,  \p_3^m b_2 \,dx \notag\\
   =&\; \f{d}{dt} \i b_1 \,  \p_3^m b_1 \, \p_3^m b_2 \, dx - \i b_1 \,  \p_3^m \p_t b_1 \, \p_3^m b_2 \, dx - \i b_1 \,  \p_3^m b_1 \, \p_3^m \p_t b_2 \, dx \notag\\
   &\; - \i (\p_1^2 b_1 - u \cdot \nabla b_1 + b \cdot \nabla u_1 - b_1 \du) \,  \p_3^m b_1 \, \p_3^m b_2 \, dx \notag\\
   =&\; \f{d}{dt} \i b_1 \,  \p_3^m b_1 \, \p_3^m b_2 \, dx - \i b_1 \,  (\p_3^m \p_1^2 b_1  \, \p_3^m b_2 +  \p_3^m \p_1^2 b_2 \, \p_3^m b_1) \, dx \notag\\
   &\; - \i b_1 \,  \{\p_3^m  \p_2 u_1 \, \p_3^m b_2 +  \p_3^m (-\du + \p_2 u_2 ) \, \p_3^m b_1  \} \, dx + \i b_1 \,  \{\p_3^m (u \cdot \nabla b_1)  \, \p_3^m b_2 +  \p_3^m (u \cdot \nabla b_2) \, \p_3^m b_1  \} \, dx \notag\\
   &\; - \i b_1 \,  \{\p_3^m (b \cdot \nabla u_1)  \, \p_3^m b_2 +  \p_3^m (b \cdot \nabla u_2) \, \p_3^m b_1  \} \, dx  + \i b_1 \,  \{\p_3^m (b_1 \, \du)  \, \p_3^m b_2 +  \p_3^m (b_2 \, \du) \, \p_3^m b_1 \} \, dx \notag\\
   &\; - \i \p_1^2 b_1  \,  \p_3^m b_1 \, \p_3^m b_2 \, dx
   + \i (u \cdot \nabla b_1 - b \cdot \nabla u_1 + b_1 \du) \,  \p_3^m b_1 \, \p_3^m b_2 \, dx \notag\\
   := &\; \f{d}{dt} \i b_1 \,  \p_3^m b_1 \, \p_3^m b_2 \, dx +\sum_{i=1}^{7} III_i.
        \end{align}
    Using the anisotropic type inequality \eqref{ie:Sobolev} and integrating by parts, it is easy to check that
    \begin{align*}
         III_1 +III_6 
        = &\;  2 \i b_1 \p_3^m \p_1 b_1 \, \p_3^m \p_1 b_2  \, dx + 2\i \p_1 b_1 \, \p_1 (\p_3^m b_1 \, \p_3^m  b_2 ) \, dx \\
        \lesssim &\; \| b\|_{L^{\infty}} \| \p_3^m \p_1 b\|_{L^2}^2  + \|\p_1 b\|_{H_{tan}^1}^{\f12}\|\p_{13} b\|_{H_{tan}^1}^{\f12} \| \p_3^m b\|_{L^2}^{\f12} \| \p_1 \p_3^m b\|_{L^2}^{\f32}\\
        \lesssim &\; \sqrt{\me(t)} \md(t), \\
        III_2  
        \lesssim &\; \| b\|_{L^2}^{\f14} \| \p_3 b\|_{L^2}^{\f14}\| \p_2 b\|_{L^2}^{\f14} \| \p_{23} b\|_{L^2}^{\f14}  \|\p_3^m b\|_{L^2}^{\f12}   \|\p_1 \p_3^m b\|_{L^2}^{\f12} (\| \p_3^m \p_2 u\|_{L^2} + \| \p_3^m \du\|_{L^2})\\
         \lesssim &\; \sqrt{\me(t)} \md(t),
    \end{align*}
    and integrating by parts, we can check that
    \beqq \bal
III_3 = &\;  \i b_1 \, u \cdot \nabla(\p_3^m b_1 \, \p_3^m b_2 )\, dx + \sum_{1\le l \le m} C_{m}^{l} \i b_1 \, (\p_3^l u \cdot \nabla \p_3^{m-l} b_1 \, \p_3^{m} b_2 + \p_3^l u \cdot \nabla \p_3^{m-l} b_2 \, \p_3^{m} b_1) \, dx\\
= &\;- \i( b_1 \, \du + u \cdot \nabla b_1) \, \p_3^m b_1 \, \p_3^m b_2 \, dx \\
&\; + \sum_{1\le l \le m} C_{m}^{l} \i b_1 \, (\p_3^l u \cdot \nabla \p_3^{m-l} b_1 \, \p_3^{m} b_2 + \p_3^l u \cdot \nabla \p_3^{m-l} b_2 \, \p_3^{m} b_1) \,dx,
\dal \deqq
which yields that
\begin{align*} 
III_3 \lesssim &\; \| \p_3^m b\|_{L^2} \| \p_1 \p_3^m b\|_{L^2} (\| b\|_{L^2}^{\f14} \| \p_3 b\|_{L^2}^{\f14}\| \p_2 b\|_{L^2}^{\f14} \| \p_{23} b\|_{L^2}^{\f14} \| \du \|_{L^2}^{\f14} \| \p_3 \du\|_{L^2}^{\f14}\| \p_2 \du\|_{L^2}^{\f14} \| \p_{23} \du\|_{L^2}^{\f14} \\
&\; + \| u\|_{L^2}^{\f14} \| \p_3 u\|_{L^2}^{\f14}\| \p_2 u\|_{L^2}^{\f14} \| \p_{23} u\|_{L^2}^{\f14} \| \nabla b_1 \|_{L^2}^{\f14} \| \p_3 \nabla b_1\|_{L^2}^{\f14}\| \p_2 \nabla b_1\|_{L^2}^{\f14} \| \p_{23} \nabla b_1\|_{L^2}^{\f14} )\\
&\; +\| b\|_{L^2}^{\f14} \| \p_3 b\|_{L^2}^{\f14}\| \p_2 b\|_{L^2}^{\f14} \| \p_{23} b\|_{L^2}^{\f14}  \|\p_3^m b\|_{L^2}^{\f12}   \|\p_1 \p_3^m b\|_{L^2}^{\f12} \\
&\;\times (\|\p_3^m u_h\|_{L^2}^{\f12}  \|\p_1 \p_3^m u_h\|_{L^2}^{\f12} \|\nabla_h b\|_{H^2} + \|\p_3^m u_3\|_{L^2}^{\f12}  \|\p_{3}^{m+1} u_3\|_{L^2}^{\f12} \|\p_3 b\|_{L^2}^{\f14} \| \nabla_h b\|_{L^2}^{\f12} \| \nabla_h^2 b\|_{L^2}^{\f14} \\
&\; + \sum_{1\le l \le m-1} \|\p_3^l u\|_{L^2}^{\f14} \| \p_3^{l+1} u\|_{L^2}^{\f14}  \|\p_2 \p_3^{l} u\|_{L^2}^{\f14} \| \p_2 \p_3^{l+1} u\|_{L^2}^{\f14} \| \nabla \p_3^{m-l} b\|_{L^2}^{\f12} \| \p_1 \nabla \p_3^{m-l} b\|_{L^2}^{\f12}) \\
 \lesssim &\; \sqrt{\me(t)} \md(t).
\end{align*}
In a similar way, we can check
\beqq
   III_5 + III_7 \lesssim  \sqrt{\me(t)} \md(t). 
    \deqq
    Finally, it remains to  estimate the term $III_4$, which is somewhat difficult. Integrating by parts and using the condition ${\rm div} b=0$, we have
    \beqq \bal
    III_4 = &\; -\i b_1 \, b \cdot \nabla \p_3^m u_1 \cdot \, \p_3^m b_2 \, dx - \i b_1 \, b \cdot \nabla \p_3^m u_2 \cdot \, \p_3^m b_1 \, dx  \\
    &\; - \sum_{1\le l \le m} C_{m}^{l} \i b_1 \, (\p_3^l b \cdot \nabla \p_3^{m-l} u_1 \, \p_3^{m} b_2 + \p_3^l b \cdot \nabla \p_3^{m-l} u_2 \, \p_3^{m} b_1) \, dx\\
= &\; \i b_1 \, \p_3^m u_1 \,  b \cdot  \nabla \p_3^m b_2 \, dx + \i b_1 \, \p_3^m u_2 \,  b \cdot  \nabla \p_3^m b_1  \, dx  + \i  \p_3^m u_1 \, b \cdot \nabla b_1 \, \p_3^m b_2   \, dx \\
&\;+ \i \p_3^m u_2 \,  b \cdot \nabla b_1 \, \p_3^m b_1  \, dx    -\sum_{1\le l \le m} C_{m}^{l} \i b_1 \, (\p_3^l b \cdot \nabla \p_3^{m-l} u_1 \, \p_3^{m} b_2 + \p_3^l b \cdot \nabla \p_3^{m-l} u_2 \, \p_3^{m} b_1) \, dx \\
: =&\; \sum_{i=1}^{5} III_{4,i}.
    \dal \deqq
Using the anisotropic type inequality \eqref{ie:Sobolev}, we have
\beqq \bal 
III_{4,3} + III_{4,4} \lesssim &\; \| \p_3^m u\|_{L^2}^{\f12} \| \p_1 \p_3^m u\|_{L^2}^{\f12} \| b\|_{L^2}^{\f14} \| \p_3 b\|_{L^2}^{\f14}\| \p_2 b\|_{L^2}^{\f14} \| \p_{23} b\|_{L^2}^{\f14}  \\
&\; \times  \| \nabla b_1 \|_{L^2}^{\f14} \| \p_3 \nabla b_1\|_{L^2}^{\f14}\| \p_2 \nabla b_1\|_{L^2}^{\f14} \| \p_{23} \nabla b_1\|_{L^2}^{\f14}   \| \p_3^m b\|_{L^2}^{\f12} \| \p_1 \p_3^m b\|_{L^2}^{\f12}\\
 \lesssim &\; \sqrt{\me(t)} \md(t), \\
 III_{4,5} \lesssim &\;  (\sum_{1\le l \le m-1} \| \p_3^{l} b\|_{L^2}^{\f14} \| \p_3^{l+1} b\|_{L^2}^{\f14}\| \p_1 \p_3^{l} b\|_{L^2}^{\f14} \|  \p_1 \p_3^{l+1} b \|_{L^2}^{\f14} \| \p_3^{m-l}  \nabla u\|_{L^2}^{\f12} \| \p_2 \p_3^{m-l}  \nabla u\|_{L^2}^{\f12}  \\
 &\; +  \| \p_3^{m} b\|_{L^2}^{\f12} \| \p_1 \p_3^{l} b\|_{L^2}^{\f12} \| \nabla u\|_{L^2}^{\f14}\| \p_3 \nabla u\|_{L^2}^{\f14}   \| \p_2 \nabla u\|_{L^2}^{\f14} \| \p_{23} \nabla u\|_{L^2}^{\f14} )\\ 
&\; \times  \|  b_1 \|_{L^2}^{\f14} \| \p_3 b_1\|_{L^2}^{\f14}\| \p_2 b_1\|_{L^2}^{\f14} \| \p_{23} b_1\|_{L^2}^{\f14}   \| \p_3^m b\|_{L^2}^{\f12} \| \p_1 \p_3^m b\|_{L^2}^{\f12}\\
 \lesssim &\; \sqrt{\me(t)} \md(t).
\dal \deqq
It remains to estimate the term $III_{4,1}$ and $III_{4,2}$. Similar to the estimate of $I_{3,3,5}$, integrating by parts, we can check 
\beqq \bal
III_{4,1} +III_{4,2} 
\le &\; \f{d}{dt}  \i  (1+ \vro) \, b_1 \, \p_3^m b_1 \,\p_3^m b_2  \, dx  -  \f{d}{dt}  \i (1+ \vro) \, b_1 \,  \p_3^{m-1} \p_2 u_3 \, \p_3^m  u_1 \, dx  \\
&\; -  \f{d}{dt}  \i (1+ \vro) \, b_1 \,  \p_3^{m-1} \p_1 u_3 \,   \p_3^m u_2 \, dx  + C \sqrt{\me(t)} \md(t),
\dal \deqq
which, together with the above estimates, yields that
\beqq \bal
III_{4} \le &\;\f{d}{dt}  \i  (1+ \vro) \, b_1 \, \p_3^m b_1 \,\p_3^m b_2  \, dx  -  \f{d}{dt}  \i (1+ \vro) \, b_1 \,  \p_3^{m-1} \p_2 u_3 \, \p_3^m  u_1 \, dx  \\
&\; -  \f{d}{dt}  \i (1+ \vro) \, b_1 \,  \p_3^{m-1} \p_1 u_3 \,   \p_3^m u_2 \, dx  + C \sqrt{\me(t)} \md(t).
\dal \deqq
Substituting the estimates from $III_1$ to $III_7$ into the equality \eqref{Nlinear3-02}, we obtain that
\beqq \bal 
 \i \p_2 u_1  \, \p_3^m b_1  \,  \p_3^m b_2 \,dx 
\le &\; \f{d}{dt}  \i   (2+\vro) \, b_1 \, \p_3^m b_1 \,\p_3^m b_2  \, dx -  \f{d}{dt}  \i (1+ \vro) \, b_1 \,  \p_3^{m-1} \p_2 u_3 \, \p_3^m  u_1 \, dx   \\
&\; - \f{d}{dt}  \i (1+ \vro) \, b_1 \,  \p_3^{m-1} \p_1 u_3 \,   \p_3^m u_2 \, dx + C \sqrt{\me(t)} \md(t). 
\dal \deqq
Therefore, we finish the proof of this lemma.
     \end{proof}
     
     \subsubsection{Estimates for the vertical derivative}
     With the Lemmas \ref{lemm:Nlinear1}-\ref{lemm:Nlinear3} in hand, we now establish the estimates for the vertical derivative. 
        \begin{lemm}\label{lemma32}
            Under the assumption \eqref{assumption}, for any smooth solution $(\vro ,u, b)$ of equation \eqref{eqr}, it holds
            \beq \label{p3m01}
            \f{d}{dt} \i \{  |\p_3^m \vro|^2 + (1+\vro) |\p_3^m u|^2 +  |\p_3^m b|^2 \} \, dx +\f{d}{dt} F(\vro, u, b) +\|\p_3^m( \nabla_h  u, \du, \p_1 b) \|_{L^2}^2  
\lesssim  \sqrt{\me(t)} \md(t),
\deq
where 
\begin{align}\label{def-F}
F(\vro, u, b):= &\; - \i \p_3^{m-1}  u_3  \, (\p_1^2 u+ \p_2^2 u + \nabla \du - \nabla b_2 + \p_2 b)_h \cdot \p_3^m u_h \, dx \notag \\
&\; - m \i \nabla_h \cdot u_h \, \p_3^{m} \vro \, \p_3^{m-1} u_3 \, dx+\f12 \i \vro \,(|\p_3^m \vro|^2-|\p_3^m b_1|^2 -|\p_3^m b_2|^2) \, dx\notag \\
&\;-(m+1) \i \du \, (\p_3^{m} \vro - \p_3^m b_2) \, \p_3^{m-1} u_3 \, dx -m \i \nabla_h \cdot u_h \, \p_3^m b_2 \, \p_3^{m-1} u_3 \, dx \notag \\
&\; - \i b_2 \,( m|\p_3^m b_1 |^2 + (2m+1) |\p_3^m b_2 |^2)\, dx  - m \i \p_3 u_h \cdot \nabla_h \p_3^{m-1} \vro \, \p_3^{m-1} u_3 \, dx \notag \\
&\; - \f12  \i  (1+ \vro) \vro \,(|\p_3^m u_1|^2+ |\p_3^m u_2|^2)  \, dx  +   \i (1+ \vro)\, \vro \,  \p_3^{m-1} \p_1 u_3  \, \p_3^m  u_1 \, dx \notag \\
&\; + \i (1+ \vro)\,  \vro \, \p_3^{m-1} \p_2 u_3 \, \p_3^m u_2\, dx -   \i  (1+ \vro)\, b_2 \,( m|\p_3^m u_1|^2 +  (2m+1)|\p_3^m u_2|^2 ) \, dx \notag\\
&\;+  2m \i (1+ \vro) \,  b_2 \,  \p_3^{m-1} \p_1  u_3 \,  \p_3^m u_1 \, dx +  2(2m+1) \i (1+ \vro) \,b_2\, \p_3^{m-1} \p_2 u_3 \,    \p_3^m  u_2 \, dx\notag \\
&\;   - \i  (2+ \vro) \, b_1 \, \p_3^m b_1 \,\p_3^m b_2  \, dx  +    \i (1+ \vro) \, b_1 \,  (\p_3^{m-1} \p_2 u_3 \, \p_3^m  u_1 + \p_3^{m-1} \p_1 u_3 \,   \p_3^m u_2  ) \, dx.
\end{align}
        \end{lemm}
        \begin{proof}
        Applying the $\p_3^m$-operator to the equation $\eqref{eqr}_2$ and multiplying by  $1+ \vro$,  we have
            \beqq \bal
            &\; (1+\vro) \p_t \p_3^m u -\p_{1}^2 \p_3^m  u-\p_{2}^2 \p_3^m  u-\nabla \p_3^m  {\rm div}u
				+\p_3^m (\nabla \vro+\nabla b_2-\p_2 b)\\
                 =&\;\vro\, (\p_{1}^2 \p_3^m  u+\p_{2}^2 \p_3^m  u+\nabla \p_3^m  {\rm div}u) - \vro \,\p_3^m (\nabla \vro+\nabla b_2-\p_2 b)  +(1+\vro)\p_3^m  F_2. 
            \dal \deqq
            Using the density equation $\eqref{eqr}_1$, it is easy to check that 
            \beqq \bal
			\i  (1+\vro) \p_t \p_3^m u \cdot \p_3^m u\, dx  = &\; \f{d}{dt} \f12 \i  (1+\vro) |\p_3^m u|^2 \, dx -\f12 \i  \p_t(1+\vro) |\p_3^m u|^2 \, dx \\
			= &\; \f{d}{dt} \f12 \i   (1+\vro) |\p_3^m u|^2 \, dx + \f12 \i |\p_3^m u|^2  (\du +u \cdot \nabla \vro + \vro \, \du)\, dx. 
			\dal \deqq
            Thus we have
			\begin{align*}
				&\f{d}{dt} \f12 \i   (1+\vro) |\p_3^m u|^2 \, dx  +\| \nabla_h \p_3^m u \|_{L^2}^2 
				+\i \p_3^m (\nabla \vro+\nabla b_2-\p_2 b) \cdot \p_3^m u \, dx\\
				=& -\f{1}{2}\i  |\p_3^m u|^2  (\du +u \cdot \nabla \vro + \vro \, \du)\, dx+\i \vro\, (\p_{1}^2 \p_3^m  u+\p_{2}^2 \p_3^m  u+\nabla \p_3^m  {\rm div}u)\cdot \p_3^m u \, dx\\
				&
				-\i \vro \, \p_3^m (\nabla \vro+\nabla b_2-\p_2 b) \cdot \p_3^m u\, dx + \i (1+\vro) \p_3^m  F_2 \cdot \p_3^m u\, dx,
			\end{align*}
            which, together with the equation $\eqref{eqr}_1$ and  $\eqref{eqr}_3$, yields that
\begin{align}\label{p3m-equ}
				&\;\f{d}{dt} \f12 \i (  |\p_3^m \vro|^2 + (1+\vro) |\p_3^m u|^2 +  |\p_3^m b|^2 ) \, dx  +\|\p_3^m( \nabla_h  u, \du, \p_1 b) \|_{L^2}^2  \notag \\
				=&\; -\f{1}{2}\i  |\p_3^m u|^2  (\du +u \cdot \nabla \vro + \vro \, \du)\, dx+\i \vro\, (\p_{1}^2 \p_3^m  u+\p_{2}^2 \p_3^m  u+\nabla \p_3^m  {\rm div}u)\cdot \p_3^m u \, dx \notag \\
				&\;
				-\i \vro \, \p_3^m \nabla \vro \cdot \p_3^m u\, dx -\i \vro \, \p_3^m (\nabla b_2-\p_2 b) \cdot \p_3^m u\, dx - \i (1+\vro) \p_3^m ( u \cdot \nabla u+\vro \nabla \vro) \cdot \p_3^m u\, dx \notag \\
                &\; - \i (1+\vro) \p_3^m\Big\{\frac{\vro}{1+\vro}
				(\p_{1}^2 u+\p_{2}^2 u+\nabla {\rm div}u -\nabla b_2+\p_2 b)\Big\} \cdot \p_3^m u\, dx \notag\\ 
                &\;  + \i (1+\vro) \p_3^m  \Big\{\frac{1}{1+\vro}b\cdot \nabla b
				-\frac{1}{2(1+\vro)}\nabla |b|^2\Big\} \cdot \p_3^m u\, dx  \notag \\
                &\;   -\i  \p_3^m (u \cdot \nabla \vro + \vro \, \du) \, \p_3^m \vro \, dx + \i  \p_3^m (- u \cdot \nabla b + b \cdot \nabla u- b \, \du ) \cdot \p_3^m b \, dx \notag\\
                := &\; \sum_{i=1}^{9} IV_{i}, 
			\end{align}
            where we have used the basic fact
			\beqq \bal
			& \i \p_3^m \nabla \vro \cdot \p_3^m u \, dx
			+\i \p_3^m \du \, \p_3^m \vro \, dx=0,\\
			& \i \p_3^m (\nabla b_2 - \p_2 b )\cdot \p_3^m u \, dx
			+\i \p_3^m (e_2 \du - \p_2 u )\cdot \p_3^m b \, dx=0.
			\dal \deqq
           We now consider the estimates for the nonlinear terms. Using the anisotropic type inequality \eqref{ie:Sobolev}, it holds 
            \beq \label{p3m-est01}\bal
           IV_1 \lesssim &\; \| \p_3^m u\|_{L^2} \| \p_1 \p_3^m u\|_{L^2}^{\f12} \|\p_2 \p_3^m u\|_{L^2}^{\f12} \| \du\|_{L^2}^{\f12} \| \p_3 \du\|_{L^2}^{\f12} (1+ \|\vro\|_{L^{\infty}}) \\
           &\; +  \| \p_3^m u\|_{L^2} \| \p_1 \p_3^m u\|_{L^2} \| u\|_{L^{2}}^{\f14} \| \p_2 u\|_{L^{2}}^{\f14} \| \p_3 u\|_{L^{2}}^{\f14} \|\p_{23} u\|_{L^{2}}^{\f14} \| \nabla \vro \|_{L^{2}}^{\f14} \| \p_2 \nabla \vro\|_{L^{2}}^{\f14}\| \p_3 \nabla \vro\|_{L^{2}}^{\f14}\| \p_{23} \nabla \vro\|_{L^{2}}^{\f14} \\
           \lesssim &\; \sqrt{\me(t)} \md(t).
           \dal \deq
           For simplicity, we denote that
                \beq\label{def-N1}
                N_1 := \p_{1}^2 u+\p_{2}^2 u+\nabla {\rm div}u -\nabla b_2+\p_2 b.
                \deq
                Using the anisotropic type inequality \eqref{ie:Sobolev}, we have 
           \begin{align} \label{p3m-246}
           &\; IV_2+IV_4+IV_6 \notag\\
           =&\; -\sum_{0 < l < m} C_{m}^{l} \i (1+\vro) \p_3^{l} (\frac{\vro}{1+\vro})
				\p_3^{m-l} N_1 \cdot \p_3^m u \, dx  - \i (1+\vro) \p_3^{m} (\frac{\vro}{1+\vro}) N_1 \cdot \p_3^m u \, dx 
                \notag\\
=&\; -\sum_{0 < l < m} C_{m}^{l} \i (1+\vro) \p_3^{l} (\frac{\vro}{1+\vro})
				\p_3^{m-l} N_1 \cdot \p_3^m u \, dx \notag\\
         &\; - \sum_{0 \le l < m} C_{m}^{l} \i (1+\vro) \p_3^{l}\vro \,  \p_3^{m-l}(\frac{1}{1+\vro}) N_1 \cdot \p_3^m u \, dx  \notag\\
                &\; - \i \p_3^{m}\vro \, (N_1)_3 \cdot \p_3^m u_3 \, dx  - \i \p_3^{m}\vro \, (N_1)_h \cdot \p_3^m u_h \, dx. 
                \end{align}
                It is easy to check that
                \begin{align} \label{p3m-246-1} 
                &\; -\sum_{0 < l < m} C_{m}^{l} \i (1+\vro) \p_3^{l} (\frac{\vro}{1+\vro})
				\p_3^{m-l} N_1 \cdot \p_3^m u \, dx \notag\\
\lesssim &\; \| \p_3^{m-1} N_1\|_{L^2}  \| \p_3^m u \|_{L^2}^{\f12} \| \p_1 \p_3^m u \|_{L^2}^{\f12} \| \p_3 \vro\|_{L^2}^{\f14}  \| \p_{23} \vro\|_{L^2}^{\f14} \| \p_3^2 \vro\|_{L^2}^{\f14} \| \p_{2} \p_3^2 \vro\|_{L^2}^{\f14}(1+ \|\vro\|_{L^{\infty}}) \notag\\
&\; + \sum_{1 < l < m}  \| \p_3^{l} (\frac{\vro}{1+\vro})\|_{L^2}^{\f12}  \|\p_1  \p_3^{l} (\frac{\vro}{1+\vro})\|_{L^2}^{\f12} \| \p_3^{m-l} N_1\|_{L^2}^{\f12} \| \p_3^{m-l+1} N_1\|_{L^2}^{\f12} \notag\\
&\; \quad\quad\quad\quad \times
\| \p_3^m u \|_{L^2}^{\f12} \| \p_2 \p_3^m u \|_{L^2}^{\f12} (1+ \|\vro\|_{L^{\infty}}) \notag\\
 \lesssim &\; \sqrt{\me(t)} \md(t),
     \end{align}
and similarly, we can check 
\beq \label{p3m-246-2}\bal
&\;  -\sum_{0 \le l < m} C_{m}^{l} \i (1+\vro) \p_3^{l}\vro \,  \p_3^{m-l}(\frac{1}{1+\vro}) N_1 \cdot \p_3^m u \, dx -
  \i \p_3^{m}\vro \, (N_1)_3 \cdot \p_3^m u_3 \, dx   \lesssim \sqrt{\me(t)} \md(t). 
\dal \deq
Clearly, the last term on the right-hand side of the equality \eqref{p3m-246} cannot be
directly controlled by $\sqrt{\me(t)} \md(t)$. In order to deal with  this problem, we substitute the equation 
\beqq \bal
\p_3 \vro = &\;  -\p_t u_3 +\p_1^2 u_3 + \p_2^2 u_3 + \p_3 \du+\p_2 b_3 - \p_3 b_2 + (F_2)_3\\
= &\;  -\p_t u_3 +\f{1}{1+\vro} (\p_1^2 u_3 + \p_2^2 u_3  + \p_3 \du+\p_2 b_3) \\
&\; + (\f{1}{1+\vro} \p_3 b_2 - u\cdot \nabla u_3 - \vro \p_3 \vro + \f{1}{1+\vro} (b \cdot \nabla b_3 - \f12 \p_3 (|b|^2)) \\
 := &\; -\p_t u_3  + N_2 +N_3
\dal 
\deqq
into the term and integrate by parts to get
\begin{align*}
&\; -\i \p_3^{m}\vro \, (N_1)_h \cdot \p_3^m u_h \, dx 
 =  -\i \p_3^{m-1} (-\p_t u_3  + N_2 +N_3) \, (N_1)_h \cdot \p_3^m u_h \, dx \\
 = &\; \f{d}{dt}  \i \p_3^{m-1}  u_3  \, (N_1)_h \cdot \p_3^m u_h \, dx - \i \p_3^{m-1}  u_3  \, \p_t (N_1)_h \cdot \p_3^m   u_h \, dx - \i \p_3^{m-1}  u_3  \, (N_1)_h \cdot \p_3^m \p_t u_h \, dx \\
 &\; -   \i \p_3^{m-1}  N_2 \, (N_1)_h \cdot \p_3^m u_h \, dx  - \i \p_3^{m-1}  N_3 \, (N_1)_h \cdot \p_3^m u_h \, dx \\
 = &\;  \f{d}{dt}  \i \p_3^{m-1}  u_3  \, (N_1)_h \cdot \p_3^m u_h \, dx - \i \p_3^{m-1}  u_3  \, \p_t (N_1)_h \cdot \p_3^m   u_h \, dx + \i \p_3(\p_3^{m-1}  u_3  \, (N_1)_h )\cdot \p_3^{m-1} \p_t u_h \, dx \\
 &\; -   \i \p_3^{m-1}  N_2 \, (N_1)_h \cdot \p_3^m u_h \, dx - \i \p_3^{m-1}  N_3 \, (N_1)_h \cdot \p_3^m u_h \, dx.
\end{align*}
Thus, using the anisotropic type inequality \eqref{ie:Sobolev}, we can check that
\begin{align} \label{p3m-246-3}
&\; -\i \p_3^{m}\vro \, (N_1)_h \cdot \p_3^m u_h \, dx - \f{d}{dt}  \i \p_3^{m-1}  u_3  \, (N_1)_h \cdot \p_3^m u_h \, dx \notag\\
\lesssim &\;   \| \p_3^{m-1} u_3\|_{H_{tan}^1}^{\f12} \| \p_3^{m} u_3\|_{H_{tan}^1}^{\f12} \| \p_3^m u_h\|_{L^2}^{\f12} \|\p_1 \p_3^m u_h\|_{L^2}^{\f12} \| \p_t (N_1)_h\|_{L^2} \notag\\
&\; + \|\p_3^{m-1} \p_t u_h\|_{L^2}( \| \p_3^{m-1} u_3\|_{H_{tan}^1}^{\f12} \| \p_3^{m} u_3\|_{H_{tan}^1}^{\f12}  \|\p_3 (N_1)_h \|_{L^2}^{\f12} \|\p_{13} (N_1)_h \|_{L^2}^{\f12}  \notag\\
&\; +\| \p_3^{m} u_3\|_{L^2}^{\f12} \| \p_3^{m+1} u_3\|_{L^2}^{\f12} \| (N_1)_h \|_{L^2}^{\f14} \|\p_1 (N_1)_h \|_{L^2}^{\f14}  \| \p_3 (N_1)_h \|_{L^2}^{\f14} \|\p_{13} (N_1)_h \|_{L^2}^{\f14}) \notag\\
&\; + \| \p_3^{m} u_h\|_{L^2}^{\f12} \| \p_2 \p_3^{m} u_h\|_{L^2}^{\f12} ( \| \p_3^{m-1} N_1\|_{L^2} \| (N_1)_h\|_{H_{tan}^1}^{\f12}  \| \p_3 (N_1)_h\|_{H_{tan}^1}^{\f12} \notag \\
&\; +\| \p_3^{m-1} N_3\|_{L^2}^{\f12}  \| \p_1 \p_3^{m-1} N_3\|_{L^2}^{\f12} \| (N_1)_h\|_{L^2}^{\f12}  \| \p_3 (N_1)_h\|_{L^2}^{\f12}) \notag\\
\lesssim &\;  \sqrt{\me(t)} \md(t),
\end{align}
here we have used
\beqq \bal
\| \p_3^{m-1} \p_t u_h\|_{L^2} = &\; \| \p_3^{m-1} (\p_1^2 u_h + \p_2^2 u_h  +\nabla_h \du - \nabla_h \vro - \nabla_h b_2+ \p_2 b_h + (F_2)_h)\|_{L^2} \lesssim \sqrt{\md(t)},\\
\| \p_t (N_1)_h \|_{L^2} \lesssim &\; (  \| \p_t \nabla_h^2 u_h \|_{L^2} + \| \nabla_h {\rm div} \p_t  u \|_{L^2} +  \| \p_t \nabla_h b_h \|_{L^2} ) \\
\lesssim &\; \| (\nabla_h^2, \nabla_h {\rm div}) (\p_1^2 u + \p_2^2 u +\nabla \du - \nabla \vro - \nabla b_2+ \p_2 b + F_2 )\|_{L^2} \\
&\;+ \| \nabla_h (\p_1^2 b_h - e_2 \du + \p_2 u + (F_3)_h )\|_{L^2} \\
\lesssim &\;  \sqrt{\md(t)}.
\dal \deqq
Substituting the estimates from \eqref{p3m-246-1} to \eqref{p3m-246-3} into \eqref{p3m-246}, we have
\beq\label{p3m-est02}
IV_2 +IV_4 +IV_6 \le  \f{d}{dt}  \i \p_3^{m-1}  u_3  \, (N_1)_h \cdot \p_3^m u_h \, dx + C \sqrt{\me(t)} \md(t).
\deq
Next, integrating by parts, we notice that
\beq\label{p3m-38-1} \bal
IV_3 + IV_8 = &\;  \i  ( \vro \, \p_3^{m} \du + \p_3^{m} u \cdot \nabla \vro ) \, \p_3^{m} \vro \, dx -\sum_{0 \le l < m} C_{m}^{l} \i  \p_3^{l} \vro \, \p_3^{m-l} \du \, \p_3^{m} \vro \, dx \\
&\;- \i   \du |\p_3^m \vro|^2 \, dx -\sum_{0 \le l < m-1 } C_{m}^{l} \i   \p_3^{m-l} u \cdot \nabla \p_3^{l}  \vro \, \p_3^{m} \vro \, dx \\
&\; -m \i \p_3 u \cdot \nabla \p_3^{m-1} \vro \, \p_3^{m} \vro \, dx   - \f12 \i  u \cdot \nabla |\p_3^{m} \vro|^2 \, dx \\
 = &\; -(\f12 +m)
 \i \du |\p_3^m \vro|^2 \, dx  + m
 \i \nabla_h \cdot u_h |\p_3^m \vro|^2 \, dx -m \i \p_3 u_h \cdot \nabla_h \p_3^{m-1} \vro \, \p_3^{m} \vro \, dx \\
 &\; -\sum_{0 < l < m} C_{m}^{l} \i  \p_3^{l} \vro \, \p_3^{m-l} \du \, \p_3^{m} \vro \, dx -\sum_{0 < l < m-1 } C_{m}^{l} \i   \p_3^{m-l} u \cdot \nabla \p_3^{l}  \vro \, \p_3^{m} \vro \, dx,
\dal \deq
where we have used $\p_3 u_3 = \du - \nabla_h \cdot u_h$.
  Substituting the equation of $\p_3 \vro$
         into the third term on the right-hand side, 
        we have
         \beq \label{p3m-38-2} \bal
        &\; - m \i \p_3 u_h \cdot \nabla_h \p_3^{m-1} \vro \, \p_3^m \vro \, dx \\
        = &\; m\f{d}{dt} \i \p_3 u_h \cdot \nabla_h \p_3^{m-1} \vro \, \p_3^{m-1} u_3 \, dx   - m\i \p_3 \p_t u_h \cdot \nabla_h \p_3^{m-1} \vro \, \p_3^{m-1} u_3 \, dx \\
        &\;-  m\i \p_3  u_h \cdot \nabla_h \p_3^{m-1} \p_t \vro \, \p_3^{m-1} u_3 \, dx   + m\i \p_3 u_h \cdot \nabla_h \p_3^{m-1} \vro \, \p_3^{m} b_2 \,dx  \\
        &\; - m\i \p_3 u_h \cdot \nabla_h \p_3^{m-1} \vro \, \p_3^{m-1} (\p_1^2 u_3 + \p_2^2 u_3 + \p_3 \du+\p_2 b_3  + (F_2)_3 )\, dx \\
        := &\; m \f{d}{dt} \i \p_3 u_h \cdot \nabla_h \p_3^{m-1} \vro \, \p_3^{m-1} u_3 \, dx   
        +\sum_{i=1}^{4} L_{i}.
         \dal \deq
         We first estimate the term $L_1$. 
        Using the estimate \eqref{Nlinear1-help-2}, we can check
         \beqq \bal
    L_1 \lesssim &\;  \|\p_3 \p_t u_h \|_{L^2}^{\f12} \|\p_1 \p_3 \p_t u_h \|_{L^2}^{\f12} \| \nabla_h \p_3^{m-1} \vro\|_{L^2} \| \p_3^{m-1} u_3 \|_{L^2}^{\f14} \| \p_2 \p_3^{m-1} u_3 \|_{L^2}^{\f14} \| \p_3^{m} u_3 \|_{L^2}^{\f14} \| \p_2 \p_3^{m} u_3 \|_{L^2}^{\f14} \\
     \lesssim &\;  \sqrt{\me(t)} \md(t).
         \dal \deqq
         Substituting the equation $\eqref{eqr}_1$ into the term $L_2$, integrating by parts, we have
         \begin{align*}
         L_2 = &\;  \i \p_3  u_h \cdot \nabla_h \p_3^{m-1} \du\, \p_3^{m-1} u_3 \, dx  + \i \p_3  u_h \cdot \nabla_h \p_3^{m-1} (u \cdot \nabla \vro + \vro \, \du  )\, \p_3^{m-1} u_3 \, dx \\
          = &\;  \i \p_3  u_h \cdot \nabla_h \p_3^{m-1} \du\, \p_3^{m-1} u_3 \, dx  - \i \p_3^{m-1} (u \cdot \nabla \vro ) \, \nabla_h \cdot (\p_3  u_h \,  \p_3^{m-1} u_3 ) \, dx \\
         &\;  - \i \p_3^{m-1} (\vro \, \du  ) \,\nabla_h \cdot( \p_3  u_h \,  \p_3^{m-1} u_3) \, dx \\
         := &\; \sum_{i=1}^{3} L_{2,i}.
         \end{align*}
         Using the anisotropic type inequality \eqref{ie:Sobolev}, we can check that
         \begin{align*}
L_{2,2} = &\; - \sum_{ 0 \le l \le m-1 }C_{m-1}^{l} \i \p_3^{l} u \cdot \p_3^{m-1-l} \nabla \vro  \, \nabla_h \cdot (\p_3  u_h \,  \p_3^{m-1} u_3 ) \, dx \\
\lesssim &\; ( \sum_{1 \le  l \le m-1} \| \p_3^{m-1-l} \nabla \vro\|_{L^2}^{\f12} \| \p_2 \p_3^{m-1-l} \nabla \vro\|_{L^2}^{\f12} \| \p_3^{l} u\|_{L^2}^{\f14} \| \p_3^{l+1} u \|_{L^2}^{\f14} \| \p_1 \p_3^{l} u \|_{L^2}^{\f14} \| \p_1 \p_3^{l+1} u \|_{L^2}^{\f14}  \\
&\; +\| \p_3^{m-1} \nabla \vro\|_{L^2} \| u\|_{L^{\infty}} ) (\| \nabla_h \p_3^{m-1} u_3 \|_{L^2}^{\f12} \| \nabla_h \p_3^{m} u_3 \|_{L^2}^{\f12} \| \p_3 u_h \|_{L^2}^{\f14} \|\nabla_h \p_3 u_h\|_{L^2}^{\f12} \|\nabla_h^2 \p_3 u_h\|_{L^2}^{\f12}\\
&\; + \| \p_3^{m-1} u_3 \|_{L^2}^{\f14} \| \p_2 \p_3^{m-1} u_3 \|_{L^2}^{\f14} \| \p_3^{m} u_3 \|_{L^2}^{\f14} \| \p_2 \p_3^{m} u_3 \|_{L^2}^{\f14} \|\nabla_h \cdot \p_3 u_h \|_{L^2}^{\f12} \|\p_1 \nabla_h \cdot \p_3 u_h\|_{L^2}^{\f12}) \\
     \lesssim &\;  \sqrt{\me(t)} \md(t).
         \end{align*}
         Similarly, we can check
         \beqq
         L_{2,1} + L_{2,3} \lesssim \sqrt{\me(t)} \md(t),
         \deqq
         which, together with above estimate, yields that
         \beqq
         L_2 \lesssim \sqrt{\me(t)} \md(t).
         \deqq
         By Lemma \ref{lemm:sobolev-ie} and using the estimate \eqref{Nlinear1-help-1}, we can obtain that
         \begin{align*}
L_3+L_4  \lesssim \sqrt{\me(t)} \md(t).
         \end{align*}
         Substituting the estimates from $L_1$ to $L_4$ into \eqref{p3m-38-2}, we obtain that 
         \beq\label{p3m-38-3}
 - m \i \p_3 u_h \cdot \nabla_h \p_3^{m-1} \vro \, \p_3^m \vro \, dx \le m \f{d}{dt} \i \p_3 u_h \cdot \nabla_h \p_3^{m-1} \vro \, \p_3^{m-1} u_3 \, dx  +  C \sqrt{\me(t)} \md(t).
         \deq
Integrating by parts, we have
\begin{align}\label{p3m-38-4}
&\; -\sum_{0 < l < m} C_{m}^{l} \i  \p_3^{l} \vro \, \p_3^{m-l} \du \cdot \p_3^{m} \vro \, dx \notag\\
= &\; \sum_{0 < l < m-1} C_{m}^{l} \i \p_3 (\p_3^{l} \vro \, \p_3^{m-l} \du) \cdot \p_3^{m-1} \vro \, dx +  \f{m}2 \i \p_3^2 \du \, |\p_3^{m-1} \vro|^2 dx\notag\\
\lesssim &\; \| \p_3^{m-1} \vro\|_{L^2}^{\f12}  \| \p_2 \p_3^{m-1} \vro\|_{L^2}^{\f12} (\| \p_3^{2} \vro\|_{L^2}^{\f12} \| \p_2 \p_3^{2} \vro\|_{L^2}^{\f12} \| \p_3^{m-1} \du\|_{L^2}^{\f12} \| \p_3^{m} \du\|_{L^2}^{\f12} \notag\\
&\;  +\| \p_3 \vro\|_{L^2}^{\f14} \| \p_{23}\vro\|_{L^2}^{\f14}  \| \p_{3}^2 \vro\|_{L^2}^{\f14} \| \p_2 \p_{3}^2 \vro\|_{L^2}^{\f14} \| \p_3^{m} \du\|_{L^2} ) \notag\\
&\; + \sum_{1 < l < m-1} \| \p_3^{m-1} \vro\|_{L^2}^{\f12}  \| \p_2 \p_3^{m-1} \vro\|_{L^2}^{\f12} (\| \p_3^{l+1} \vro\|_{L^2}^{\f12} \| \p_2 \p_3^{l+1} \vro\|_{L^2}^{\f12} \| \p_3^{m-l} \du\|_{L^2}^{\f12} \| \p_3^{m+1-l} \du\|_{L^2}^{\f12} \notag \\
&\;  +\| \p_3^{l} \vro\|_{L^2}^{\f12} \| \p_2 \p_3^{l} \vro\|_{L^2}^{\f12} \| \p_3^{m+1-l} \du\|_{L^2}^{\f12} \| \p_3^{m+2-l} \du\|_{L^2}^{\f12} ) \notag\\
&\; +  \| \p_3^{m-1} \vro\|_{L^2} \| \p_1 \p_3^{m-1} \vro\|_{L^2}^{\f12} \| \p_2 \p_3^{m-1} \vro\|_{L^2}^{\f12} \| \p_3^2 \du \|_{L^2}^{\f12} \| \p_3^3 \du \|_{L^2}^{\f12}\notag\\
\lesssim &\; \sqrt{\me(t)} \md(t).
\end{align}
Similarly, we can check that
\beq \label{p3m-38-5}
-\sum_{0 < l < m-1 } C_{m}^{l} \i   \p_3^{m-l} u \cdot \nabla \p_3^{l}  \vro \cdot \p_3^{m} \vro \, dx  \lesssim \sqrt{\me(t)} \md(t).
\deq
Thus, substituting the estimates \eqref{p3m-38-3}-\eqref{p3m-38-5} into \eqref{p3m-38-1}, we have
\beq \label{p3m-est03} \bal
IV_3 + IV_8 \le &\; m \f{d}{dt} \i \p_3 u_h \cdot \nabla_h \p_3^{m-1} \vro \, \p_3^{m-1} u_3 \, dx -(\f12 +m)
 \i \du |\p_3^m \vro|^2 \, dx   \\
 &\;  + m
 \i \nabla_h \cdot u_h |\p_3^m \vro|^2 \, dx +C\sqrt{\me(t)} \md(t).
\dal \deq
Next, we deal with the term $IV_5$. Integrating by parts, we have
\begin{align*}
IV_5 = &\;- \i (1+\vro) u \cdot  \nabla \p_3^m u \cdot \p_3^m u \,dx -\sum_{0 < l \le m} C_{m}^{l} \i (1+\vro) \p_3^{l} u \cdot \nabla \p_3^{m-1} u \cdot \p_3^m u \, dx \\
&\; + \f12 \i \p_3^m ( \vro^2) \, {\rm div}( (1+\vro) \p_3^m u )\, dx\\
= &\; \f12 \i |\p_3^m u|^2 \, ((1+\vro)\, \du + u \cdot \nabla \vro) \,dx - \sum_{0 < l \le m} C_{m}^{l} \i (1+\vro) \p_3^{l} u \cdot \nabla \p_3^{m-l} u \cdot \p_3^m u \, dx \\
&\; + \f12 \sum_{0 \le l \le m} C_{m}^{l} \i \p_3^l \vro \, \p_3^{m-l} \vro \, \p_3^m u \cdot \nabla \vro \, dx +  \f12 \sum_{0 \le l \le m} C_{m}^{l} \i (1+\vro) \p_3^l \vro \, \p_3^{m-l} \vro \, \p_3^m \du\, dx \\
:= &\; \sum_{i=1}^{4} IV_{5,i}.
\end{align*}
Using the anisotropic type inequality \eqref{ie:Sobolev}, we can check that
\begin{align*}
 IV_{5,4} \le &\;  \i \vro \, \p_3^m \vro \, \p_3^m \du \, dx + C \| \p_3^m \vro\|_{L^2} \| \p_3^m \du\|_{L^2} \| (\vro, \p_3 \vro)\|_{L^2}^{\f12}  \| \nabla_h (\vro, \p_3 \vro)\|_{H_{tan}^1}^{\f12} \\
 &\; +C \| \p_3^m \du\|_{L^2} \| \p_3^{m-1} \vro\|_{L^2}^{\f12} \| \p_1 \p_3^{m-1} \vro\|_{L^2}^{\f12} \| \p_3 \vro\|_{L^2}^{\f14} \| \p_3^2 \vro\|_{L^2}^{\f14} \| \p_{23} \vro\|_{L^2}^{\f14} \| \p_2 \p_3^2 \vro\|_{L^2}^{\f14}   \\
 &\; + C \sum_{2 \le l \le m-2} \| \p_3^m \du\|_{L^2} \| \p_3^{l} \vro\|_{L^2}^{\f12} \| \p_1 \p_3^{l} \vro\|_{L^2}^{\f12} \| \p_3^{m-l} \vro\|_{L^2}^{\f14} \| \p_3^{m-l+1} \vro\|_{L^2}^{\f14} \| \p_{2}  \p_3^{m-l}\vro\|_{L^2}^{\f14} \| \p_2  \p_3^{m-l+1} \vro\|_{L^2}^{\f14}   \\
 \le &\; \i \vro \, \p_3^m \vro \, \p_3^m \du \, dx + C\sqrt{\me(t)} \md(t).
\end{align*}
We now estimate the first term on the right-hand side.
Substituting the equation of $\du$ into this term, we have
         \begin{align} \label{p3m-Nlinear} 
         &\; \i \vro \, \p_3^{m} \vro \, \p_3^m \du \, dx 
        =  \i \vro \, \p_3^{m} \vro \, \p_3^{m} (- \p_t \vro - u\cdot \nabla \vro - \vro \, \du) \, dx \notag\\
        = &\; - \f12 \f{d}{dt} \i \vro \,|\p_3^m \vro|^2 \, dx + \f12 \i \p_t \vro \,|\p_3^m \vro|^2 \, dx - \f12 \i \vro \,   u \cdot \nabla |\p_3^{m}  \vro|^2 \, dx \notag\\
        &\; - \sum_{ 1\le l\le m} C_m^l \i \vro \,  \p_3^{m} \vro \, \p_3^{l}  u\cdot \nabla \p_3^{m-l} \vro \, dx - \sum_{ 0\le l\le m} C_m^l \i \vro \,  \p_3^{m} \vro \, \p_3^{l}\, \vro \, \p_3^{m-l} \du \, dx \notag\\
         = &\; -\f12 \f{d}{dt}  \i \vro \,|\p_3^m \vro|^2 \, dx-\f12 \i \du \,|\p_3^m \vro|^2 \, dx + m \i \vro  \, \nabla_h \cdot u_h \,  |\p_3^{m} \vro|^2 \, dx \notag\\
         &\; -\sum_{2 \le l\le m} C_m^l \i \vro \,  \p_3^{m} \vro \, \p_3^{l}  u_3 \cdot  \p_3^{m-l+1} \vro \, dx -\sum_{1 \le l\le m} C_m^l \i \vro \,  \p_3^{m} \vro \, \p_3^{l}  u_h \cdot \nabla_h \p_3^{m-l} \vro \, dx\notag\\
        &\;  - \sum_{ 0\le l\le m-1} C_m^l \i \vro \,  \p_3^{m} \vro \, \p_3^{l}\, \vro \, \p_3^{m-l} \du \, dx - (m+1) \i \vro \, \du \,|\p_3^m \vro|^2 \, dx \notag\\
        = &\;-\f12 \f{d}{dt}  \i \vro \,|\p_3^m \vro|^2 \, dx -\f12 \i \du \,|\p_3^m \vro|^2 \, dx + \sum_{i=1}^{5} IV_{5,4,i},
        \end{align}
        here we have used
        \beqq \bal
        \f12 \i \p_t \vro \,|\p_3^m \vro|^2 \, dx - \f12 \i \vro \,   u \cdot \nabla |\p_3^{m}  \vro|^2 \, dx 
        = &\;   \f12 \i \p_t \vro \,|\p_3^m \vro|^2 \, dx + \f12 \i {\rm div} (\vro \,   u) \cdot \nabla |\p_3^{m}  \vro|^2 \, dx\\
        = &\; -\f12 \i \du \,|\p_3^m \vro|^2 \, dx,\\
        -m \i \vro\, \p_3 u_3 \, |\p_3^m \vro|^2 \, dx = &\;  - m \i \vro\,\du \,|\p_3^m \vro|^2 \, dx + m \i \vro  \, \nabla_h \cdot u_h \,  |\p_3^{m} \vro|^2 \, dx. 
        \dal \deqq
        Using the anisotropic type inequality \eqref{ie:Sobolev}, we can easily check that
\begin{align*}
IV_{5,4,2}+IV_{5,4,3}+IV_{5,4,4}
\lesssim \sqrt{\me(t)} \md(t).
\end{align*}
Similar to the estimate of $I_4$ and $II_4$, we can immediately obtain that
\beqq \bal
IV_{5,4,1} \le &\; - m \f{d}{dt} \i \vro \, \nabla_h \cdot u_h \, \p_3^{m} \vro \, \p_3^{m-1} u_3 \, dx  + C \sqrt{\me(t)} \md(t),\\
IV_{5,4,5} \le  &\; \f{m+1}{2}  \f{d}{dt} \i \vro^2 \, |\p_3^{m} \vro|^2 \, dx + C \sqrt{\me(t)} \md(t).
\dal \deqq
Substituting the estimates from $IV_{5,4,1}$ to  $IV_{5,4,5}$ into the equality \eqref{p3m-Nlinear}, we have
\beqq \bal
  \i \vro \, \p_3^{m} \vro \, \p_3^m \du \, dx 
 \le &\; - m \f{d}{dt} \i \vro \, \nabla_h \cdot u_h \, \p_3^{m} \vro \, \p_3^{m-1} u_3 \, dx -\f12 \f{d}{dt}  \i \vro \,|\p_3^m \vro|^2 \, dx   \\
 &\;+\f{m+1}{2} \f{d}{dt}  \i \vro^2 \,|\p_3^m \vro|^2 \, dx -\f12 \i \du \,|\p_3^m \vro|^2 \, dx+C\sqrt{\me(t)} \md(t).
\dal \deqq 
Thus, we have
\beqq \bal
IV_{5,4} \le &\; - m \f{d}{dt} \i \vro \, \nabla_h \cdot u_h \, \p_3^{m} \vro \, \p_3^{m-1} u_3 \, dx -\f12 \f{d}{dt}  \i \vro \,|\p_3^m \vro|^2 \, dx \\
&\; +\f{m+1}{2} \f{d}{dt}  \i \vro^2 \,|\p_3^m \vro|^2 \, dx-\f12 \i \du \,|\p_3^m \vro|^2 \, dx +C\sqrt{\me(t)} \md(t).
\dal \deqq
Similarly, we can obtain
\beqq
IV_{5,1} + IV_{5,2} + IV_{5,3} \lesssim \sqrt{\me(t)} \md(t).
\deqq
Thus, we have
\beq\label{p3m-est04} \bal
IV_{5} \le &\; - m \f{d}{dt} \i \vro \, \nabla_h \cdot u_h \, \p_3^{m} \vro \, \p_3^{m-1} u_3 \, dx -\f12 \f{d}{dt}  \i \vro \,|\p_3^m \vro|^2 \, dx  \\
&\; +\f{m+1}{2} \f{d}{dt}  \i \vro^2 \,|\p_3^m \vro|^2 \, dx-\f12 \i \du \,|\p_3^m \vro|^2 \, dx +C\sqrt{\me(t)} \md(t).
\dal \deq
Then we deal with the term $IV_7$. Integrating by parts, we have
\begin{align*}
IV_7 = &\;  \i b \cdot \nabla \p_3^m b \cdot \p_3^m u \, dx + \sum_{0 < l \le m} C_{m}^{l} \i \p_3^{l} b \cdot \nabla \p_3^{m-l} b \cdot \p_3^m u \, dx  \\
&\; + 
\sum_{0 < l \le m} C_{m}^{l} \i (1+\vro)\p_3^{l} (\f{1}{1+\vro}) \cdot \p_3^{m-l} (b \cdot \nabla b) \cdot \p_3^m u \, dx +  \f12\sum_{0 \le l \le m} C_{m}^{l} \i \p_3^{l} b \cdot  \p_3^{m-l} b \cdot \p_3^m \du \, dx \\
&\; -\f12 \sum_{0 < l \le m} C_{m}^{l} \i (1+\vro)\p_3^{l} (\f{1}{1+\vro}) \cdot \nabla \p_3^{m-l} (|b|^2) \cdot \p_3^m u \, dx \\
:= &\;  \i b \cdot \nabla \p_3^m b \cdot \p_3^m u \, dx + \sum_{i=1}^{4} IV_{7,i}.
\end{align*}
We only give the estimate of $IV_{7,1}$, and the terms from  $IV_{7,2}$ to $IV_{7,4}$ can be estimated similarly.  Using the anisotropic type inequality \eqref{ie:Sobolev}, we have
\begin{align*}
IV_{7,1}
\lesssim &\; \| \p_3^{m} u\|_{L^2}^{\f12} \| \p_2 \p_3^{m} u\|_{L^2}^{\f12} (\| \p_3 b_h\|_{L^2}^{\f14}  \| \p_1 \p_3 b_h\|_{L^2}^{\f14}  \| \p_3^2 b_h\|_{L^2}^{\f14} \| \p_1 \p_3^2 b_h\|_{L^2}^{\f14} \|\nabla_h \p_3^{m-1} b\|_{L^2}  \\
&\; + \sum_{2\le l \le m} \| \p_3^{l} b_h\|_{L^2}^{\f12}  \| \p_1 \p_3^{l} b_h\|_{L^2}^{\f12}  \|\nabla_h \p_3^{m-l} b \|_{L^2}^{\f12} \|\nabla_h  \p_3^{m-l+1} b \|_{L^2}^{\f12})\\
&\; + \| \p_3^{m} u\|_{L^2}^{\f12} \| \p_2 \p_3^{m} u\|_{L^2}^{\f12} (\sum_{1\le l \le m-1} \| \p_3^{l} b_3\|_{L^2}^{\f12} \| \p_3^{l+1} b_3\|_{L^2}^{\f12}  \| \p_3^{m-l+1} b\|_{L^2}^{\f12} \| \p_1 \p_3^{m-l+1} b\|_{L^2}^{\f12} \\
&\; + \| \p_3^m b_3\|_{L^2}\| \p_3 b\|_{L^2}^{\f14} \| \p_3^2 b\|_{L^2}^{\f14}  \| \p_1 \p_3 b\|_{L^2}^{\f14} \| \p_1 \p_3^{2} b\|_{L^2}^{\f14}  )\\
\lesssim &\; \sqrt{\me(t)} \md(t),
\end{align*}
and 
\beqq
IV_{7,2} + IV_{7,3} + IV_{7,4} \lesssim \sqrt{\me(t)} \md(t).
\deqq
Thus, we can check that
\beq \label{p3m-est05}
IV_7  \le \i b \cdot \nabla \p_3^m b \cdot \p_3^m u \, dx + C \sqrt{\me(t)} \md(t).
\deq
It remains to estimate the term $IV_9$.  We separate it as follows:   
\begin{align}\label{p3m-9}
IV_9 = &\;- \f12 \i  u \cdot \nabla |\p_3^{m} b|^2 \, dx -  m\i \p_3 u \cdot \nabla \p_3^{m-1} b \cdot \p_3^m b \, dx - \sum_{1 < l \le m} C_{m}^{l} \i \p_3^{l} u \cdot \nabla \p_3^{m-l} b \cdot \p_3^m b \, dx \notag\\
&\; + \i  b \cdot \nabla \p_3^{m} u \cdot \p_3^m b \, dx + \sum_{0 < l < m} C_{m}^{l} \i \p_3^{l} b \cdot \nabla \p_3^{m-l} u \cdot \p_3^m b \, dx  + \i \p_3^{m} b \cdot \nabla u \cdot \p_3^m b \, dx \notag\\
 &\; - \sum_{0 \le l < m} C_{m}^{l} \i \p_3^{l} b \cdot \p_3^{m-l} \du \cdot \p_3^m b \, dx -\i \du \,|\p_3^m b|^2 \,   dx \notag\\
 = &\; \sum_{i=1}^{8} IV_{9,i}.
\end{align}
First of all, due to the condition $ {\rm div} b=0$, integrating by parts, we have
\beqq
IV_{9,4} + \i b \cdot \nabla \p_3^m b \cdot \p_3^m u \, dx = \i b \cdot \nabla (\p_3^m b \cdot \p_3^m u)\, dx =0,
\deqq
and 
\beqq \bal
IV_{9,1} + IV_{9,8}= &\; - \f12\i \du \,  | \p_3^m b_3 |^2 \, dx  -\f12 \i  \du \, | \p_3^m b_h |^2 \, dx \\
\le &\; C \| \p_3^m b_3\|_{L^2}^2 \| \du \|_{L^\infty} -\f12 \i   \du \, | \p_3^m b_h |^2 \,dx\\
\le  &\;- \f12  \i   \du \, | \p_3^m b_h |^2 \,dx +C \sqrt{\me(t)} \md(t)  .
\dal \deqq
Next, integrating by parts, we can check that
\begin{align*}
IV_{9,2}
= &\; -m \i  (\du -\p_1 u_1 -\p_2 u_2)  \,|\p_3^m b_h|^2 \,  dx - m \i \p_3 u_3  \,|\p_3^m b_3|^2 \,  dx\\
&- m \i  \p_3 u_h \cdot \nabla_h \p_3^{m-1} b \cdot \p_3^m b \,dx \\
= &\;  - m \i  (\du -\p_2 u_2)  \, |\p_3^m b_h|^2 \,dx- m \i  u_1 \, \p_1 |\p_3^m b_h|^2  \,dx\\
&\; - m \i \p_3 u_3  \, |\p_3^m b_3|^2 \,  dx 
 - m \i  \p_3 u_h \cdot \nabla_h \p_3^{m-1} b \cdot \p_3^m b \,dx \\
\le  &\;   -m \i (\du -\p_2 u_2)   \, |\p_3^m b_h|^2  \, dx +C \sqrt{\me(t)} \md(t), 
\end{align*}
and integrating by parts, we have
\begin{align*}
 IV_{9,6} = &\;  \i \p_3^m b_1 \, \p_1 u   \cdot \p_3^m b \,dx +   \i \p_3^m b_2 \, \p_2 u  \cdot \p_3^m b \,dx +   \i \p_3^m b_3 \, \p_3 u   \cdot \p_3^m b \,dx \\
= &\;  -\i  u \cdot \p_1(\p_3^m b_1   \, \p_3^m b )\,dx +   \i \p_3^m b_2 \, \p_2 u_1   \, \p_3^m b_1 \,dx  +   \i \p_3^m b_2 \, \p_2 u_2   \, \p_3^m b_2 \,dx  \\
 &\; +   \i \p_3^m b_2 \, \p_2 u_3   \, \p_3^m b_3 \,dx +   \i \p_3^m b_3 \, \p_3 u   \cdot \p_3^m b \,dx \\
 \le &\; C \| u\|_{L^2}^{\f14} \| \p_2 u\|_{L^2}^{\f14}\| \p_3 u\|_{L^2}^{\f14} \| \p_{23} u\|_{L^2}^{\f14} \| \p_1 \p_3^m b\|_{L^2} \| \p_3^m b\|_{L^2}^{\f12} \| \p_1 \p_3^m b\|_{L^2}^{\f12} +   \i \p_3^m b_2 \, \p_2 u_1   \, \p_3^m b_1 \,dx   \\
 &\; +   \i \p_2 u_2 \, |\p_3^m b_2|^2 \, dx + C \| \p_3^m b_3\|_{L^2}  \| \p_3^m b\|_{L^2}^{\f12}  \| \p_1 \p_3^m b\|_{L^2}^{\f12} \|\p_2 u\|_{H_{tan}^1}^{\f12} \|\p_3 \p_2 u\|_{H_{tan}^1}^{\f12} \\
&\;  + C \| \p_3^m b_3\|_{L^2}  \| \p_3^m b\|_{L^2}^{\f12}  \| \p_1 \p_3^m b\|_{L^2}^{\f12} \|\p_3 u\|_{L^2}^{\f14}  \|\p_3^2 u\|_{L^2}^{\f14} \|\p_{23} u\|_{L^2}^{\f14} \|\p_2 \p_3^2 u\|_{L^2}^{\f14} \\
\le  &\;   \i \p_3^m b_2 \, \p_2 u_1   \, \p_3^m b_1 \,dx  +   \i  \p_2 u_2 \, |\p_3^m b_2|^2 \,dx +C \sqrt{\me(t)} \md(t).
\end{align*}
Similar to the estimate of $IV_{7,1}$, we have
\beqq
IV_{9,3} + IV_{9,5} +IV_{9,7} \lesssim    \sqrt{\me(t)} \md(t).
\deqq
Substituting the estimates from $IV_{9,1}$ to $IV_{9,8}$ into \eqref{p3m-9} and combining the estimate \eqref{p3m-est05} of $IV_7$, we have
\beq\label{p3m-est06} \bal
IV_{7}+IV_{9} 
\le &\;   -(m+\f12) \i \du \, | \p_3^m b_h |^2 \,  dx + m \i \p_2 u_2 \, |\p_3^m b_h|^2  \, dx\\
&\;+   \i \p_2 u_2 \, |\p_3^m b_2|^2 \, dx + \i \p_3^m b_2 \, \p_2 u_1   \, \p_3^m b_1 \,dx  +C \sqrt{\me(t)} \md(t).
\dal \deq
Substituting the estimates \eqref{p3m-est01}, \eqref{p3m-est02}, \eqref{p3m-est03}, \eqref{p3m-est04} and \eqref{p3m-est06} into the equality \eqref{p3m-equ}, we have
\begin{align*}
&\;\f{d}{dt} \f12 \i (  |\p_3^m \vro|^2 + (1+\vro) |\p_3^m u|^2 +  |\p_3^m b|^2 ) \, dx - \f{d}{dt}  \i \p_3^{m-1}  u_3  \, (N_1)_h \cdot \p_3^m u_h \, dx  \\
&\; + m \f{d}{dt} \i \vro \, \nabla_h \cdot u_h \, \p_3^{m} \vro \, \p_3^{m-1} u_3 \, dx+\f12 \f{d}{dt}  \i \vro \,|\p_3^m \vro|^2 \, dx  -\f{m+1}{2} \f{d}{dt}  \i \vro^2 \,|\p_3^m \vro|^2 \, dx \\
&\;-m \f{d}{dt} \i \p_3 u_h \cdot \nabla_h \p_3^{m-1} \vro \, \p_3^{m-1} u_3 \, dx+\|\p_3^m( \nabla_h  u, \p_1 b) \|_{L^2}^2  \\
\le &\;  -(m+1)
 \i \du\, |\p_3^m \vro|^2 \, dx  + m
 \i \nabla_h \cdot u_h |\p_3^m \vro|^2 \, dx   - (m+\f12) \i   \du \, | \p_3^m b_h |^2 \,dx \\
 &\;+ m \i \p_2 u_2 \, |\p_3^m b_1|^2  \, dx+  (m+1) \i \p_2 u_2 \,|\p_3^m b_2|^2 \, dx+ \i \p_3^m b_2 \, \p_2 u_1   \, \p_3^m b_1 \,dx \\
 &\;+ C \sqrt{\me(t)} \md(t),
\end{align*}
where $N_1$ is defined in \eqref{def-N1}.
Combining the estimates in Lemmas \ref{lemm:Nlinear1}-\ref{lemm:Nlinear3}, we finish the proof of this lemma.
        \end{proof}

 \subsection{Horizontal derivative estimates}\label{sec-tan-vertical}
        In this subsection, we will estimate the horizontal derivative of density, velocity and magnetic field.
        \begin{lemm}\label{lemma33}
			Under the assumption \eqref{assumption}, for any smooth solution $(\vro ,u, b)$ of equation \eqref{eqr},
			it holds 
			\beq\label{3101} \bal
			&\; \frac{d}{dt}\Big\{\| (\vro, \sqrt{1+\vro}\, u, b)\|_{L^2}^2 + \sum\limits_{|\alpha_h| = m}\| (\p^{\ah}\vro, \sqrt{1+\vro}\, \p^{\ah} u, \p^{\al_h} b)\|_{L^2}^2 \Big\} 
			\\
			&\;+  \|(\nabla_h u, \du,\p_1  b)\|_{H_{tan}^m}^2 
			\lesssim  \sqrt{\me(t)}\md(t).
			\dal \deq
		\end{lemm}
        \begin{proof}
            Similar to the equality \eqref{p3m-equ}, we have
            \beq\label{ph-equ}
			\begin{aligned}
				&\f{d}{dt} \f12 \i (  |\p^{\al_h} \vro|^2 + (1+\vro) |\p^{\al_h} u|^2 +  |\p^{\al_h} b|^2 ) \, dx  +\|\p^{\al_h}( \nabla_h  u, \p_1 b) \|_{L^2}^2 \\
				=& -\f{1}{2}\i  |\p^{\al_h} u|^2  (\du +u \cdot \nabla \vro + \vro \, \du)\, dx+\i \vro\, (\p_{1}^2 \p^{\ah}  u+\p_{2}^2 \p^{\ah}  u+\nabla \p^{\ah}  {\rm div}u)\cdot \p^{\ah} u \, dx\\
				&\;
				-\i \vro \, \p^{\ah} \nabla \vro \cdot \p^{\ah} u\, dx -\i \vro \, \p^{\ah} (\nabla b_2-\p_2 b) \cdot \p^{\ah} u\, dx - \i (1+\vro) \p^{\ah} ( u \cdot \nabla u+\vro \nabla \vro) \cdot \p^{\ah} u\, dx\\
                &\; - \i (1+\vro) \p^{\ah}\Big\{\frac{\vro}{1+\vro}
				(\p_{1}^2 u+\p_{2}^2 u+\nabla {\rm div}u -\nabla b_2+\p_2 b)\Big\} \cdot \p^{\ah} u\, dx  \\ 
                &\;  + \i (1+\vro) \p^{\ah}  \Big\{\frac{1}{1+\vro}b\cdot \nabla b
				-\frac{1}{2(1+\vro)}\nabla |b|^2\Big\} \cdot \p^{\ah} u\, dx   \\
                &\;   - \i  \p^{\ah} (u \cdot \nabla \vro + \vro \, \du) \, \p^{\ah} \vro \, dx + \i  \p^{\ah} (- u \cdot \nabla b + b \cdot \nabla u- b \, \du ) \cdot \p^{\ah} b \, dx\\
                := &\; \sum_{i=1}^{9} V_{i}.
			\end{aligned}
			\deq
           Using the anisotropic type inequalities \eqref{ie:Sobolev}, it is easy to check that
            \beqq \bal
            V_1 \lesssim &\; \| \p^{\al_h} u \|_{L^2} \| \p_1 \p^{\al_h} u \|_{L^2}^{\f12}  \| \p_2 \p^{\al_h} u \|_{L^2}^{\f12} ( \| \du\|_{L^2}^{\f12} \| \p_3 \du\|_{L^2}^{\f12}(1+\|\vro\|_{L^{\infty}}) + \|u\|_{L^{\infty}} \| \nabla \vro\|_{L^2}^{\f12}\| \p_3 \nabla \vro\|_{L^2}^{\f12}) \\
            \lesssim &\; \sqrt{\me(t)} \md(t).
            \dal \deqq
            Now we estimate the other terms.
            	\textbf{First of all, we deal with the case $\alpha_h = 0$. }  
                It is easy to check that
                \beqq \bal
            &\;  V_2 + V_4 + V_6 =0.
               \dal \deqq
               Integrating by parts and using the anisotropic type inequalities \eqref{ie:Sobolev}, we have
                \begin{align*}
                V_3 + V_8 = &\; \i  \vro  \, (\vro \, \du +  u \cdot \nabla \vro ) \, dx  -\i  \vro  \, (\vro \, \du +  u \cdot \nabla \vro ) \, dx  =0,\\
                V_5 =&\; - \f12 \i (1+\vro) u \cdot \nabla |u|^2 \, dx  - \f12 \i (1+\vro) \nabla (\vro^2) \cdot u \, dx\\
                =&\; \f12 \i (|u|^2 + \vro^2)((1+\vro) \du + u \cdot \nabla \vro) \, dx\\
                \lesssim &\; \| \du\|_{L^2}^{\f12} \| \p_3 \du\|_{L^2}^{\f12}  \|(\vro, u)\|_{L^2} \|\p_1(\vro, u)\|_{L^2}^{\f12}  \|\p_2 (\vro, u)\|_{L^2}^{\f12}  \\
                &\; + \| \nabla \vro\|_{L^2}^{\f12} \| \p_1  \nabla \vro\|_{L^2}^{\f12} \| u\|_{L^2}^{\f14}  \| \p_3 u\|_{L^2}^{\f14}  \| \p_2 u\|_{L^2}^{\f14}  \| \p_{23} u\|_{L^2}^{\f14} \\
                &\; \times\|(\vro, u)\|_{L^2}^{\f12} \| \p_1 (\vro, u)\|_{L^2}^{\f12} \|(\vro, u)\|_{L^2}^{\f14} \|\p_2 (\vro, u)\|_{L^2}^{\f14} \|\p_3 (\vro, u)\|_{L^2}^{\f14}\|\p_{23} (\vro, u)\|_{L^2}^{\f14}\\
                \lesssim &\; \sqrt{\me(t)} \md(t),
                \end{align*}
                and due to the condition ${\rm div} b=0$, we have
                \beqq \bal
                V_7 + V_9 = &\; \i (b \cdot \nabla b \cdot u+  b \cdot \nabla u \cdot b ) \, dx - \f12 \i u \cdot \nabla |b|^2 \, dx - \i |b|^2 \, \du \, dx \\
                = &\;- \i (b \cdot u) \, {\rm div} b \, dx - \f12 \i |b|^2 \, \du \, dx \\
                \lesssim &\; \| \du\|_{L^2}^{\f12} \| \p_3 \du\|_{L^2}^{\f12}  \|b\|_{L^2} \|\p_1 b\|_{L^2}^{\f12}  \|\p_2 b\|_{L^2}^{\f12} \\
                \lesssim  &\;  \sqrt{\me(t)} \md(t).
                \dal \deqq
                Combining the estimates from $V_1$ to $V_9$, we obtain that
                \beq \label{3102}
\f{d}{dt} \f12 \i (  |\vro|^2 + (1+\vro) |u|^2 +  |b|^2 ) \, dx  +\|( \nabla_h  u, \p_1 b) \|_{L^2}^2  \lesssim   \sqrt{\me(t)} \md(t).
                \deq
                \textbf{Next, we deal with the case $|\alpha_h |= m$. } 
                Obviously, it holds
                \beqq \bal
                V_2 + V_4 +V_6 
                =&\; \sum_{0 < \beta_h \le \ah} C_{\ah}^{\beta_h} \i (1+\vro) \p^{\bh} (\frac{\vro}{1+\vro})
				\p^{\ah-\bh}N_1 \cdot \p^{\ah} u \, dx,
                \dal \deqq
                which we recall that
                $N_1 = \p_{1}^2 u+\p_{2}^2 u+\nabla {\rm div}u -\nabla b_2+\p_2 b$.
                If $\bh=\ah$, under the assumption \eqref{assumption},  the anisotropic type inequality $\eqref{ie:Sobolev}_2$ yields that
                \beqq \bal
&\; \i (1+\vro) \p^{\ah} (\frac{\vro}{1+\vro})
				 N_1 \cdot \p^{\ah} u \, dx \\
                \lesssim &\; (1+\|\vro\|_{L^{\infty}}) \|\p^{\ah} (\frac{\vro}{1+\vro}) \|_{L^2} \| N_1\|_{L^2}^{\f14} \| \p_1 N_1\|_{L^2}^{\f14}\| \p_2 N_1\|_{L^2}^{\f14}\|\p_{12} N_1\|_{L^2}^{\f14} \| \p^{\ah} u \|_{L^2}^{\f12} \|\p_2 \p^{\ah} u \|_{L^2}^{\f12} \\
                \lesssim  &\;  \sqrt{\me(t)} \md(t).
                \dal \deqq
                If $|\bh|=1$, we have
                \beqq
              \i (1+\vro) \p^{\bh} (\frac{\vro}{1+\vro})
				\p^{\ah-\bh}N_1 \cdot \p^{\ah} u \, dx
                \lesssim \| \p^{\bh} \vro\|_{L^{\infty}} \| \p^{\ah-\bh}N_1\|_{L^2} \| \p^{\ah} u \|_{L^2}  \lesssim  \sqrt{\me(t)} \md(t),
                \deqq
                and if $ 1 < |\bh| < |\ah|$, we can obtain that
                \beqq \bal
             &\;  \i (1+\vro) \p^{\bh} (\frac{\vro}{1+\vro})
				\p^{\ah-\bh}N_1 \cdot \p^{\ah} u \, dx\\
                \lesssim &\; \| \p^{\bh} \vro\|_{L^{2}}^{\f12} \| \p_3 \p^{\bh} \vro\|_{L^{2}}^{\f12} \| \p^{\ah-\bh}N_1\|_{L^2}^{\f12}  \| \p_1 \p^{\ah-\bh}N_1\|_{L^2}^{\f12}  \| \p^{\ah} u \|_{L^2}^{\f12} \|\p_3 \p^{\ah} u \|_{L^2}^{\f12} \\
                \lesssim &\; \sqrt{\me(t)} \md(t).
               \dal \deqq
               Thus, the combination of above estimates yields directly
               \beqq
             V_2 +V_4+V_6 \lesssim \sqrt{\me(t)} \md(t).
               \deqq
               Similarly, integrating by parts, we can check that
             \begin{align*}
               V_3 + V_8 = &\; \i  \p^{\ah} \vro  \, (\vro \, \p^{\ah} \du + \p^{\ah} u \cdot \nabla \vro ) \, dx  \\
               &\;-\sum_{0 \le \beta_h \le \ah} C_{\ah}^{\beta_h} \i  (\p^{\bh} \vro \, \p^{\ah-\bh} \du +  \p^{\ah-\bh} u \cdot \nabla \p^{\bh}  \vro) \, \p^{\ah} \vro \, dx  \\
               = &\;  -\sum_{0 < \beta_h \le \ah} C_{\ah}^{\beta_h} \i  (\p^{\bh} \vro \, \p^{\ah-\bh} \du + \p^{\ah-\bh} u \cdot \nabla \p^{\bh}  \vro ) \, \p^{\ah} \vro \, dx \\
               \lesssim &\; \sqrt{\me(t)} \md(t),
               \end{align*}
               and
               \begin{align*}
               V_5 = &\; - \f12 \i (1+\vro) u \cdot \nabla |\p^{\ah} u|^2 \, dx -\sum_{0 < \beta_h \le \ah} C_{\ah}^{\beta_h} \i  (1+\vro) \p^{\bh} u \, \p^{\ah-\bh} \nabla u \cdot \p^{\ah} u  \, dx\\
               &\; - \f12 \i (1+\vro) \nabla \p^{\ah}(\vro^2) \cdot \p^{\ah} u \, dx\\
                =&\;\f12  \i |\p^{\ah} u|^2 ( (1+\vro) \p^{\ah} \du + \p^{\ah} u \cdot \nabla \vro) \, dx -\sum_{0 < \beta_h \le \ah} C_{\ah}^{\beta_h} \i  (1+\vro) \p^{\bh} u \, \p^{\ah-\bh} \nabla u \cdot \p^{\ah} u  \, dx\\
                &\; + \f12 \sum_{0 \le \beta_h \le \ah} C_{\ah}^{\beta_h} \i   \p^{\bh} \vro \,\p^{\ah-\bh} \vro ( (1+\vro) \p^{\ah} \du + \p^{\ah} u \cdot \nabla \vro) \, dx \\
                := &\; \sum_{i=1}^{3} V_{5,i}.
               \end{align*}
               We now estimate the term  $V_{5,2}$.
               If $\bh=\ah$, the anisotropic type inequality $\eqref{ie:Sobolev}_2$ yields that
                \beqq \bal
&\; \i (1+\vro) \p^{\ah} u
				\cdot \nabla u \cdot \p^{\ah} u \, dx \\
                \lesssim &\; (1+\|\vro\|_{L^{\infty}}) \|\p^{\ah} u\|_{L^2} \| \p_1 \p^{\ah} u\|_{L^2}^{\f12}  \| \p_2 \p^{\ah} u\|_{L^2}^{\f12}
                \|\nabla u\|_{L^2}^{\f12} \|\p_3 \nabla u \|_{L^2}^{\f12} \\
                \lesssim  &\;  \sqrt{\me(t)} \md(t).
                \dal \deqq
                If $|\bh|=1$, we have
                \beqq \bal
              \i (1+\vro) \p^{\bh} u \cdot
				\p^{\ah-\bh} \nabla u \cdot \p^{\ah} u \, dx
                \lesssim &\; \| \p^{\bh} u \|_{L^{\infty}} \| \p^{\ah-\bh} \nabla u\|_{L^2} \| \p^{\ah} u \|_{L^2} (1+\|\vro\|_{L^{\infty}})\\ 
                \lesssim &\; \sqrt{\me(t)} \md(t),
              \dal \deqq
                and if $ 1 < |\bh| < |\ah|$, we can obtain that
                \beqq \bal
             &\;  \i (1+\vro) \p^{\bh} u \cdot
				\p^{\ah-\bh} \nabla u \cdot \p^{\ah} u \, dx \\
                \lesssim &\;  \| \p^{\bh} u \|_{L^{2}}^{\f12} \| \p_1 \p^{\bh} u \|_{L^{2}}^{\f12} \| \p^{\ah-\bh} \nabla u\|_{L^2}^{\f12} \| \p_3  \p^{\ah-\bh} \nabla u\|_{L^2}^{\f12} \| \p^{\ah} u \|_{L^2}^{\f12} \| \p_2 \p^{\ah} u \|_{L^2}^{\f12} (1+\|\vro\|_{L^{\infty}}) \\
                \lesssim &\; \sqrt{\me(t)} \md(t).
                \dal \deqq
                Thus, the combination of above estimates yields directly
               \beqq
             V_{5,2} \lesssim \sqrt{\me(t)} \md(t).
               \deqq
               Similarly, we have
        \beqq
             V_{5,1} +V_{5,3} \lesssim \sqrt{\me(t)} \md(t),
               \deqq
               which, together with the above estimate, yields that
               \beqq 
               V_5 \lesssim \sqrt{\me(t)} \md(t).
               \deqq
               It remains to estimate the term $V_7$ and  $V_9$. Integrating by parts gives directly that
               \beqq \bal
&\; \i b \cdot \p^{\ah} \nabla b \cdot \p^{\ah} u  \, dx + \i b \cdot \p^{\ah} \nabla u \cdot \p^{\ah} b  \, dx  =0,\\
&\; \i u \cdot \p^{\ah} \nabla b \cdot \p^{\ah} b  \, dx  = - \f12 \i \du \, |\p^{\ah}  b|^2 \, dx.
               \dal \deqq
              Then, using the anisotropic type inequality \eqref{ie:Sobolev}, it holds 
               \begin{align*}
&\; V_7 + V_9 \\
= &\; \sum_{0 < \beta_h \le \ah} C_{\ah}^{\beta_h} \i (1+\vro) \p^{\bh} (\f{1}{1+\vro}) \, \p^{\ah-\bh} (b \cdot \nabla b) \cdot \p^{\ah} u  \, dx  +  \sum_{0 < \beta_h \le \ah} C_{\ah}^{\beta_h} \i \p^{\bh} b \cdot \p^{\ah-\bh} \nabla b \cdot \p^{\ah} u  \, dx \\
&\; - \f12 \sum_{0 < \beta_h \le \ah} C_{\ah}^{\beta_h} \i  (1+\vro) \p^{\bh} (\f{1}{1+\vro}) \, \p^{\ah-\bh} \nabla (|b|^2) \cdot \p^{\ah} u  \, dx + \f12 \i  \p^{\ah}( |b|^2 )\, \p^{\ah} \du \, dx\\
&\; - \sum_{0 < \beta_h \le \ah} C_{\ah}^{\beta_h} \i  \p^{\bh} u \cdot \p^{\ah-\bh} \nabla b \cdot \p^{\ah} b  \, dx 
+ \sum_{0 < \beta_h \le \ah} C_{\ah}^{\beta_h} \i  \p^{\bh} b \cdot \p^{\ah-\bh} \nabla u  \cdot \p^{\ah} b  \, dx \\
&\; - \sum_{0 \le \beta_h < \ah} C_{\ah}^{\beta_h} \i    \p^{\bh} b \, \p^{\ah-\bh} \du     \cdot \p^{\ah} b  \, dx - \f12 \i \du \, |\p^{\ah}  b|^2 \, dx\\
\lesssim &\; \sqrt{\me(t)} \md(t).
               \end{align*}
               Combining the estimates from $V_1$ to $V_9$, we obtain that
                \beqq \bal
\f{d}{dt} \f12 \i (  |\p^{\ah} \vro|^2 + (1+\vro) |\p^{\ah} u|^2 +  |\p^{\ah} b|^2 ) \, dx  +\|\p^{\ah} ( \nabla_h  u, \p_1 b) \|_{L^2}^2 
\lesssim  \sqrt{\me(t)} \md(t).
                \dal \deqq
                Combining the above estimates and estimate \eqref{3102}, we finish the proof of this lemma.
        \end{proof}
        
        \subsection{Enhanced dissipation estimates}
		In this subsection, we will establish the dissipative estimates of  the density in the horizontal directions and magnetic field in the $x_2$ direction. 
        First of all, we establish the dissipative estimates of  the magnetic field in the $x_2$ direction.
        \begin{lemm}\label{lemma34}
			Under the assumption \eqref{assumption}, for any smooth solution $(\vro ,u, b)$ of equation \eqref{eqr}, it holds
            \beq\label{p2b01}
			\begin{aligned}
				&\; \f{d}{dt} \i \Big( \p_{2} u \cdot b +  \sum_{|\al|=m-1}   \p_{2} \p^{\al} u \cdot  \p^{\al} b  \Big)\, dx  + \| \p_{2} b\|_{H^{m-1}}^2\\
                \lesssim &\; \| (\p_1 u, \p_2 u, \p_1 b)\|_{H^{m}}^2 +\| \du\|_{H^{m-1}}^2 
				+\sqrt{\me(t)}\md(t).
			\end{aligned}
			\deq
\end{lemm}
\begin{proof}
    For any $|\al|\le m-1$, the equation $\eqref{eqr}_2$  and the condition ${\rm div} b =0$ yield that
			\beq\label{p2b02} \bal
			\|\p^{\al} \p_2 b\|_{L^2}^2  = &\;  \i \p_t \p^{\al} u \cdot \p^{\al} \p_2 b\, dx + \i \p^{\al} (\p_{1}^2 u+\p_{2}^2 u) \cdot \p^{\al} \p_2 b \, dx\\
            &\; + \i \p^{\al} (\nabla {\rm div}u + \nabla \vro-\nabla b_2) \cdot \p^{\al} \p_2 b \, dx
			 - \i \p^{\al} F_2 \cdot \p^{\al} \p_2 b \, dx\\
            = &\;  \i \p_t \p^{\al} u \cdot \p^{\al} \p_2 b\, dx + \i \p^{\al} (\p_{1}^2 u+\p_{2}^2 u) \cdot \p^{\al} \p_2 b \, dx
			 - \i \p^{\al} F_2 \cdot \p^{\al} \p_2 b \, dx.
			\dal \deq
            Using the equation $\eqref{eqr}_3$ and integrating by parts, we conclude
            \begin{align*}
				 \i \p_t \p^{\al} u \cdot \p^{\al} \p_2 b\, dx 
				= &\;  \p_t \i \p^{\al} u \cdot \p^{\al} \p_2 b\, dx  -\i  \p^{\al} u \cdot \p^{\al} \p_t \p_2 b \, dx \\
				= &\; -   \p_t \i \p^{\al} \p_2 u \cdot \p^{\al} b \, dx   +\i  \p^{\al} \p_2  u \cdot \p^{\al}  (\p_1^2 b + \p_2 u - e_2 \du+F_3 ) \, dx.
			\end{align*}
            Substituting above equality into \eqref{p2b02}, we have
            \beqq \bal
&\;  \f{d}{dt} \i \p^{\al} \p_2 u \cdot \p^{\al} b \, dx  + \| \p^{\al} \p_2 b\|_{L^2}^2\\
= &\; 
 \i  \p^{\al} \p_2  u \cdot \p^{\al}  (\p_1^2 b + \p_2 u - e_2 \du) \, dx + \i \p^{\al} (\p_{1}^2 u+\p_{2}^2 u) \cdot \p^{\al} \p_2 b \, dx \\
			&\;+ \i  \p^{\al} \p_2  u \cdot \p^{\al} F_3  \, dx    - \i \p^{\al} F_2 \cdot \p^{\al} \p_2 b \, dx\\
            :=&\; \sum_{i=1}^{4} VI_{i}.
            \dal \deqq
            Using the H\"older inequality, we have
            \beqq \bal
            VI_1 \lesssim &\; \| \p^{\al} \p_2 u\|_{L^2} (\| \p^{\al} \p_1^2 b\|_{L^2}  + \| \p^{\al} \p_2 u\|_{L^2}+\| \p^{\al} \du\|_{L^2}),\\
            VI_2 \le &\;  \nu \| \p^{\al} \p_2 b\|_{L^2}^2 + C_{\nu}(\| \p^{\al} \p_1^2 u\|_{L^2}^2 + \| \p^{\al} \p_2^2 u\|_{L^2}^2 ).
            \dal \deqq
            Now we estimate the terms $VI_3$ and $VI_4$. {\textbf{First, we estimate the case $\al=0$}.}
            Using the anisotropic type inequality \eqref{ie:Sobolev}, we have
            \beqq \bal
            VI_3 = &\;  \i  \p_2 u \cdot (- u \cdot  \nabla b +  b \cdot  \nabla u - b \,  \du) \, dx\\
            \lesssim &\;  \| \p_2 u\|_{L^2}^{\f12}  \| \p_3 \p_2 u\|_{L^2}^{\f12} \Big( \| u\|_{L^2}^{\f12} \| \p_2  u\|_{L^2}^{\f12} \| \nabla b\|_{L^2}^{\f12} \| \p_1  \nabla b\|_{L^2}^{\f12} + \| b\|_{L^2}^{\f12} \| \p_1 b\|_{L^2}^{\f12} \|\nabla u\|_{L^2}^{\f12} \| \p_2 \nabla u\|_{L^2}^{\f12} \\&\; 
            + \| b\|_{L^2}^{\f12} \| \p_1 b\|_{L^2}^{\f12} \| \du\|_{L^2}^{\f12} \| \p_2 \du\|_{L^2}^{\f12} \Big) \\
            \lesssim &\; \sqrt{\me(t)} \md(t),
            \dal \deqq
            and using the condition ${\rm div} u =0$, we have
            \beqq \bal
            VI_4 = &\; - \i u \cdot \nabla u \cdot \p_2 b \, dx - \f12 \i  \nabla |\vro|^2 \cdot \p_2 b \, dx   + \i \frac{1}{1+\vro} (b\cdot \nabla b -\nabla |b|^2)  \cdot \p_2 b \, dx \\
            &\; - \i \frac{\vro}{1+\vro}
				(\p_{1}^2 u+\p_{2}^2 u+\nabla {\rm div}u -\nabla b_2+\p_2 b) \cdot \p_2 b \, dx \\
               = &\; - \i u \cdot \nabla u \cdot \p_2 b \, dx  - \i \frac{\vro}{1+\vro}
				(\p_{1}^2 u+\p_{2}^2 u+\nabla {\rm div}u -\nabla b_2+\p_2 b) \cdot \p_2 b \, dx \\
               &\;  + \i \frac{1}{1+\vro} (b\cdot \nabla b -\nabla |b|^2)  \cdot \p_2 b \, dx \\
               \lesssim &\; \| \p_2 b\|_{L^2} \| (\vro, u, b)\|_{L^2}^{\f14}\| \p_2(\vro, u, b)\|_{L^2}^{\f14}\|\p_3 (\vro, u, b)\|_{L^2}^{\f14} \|\p_{23} (\vro, u, b)\|_{L^2}^{\f14} \|\nabla(\vro, u, b)\|_{L^2}^{\f12} \|\p_1 \nabla(\vro, u, b)\|_{L^2}^{\f12}\\
               &\; + \| \p_2 b\|_{L^2} \| \vro\|_{L^{\infty}} (\| \Delta_h u\|_{L^2}  + \| \nabla \du\|_{L^2}) \\
               \lesssim &\; \sqrt{\me(t)} \md(t).
            \dal \deqq
            Combining the estimates from $VI_1$ to $VI_4$ and using the smallness of $\nu$, we have
            \beq\label{p2b03} \bal
\f{d}{dt} \i  \p_2 u \cdot  b \, dx  + \| \p_2 b\|_{L^2}^2 \lesssim  \sqrt{\me(t)} \md(t) + \| (\p_1^2 u, \p_2^2 u, \p_1^2 b)\|_{L^2}^2  + \|\p_2 u\|_{L^2}^2 + \| \du\|_{L^2}^2 .
            \dal \deq
            {\textbf{Finally, we estimate the case $|\al|=m-1$}.}
            Using the anisotropic type inequality \eqref{ie:Sobolev}, we can check that
            \beqq \bal
            VI_3 = &\;  \sum_{0 \le \beta \le \al} C_{\al}^{\beta} \i \p^{\al} \p_2 u \cdot (-\p^{\beta} u \cdot  \p^{\al- \beta} \nabla b + \p^{\beta} b \cdot  \p^{\al- \beta} \nabla u -\p^{\beta} b \, \p^{\al- \beta} \du) \, dx\\
            \lesssim &\; \sum_{0 \le \beta \le \al}  \| \p^{\al} \p_2 u\|_{L^2}^{\f12}  \| \p_3 \p^{\al} \p_2 u\|_{L^2}^{\f12} \Big( \| \p^{\beta} u\|_{L^2}^{\f12} \| \p_2 \p^{\beta} u\|_{L^2}^{\f12} \| \p^{\al-\beta} \nabla b\|_{L^2}^{\f12} \| \p_1 \p^{\al-\beta} \nabla b\|_{L^2}^{\f12} \\&\; 
            + \| \p^{\beta} b\|_{L^2}^{\f12} \| \p_1 \p^{\beta} b\|_{L^2}^{\f12} (\| \p^{\al-\beta} \nabla u\|_{L^2}^{\f12} \| \p_2 \p^{\al-\beta} \nabla u\|_{L^2}^{\f12} + \| \p^{\al-\beta}\du\|_{L^2}^{\f12} \| \p_2 \p^{\al-\beta} \du\|_{L^2}^{\f12}) \Big) \\
            \lesssim &\; \sqrt{\me(t)} \md(t).
            \dal \deqq
            Similarly, we can check that
            \beqq
VI_4 \lesssim \sqrt{\me(t)} \md(t).
            \deqq
            Thus, combining the estimates from $VI_1$ to $VI_4$ and using the smallness of $\nu$, we have
            \beqq\bal
\f{d}{dt} \i  \p^{\al} \p_2 u \cdot  \p^{\al}  b \, dx  + \|  \p^{\al}  \p_2 b\|_{L^2}^2 \lesssim  \sqrt{\me(t)} \md(t) + \|  \p^{\al} (\p_1^2 u, \p_2^2 u, \p_1^2 b)\|_{L^2}^2  + \| \p^{\al}  \p_2 u\|_{L^2}^2 + \|  \p^{\al}  \du\|_{L^2}^2,
            \dal \deqq
            together with the estimate \eqref{p2b03}, we finish the proof of this lemma.
\end{proof}
Next, we establish the dissipative estimates of  the density in the horizontal directions.
\begin{lemm}\label{lemma35}
			Under the assumption \eqref{assumption}, for any smooth solution $(\vro ,u, b)$ of equation \eqref{eqr}, it holds, 
			\beq \label{phvro01}
			\begin{aligned}
				&\; \f{d}{dt} \i \Big(u_h \cdot \nabla_h  \vro +  \sum_{|\al|=m-1} \p^{\al}  u_h \cdot \nabla_h \p^{\al} \vro  \Big)\ dx  + \|\nabla_h  \vro\|_{H^{m-1}}^2  \\
				\lesssim &\;\| (\p_1^2 u, \p_2^2 u, \p_1 b)\|_{H^{m-1}}^2  + \| \du \|_{H^{m-1}}^2 + \|\p_2 b\|_{H^{m-1}}^2  + \sqrt{\me(t)}\md(t).
			\end{aligned}
			\deq
		\end{lemm}
            \begin{proof}
                For any $|\al|\le m-1$, the equation $\eqref{eqr}_2$ yields that
			\beqq \bal
			\|\p^{\al} \nabla_h \vro\|_{L^2}^2  = &\; -\i \p_t \p^{\al} u_h \cdot \nabla_h \p^{\al} \vro \ dx + \i \p^{\al} (F_2)_h \cdot \nabla_h \p^{\al} \vro \ dx \\
            &\; +   \i \p^{\al}(\Delta_h u_h + \nabla_h \du + \p_2 b_h - \nabla_h b_2) \cdot \nabla_h \p^{\al} \vro \ dx 
			\dal \deqq
			Using the density equation $\eqref{eqr}_1$ and integrating by parts, we conclude
			\beqq \bal
			&\; -\i \p_t \p^{\al} u_h \cdot \nabla_h \p^{\al} \vro \ dx \\ 
			= &\; - \p_t \i \p^{\al} u_h \cdot \nabla_h \p^{\al} \vro \ dx  +\i  \p^{\al} u_h \cdot \nabla_h \p^{\al} \p_t \vro \ dx \\
			= &\;  -  \p_t \i \p^{\al} u_h \cdot \nabla_h \p^{\al} \vro \ dx   -\i  \p^{\al} u_h \cdot \nabla_h \p^{\al} (\du + u \cdot \nabla \vro + \vro \, \du)\ dx \\
			= &\;  -  \p_t \i \p^{\al} u_h \cdot \nabla_h \p^{\al} \vro \ dx   +\i  \p^{\al} \nabla_h \cdot  u_h \cdot \p^{\al} (\du + u \cdot \nabla \vro + \vro \, \du)\ dx. 
			\dal \deqq
            Thus, we have
			\begin{align*}
				&\; \f{d}{dt} \i \p^{\al}  u_h \cdot \nabla_h \p^{\al} \vro \ dx  + \|\p^{\al} \nabla_h  \vro\|_{L^2}^2\\
				= &\;   \i  \p^{\al}  \nabla_h \cdot   u_h \cdot \p^{\al} (\du + u \cdot \nabla \vro + \vro \, \du)\ dx + \i \p^{\al} (F_2)_h \cdot \nabla_h \p^{\al} \vro \ dx \\
				&\; +   \i \p^{\al}(\Delta_h u_h + \nabla_h \du + \p_2 b_h - \nabla_h b_2) \cdot \nabla_h \p^{\al} \vro \ dx\\
                := &\; \sum_{i=1}^{3} VII_{i}.
                \end{align*}
                It is easy to check that
                \beqq
               VII_3 \le \nu \| \nabla_h \p^{\al} \vro\|_{L^2}^2 + C_{\nu} (\|\p^{\al} (\p_1^2 u, \p_2^2 u, \p_1 b)\|_{L^2}^2+ \|\p^{\al} \du \|_{L^2}^2 + \|\p^{\al} \p_2 b\|_{L^2}^2 ).
                \deqq
                {\textbf{First, we estimate the case $\al=0$}.} Using the anisotropic type inequality \eqref{ie:Sobolev}, we have
                \beqq \bal
                VII_1 \lesssim &\; \| \nabla_h u\|_{L^2} \| \du\|_{L^2}(1+ \|\vro\|_{L^{\infty}} )+  \| \nabla_h u\|_{L^2} \| \nabla \vro\|_{L^2}^{\f12} \| \p_1 \nabla \vro\|_{L^2}^{\f12} \| u\|_{L^2}^{\f14} \|\p_2 u\|_{L^2}^{\f14} \| \p_3 u\|_{L^2}^{\f14}\| \p_{23} u\|_{L^2}^{\f14}  \\
                \lesssim  &\; \sqrt{\me(t)}\md(t),
                \dal \deqq
                and similar to the estimate of $VI_4$, we can obtain that
                \begin{align*}
                VII_2 = &\;  - \i u \cdot \nabla_h u \cdot \nabla_h \vro \, dx - \f12 \i  \nabla_h |\vro|^2 \cdot \nabla_h \vro  \, dx  \\
                &\; - \i \frac{\vro}{1+\vro}
				(\p_{1}^2 u_h+\p_{2}^2 u_h+\nabla_h {\rm div}u -\nabla_h b_2+\p_2 b_h) \cdot \nabla_h \vro  \, dx \\
               &\;  + \i \frac{1}{1+\vro} (b\cdot \nabla_h b -\nabla_h |b|^2)  \cdot \nabla_h \vro \, dx \\
               \lesssim  &\; \sqrt{\me(t)}\md(t),
                \end{align*}
                which, together with the estimates of $VII_1$ and $VII_3$, using the smallness of $\nu$, yields that
                \beq \label{phvro02}
 \f{d}{dt} \i  u_h \cdot \nabla_h  \vro \ dx  + \|\nabla_h  \vro\|_{L^2}^2 \lesssim  \| (\p_1^2 u, \p_2^2 u, \p_1 b)\|_{L^2}^2 + \| \du \|_{L^2}^2 + \|\p_2 b\|_{L^2}^2  + \sqrt{\me(t)}\md(t).
                \deq
                {\textbf{Finally, we deal with the case $|\al|=m-1$}.} Similar to the estimate \eqref{p2b02}, we can obtain that
                 \beqq \bal
&\;  \f{d}{dt} \i \p^{\al} u_h \cdot \nabla_h \p^{\al}  \vro \ dx  + \|\nabla_h \p^{\al}  \vro\|_{L^2}^2 \\ \lesssim  &\; \| \p^{\al} (\p_1^2 u, \p_2^2 u, \p_1 b)\|_{L^2}^2+ \| \p^{\al} \du \|_{L^2}^2 + \|\p^{\al} \p_2 b\|_{L^2}^2  + \sqrt{\me(t)}\md(t),
                \dal \deqq
                here we omit the proof for simplicity. 
                Combining the above estimate and \eqref{phvro02}, we finish the proof of this lemma.
            \end{proof}
            
    \subsection{Proof of Proposition \ref{main_pro}}

    In this subsection, we will prove Proposition \ref{main_pro} and Theorem \ref{main_result}.
     First, we will give the proof of Proposition \ref{main_pro}.
		Indeed,  the combination of estimates  \eqref{p3m01} and \eqref{3101}
		yields directly
		\beqq
		\begin{aligned}
	&\;  \frac{d}{dt}\Big\{ \| (\vro, \sqrt{1+\vro}\, u, b)\|_{L^2}^2 + \sum\limits_{|\alpha_h| = m}\| (\p^{\ah}\vro, \sqrt{1+\vro}\, \p^{\ah} u, \p^{\al_h} b)\|_{L^2}^2 + \| (\p_3^m \vro, \sqrt{1+\vro}\, \p_3^m  u, \p_3^m  b)\|_{L^2}^2 \Big\} \\
    &\; + \frac{d}{dt} F(\vro, u, b) +  \|(\nabla_h u, \du,\p_1  b)\|_{H^m}^2\lesssim \sqrt{\me(t)}\md(t),   
		\end{aligned}
		\deqq
        where $F(\vro, u,b)$ is defined in \eqref{def-F}. 
       From the estimates \eqref{p2b01} and \eqref{phvro01}, for some small positive constant $\kappa_1$, we have
       \beqq \bal
       &\; \f{d}{dt} \i \Big( \p_{2} u \cdot b +  \sum_{|\al|=m-1}   \p_{2} \p^{\al} u \cdot  \p^{\al} b + \kappa_1 u_h \cdot \nabla_h  \vro +   \kappa_1 \sum_{|\al|=m-1} \p^{\al}  u_h \cdot \nabla_h \p^{\al} \vro  \Big)\, dx  \\
       &\; + \kappa_1 (\| \p_{2} b\|_{H^{m-1}}^2+  \|\nabla_h  \vro\|_{H^{m-1}}^2 )
				\lesssim \| (\p_1 u, \p_2 u, \p_1 b)\|_{H^{m}}^2 + \| \du\|_{H^{m-1}}^2 +\sqrt{\me(t)}\md(t).
       \dal \deqq
		Then, for some small positive constant $\kappa_2$, we can check
		\beqq
		\begin{aligned}
			\f{d}{dt} \widetilde{\me}(t) + 2 \ka_1 \ka_2 \md(t) \lesssim \sqrt{\me(t)}\md(t),
		\end{aligned}
		\deqq
        where the quantity $\widetilde{\me}(t)$ is defined as 
        \beqq
		\begin{aligned}
			\widetilde{\me}(t):= &\;  \| (\vro, \sqrt{1+\vro}\, u, b)\|_{L^2}^2 + \sum\limits_{|\alpha_h| = m}\| (\p^{\ah}\vro, \sqrt{1+\vro}\, \p^{\ah} u, \p^{\al_h} b)\|_{L^2}^2  + \| (\p_3^m \vro, \sqrt{1+\vro}\, \p_3^m  u, \p_3^m  b)\|_{L^2}^2 \\
            &\;+  F(\vro, u, b) + 2 \kappa_2  \i \Big( \p_{2} u \cdot b +  \sum_{|\al|=m-1}   \p_{2} \p^{\al} u \cdot  \p^{\al} b \Big) \, dx \\
            &\;+ 2 \kappa_1 \kappa_2 \i \Big(u_h \cdot \nabla_h  \vro  +  \sum_{|\al|=m-1} \p^{\al}  u_h \cdot \nabla_h \p^{\al} \vro \Big)  \, dx.
		\end{aligned}
		\deqq
        Under the assumption \eqref{assumption}, due to the smallness of $\kappa_1$ and $\kappa_2$, it is easy to check that
        $\widetilde{\me}(t)$ is equivalent to  $\me(t)$. Thus, due to the a prior assumption \eqref{assumption}, we have
        \beqq
        	\f{d}{dt} \widetilde{\me}(t) + \ka_1 \ka_2 \md(t) \le 0,
        \deqq
        which yields that for any $t \in (0,T]$
        \beqq
         \underset{0\le \tau \le t}{\sup} \me(\tau) + \int_{0}^{t} \md(\tau)  d\tau \le C \Big(
        \underset{0\le \tau \le t}{\sup} \widetilde{\me}(\tau) + \int_{0}^{t} \md(\tau)  d\tau \Big)\le C \widetilde{\me}(0) \le C \me(0),
        \deqq
        where the constant C is independent of time.
        Choose the small constant $\delta:= 2 C\me(0) \le 1$, then we have
        \beqq
        \underset{0\le \tau \le t}{\sup} \me(\tau) + \int_{0}^{t} \md(\tau)  d\tau \le C \me(0) = \f{\delta}{2},
        \deqq
        which implies the estimate \eqref{close_assumption}. Then we finish the proof of Proposition \ref{main_pro}.

   \section{Proof of Theorem \ref{main_result}}\label{sec-theorem}     
			Suppose the assumptions in Theorem \ref{main_result} hold,
			similar to the result in \cite{Matsumura-Nishida1980}, 
			one can establish the local-in-time well-posedness for the equation \eqref{eqr}.
			Next, we use the standard continuity argument to show the global well-posedness.
			From the local existence result
			and smallness assumption of initial condition,
			it holds
			\beqq
			\underset{0\le \tau \le t}{\sup}\mathcal{E}(\tau)
			+\int_0^t \md(\tau)d\tau \le \delta,
			\deqq
			for all $t\in [0, T_0)$ and $\delta=2C\me(0)$,
			where $C$ is a positive constant independent of time $t$.
			Set
			\beq\label{criterion}
            \begin{aligned}
			T_1:=\underset{T_0}{\sup}
			&\left\{T_0~|
			\underset{0\le \tau \le t}{\sup}\mathcal{E}(\tau)
			+\!\int_0^t \! \md(\tau)d\tau \le \delta,
			\quad \forall ~ t\in [0, T_0)\right\},
            \end{aligned}
			\deq
			we claim that $T_1=+\infty$. Otherwise, applying the estimate
			\eqref{close_assumption}
			and the local-in-time existence result,
			there exists a positive constant $T_2$ such that $T_2>T_1$,
			it holds that for any $T\in [T_1, T_2)$,
			\beqq
			\underset{0\le \tau \le t}{\sup}\mathcal{E}(\tau)
			+\int_0^t \md(\tau)d\tau \le \delta,
			\quad \forall ~ t\in [0, T).
			\deqq
			This contradicts the definition of $T_1$ in \eqref{criterion}.
			Therefore, we can deduce that $T_1=+\infty$.
			Therefore, we complete the proof of Theorem \ref{main_result}.			

\section*{Acknowledgments}
		Jincheng Gao was partially supported by the National Key Research and Development Program of China(2021YFA1002100), Guangdong Special Support Project (2023TQ07A961) and Guangzhou Science and Technology Program (2024A04J6410).
        Xianpeng Hu's research was partially supported by the RFS Grant and GRF Grants from the Research Grants Council (Nos. PolyU 11302021, 11310822 and 11302523).
       Lianyun Peng was partially supported by a fellowship award (PolyU RFS2122-1S05).
		Jiahong Wu was partially supported by the
		National Science Foundation of the United States (DMS 2104682, DMS 2309748).

		\section*{Data Availability}
		Data sharing is not applicable to this article as no new data were created or analysed in this study.

		\section*{Conflict of interest}
		The authors declared that they have no Conflict of interest to this work.
        
            	\phantomsection
		
\end{sloppypar}
\end{document}